\RequirePackage{fix-cm}
\documentclass[11pt,reqno]{amsart}
\usepackage{amsmath}
\usepackage{setspace}
\usepackage{fancyhdr}
\usepackage{mathrsfs}
\usepackage{xifthen}
\usepackage{verbatim}
\usepackage{hyperref}
\usepackage[all]{xcolor}

\usepackage{geometry}
\usepackage{fancyhdr}
\usepackage{textcomp}
\usepackage{amsthm}
\usepackage{amstext}
\usepackage{amsfonts}
\usepackage{amssymb}
\usepackage{fixltx2e}
\usepackage{graphicx}

\usepackage{tikz}
\usepackage{tikz-cd}
\usepackage{caption}
\usetikzlibrary{calc,backgrounds}
\usetikzlibrary{matrix,arrows,decorations.pathmorphing}

\usepackage{bigstrut}
\usetikzlibrary{external}
\makeatletter

\newcommand{\adicff}{\mathcal{X}}
\newcommand{\Spa}{\mathrm{Spa}}
\newcommand{\Spd}{\mathrm{Spd}}
\newcommand{\HN}{\mathrm{HN}}
\newcommand{\HNvec}{\overrightarrow{\mathrm{HN}}}
\newcommand{\Surj}{\mathcal{S}\mathrm{urj}}
\newcommand{\Inj}{\mathcal{I}\mathrm{nj}}
\newcommand{\Hom}{\mathcal{H}\mathrm{om}}

\numberwithin{equation}{section}

\usepackage{color}

\newcommand{\heart}{\ensuremath\heartsuit}

\newcommand{\RR}{\mathbf{R}}

\newcommand{\Q}{\mathbf{Q}}
\newcommand{\Z}{\mathbf{Z}}

\newcommand{\ol}[1]{\overline{#1}}

\newcommand{\Cal}[1]{\mathcal{#1}}

\newcommand{\mrm}[1]{\mathrm{#1}}

\DeclareMathOperator{\Ima}{Im\,}
\DeclareMathOperator{\rank}{rank}

\DeclareMathOperator{\Bun}{Bun}

\DeclareMathOperator{\Id}{Id}

\DeclareMathOperator{\Proj}{Proj\,}

\DeclareFontFamily{OT1}{rsfs}{}
\DeclareFontShape{OT1}{rsfs}{n}{it}{<-> rsfs10}{}
\DeclareMathAlphabet{\mathscr}{OT1}{rsfs}{n}{it}

\newcommand{\acal}{\mathcal{A}}

\newcommand{\ecal}{\mathcal{E}}
\newcommand{\fcal}{\mathcal{F}}

\newcommand{\hcal}{\mathcal{H}}

\newcommand{\kcal}{\mathcal{K}}
\newcommand{\lcal}{\mathcal{L}}

\newcommand{\ocal}{\mathcal{O}}

\newcommand{\qcal}{\mathcal{Q}}

\newcommand{\vcal}{\mathcal{V}}

\newcommand{\xcal}{\mathcal{X}}

\newcommand{\ten}{\otimes}
\newcommand{\dsm}{\oplus}

\newcommand{\surj}{\twoheadrightarrow}

\newtheorem{thm}{Theorem}[subsection]
\newtheorem{lemma}[thm]{Lemma}
\newtheorem{prop}[thm]{Proposition}
\newtheorem{cor}[thm]{Corollary}

\theoremstyle{remark}
\newtheorem{remark}[thm]{Remark}
\newtheorem{defn}[thm]{Definition}
\newtheorem{example}[thm]{Example}

\newtheorem{claim}[thm]{Claim}
\newtheorem*{thm*}{Theorem}
\makeatletter
\def\th@remark{%
  \thm@headfont{\bfseries}%
  \normalfont % body font
}
\def\imod#1{\allowbreak\mkern5mu({\operator@font mod}\,\,#1)}
\makeatother

\usepackage{enumitem}		% customizable list environments
\theoremstyle{theorem}

\numberwithin{equation}{section}

\makeatother

\begin{document}
	
	\tikzset{
		node style sp/.style={draw,circle,minimum size=\myunit},
		node style ge/.style={circle,minimum size=\myunit},
		arrow style mul/.style={draw,sloped,midway,fill=white},
		arrow style plus/.style={midway,sloped,fill=white},
	}
    
	\title{Extensions of Vector Bundles on the Fargues-Fontaine Curve}
	
	\author[C. Birkbeck]{Christopher Birkbeck}
    \address{Mathematics Institute, University of Warwick, Zeeman Building, Coventry CV4 7AL}
    \email{c.d.birkbeck@warwick.ac.uk}
	
    \author[T. Feng]{Tony Feng}
    \address{Department of Mathematics, Stanford University, 450 Serra Mall, Stanford CA 94305}
    \email{tonyfeng@stanford.edu}
    
    \author[D. Hansen]{David Hansen}
    \address{Department of Mathematics, Columbia University, 2990 Broadway, New York NY 10027}
   \email{hansen@math.columbia.edu}
   
    \author[S. Hong]{Serin Hong}
    \address{Department of Mathematics, California Institute of Technology, 1200 E. California Blvd, Pasadena CA 91125}
    \email{shong2@caltech.edu}
    \author[Q. Li]{Qirui Li} 
    \address{Department of Mathematics, Columbia University, 2990 Broadway, New York NY 10027}
   \email{qiruili@math.columbia.edu}
   
    \author[A. Wang]{Anthony Wang}
    \address{Department of Mathematics, University of Chicago, 5734 S. University Avenue, Chicago, IL, 60637}
    \email{anthonyw@math.uchicago.edu}
    
    \author[L. Ye]{Lynnelle Ye}
    \address{Department of Mathematics, Harvard University, 1 Oxford Street, Cambridge, MA 02138}
    \email{lynnelle@math.harvard.edu}
    
    \begin{abstract}We completely classify the possible extensions between semistable vector bundles on the Fargues-Fontaine curve (over an algebraically closed perfectoid field), in terms of a simple condition on Harder-Narasimhan polygons. Our arguments rely on a careful study of various moduli spaces of bundle maps, which we define and analyze using Scholze's language of diamonds. This analysis reduces our main results to a somewhat involved combinatorial problem, which we then solve via a reinterpretation in terms of the Euclidean geometry of Harder-Narasimhan polygons.
    \end{abstract}
	
	\maketitle

	\tableofcontents
	
	\rhead{}

	\chead{}
	%---------------------------------------------------------------%
\section{Introduction}
    
Let $E$ be a $p$-adic local field with residue field $\mathbf{F}_q$, and let $F/\mathbf{F}_q$ be an algebraically closed complete nonarchimedean field of characteristic $p$. Given any such pair, Fargues and Fontaine \cite{FF08} defined a remarkable scheme $X=X_{E,F}$, the so-called \emph{Fargues-Fontaine curve}. Many constructions in classical $p$-adic Hodge theory can be reinterpreted ``geometrically'' in terms of vector bundles on $X$.

One of the main results of \cite{FF08} is a complete classification of vector bundles on $X$. In particular, Fargues and Fontaine proved that any bundle is determined up to isomorphism by its Harder-Narasimhan (HN) polygon, that the Harder-Narasimhan filtration of any bundle splits, and that there is a unique isomorphism class of stable bundles of any specified slope $\lambda \in \mathbf{Q}$.

In this paper we study the question of which bundles can occur as \emph{extensions} between two given vector bundles on $X$. This question turns out to have a rather combinatorial flavor, since vector bundles on the curve are determined by their Harder-Narasimhan polygons. Moreover, abstract slope theory imposes certain conditions on the HN polygon of any bundle $\ecal$ appearing as an extension of two specified bundles $\fcal_1,\fcal_2$.  Our main result is that when both bundles $\fcal_i$ are semistable, these necessary conditions are also sufficient; for a precise statement, see Theorem \ref{main_thm_3}. We also prove a natural generalization characterizing bundles admitting multi-step filtrations with specified semistable graded pieces, cf. Theorem \ref{main_thm_2}.

The results in this paper have applications towards understanding the geometry of the stack $\Bun_G$ of $G$-bundles on the Fargues-Fontaine curve, for $G$  a connected reductive group over $E$. More precisely, this stack has a natural stratification into locally closed strata $\Bun_G^{b}$ indexed by the Kottwitz set $B(G)$, and our results have implications for the problem of computing the closures of the individual strata. In the case $G=\mathrm{GL}_n$, the aforementioned strata are simply indexed by Harder-Narasimhan polygons, and the exact statement is given in Theorem \ref{strataclosure} below. These applications are the subject of a companion paper by one of us \cite{Hanvb}.

\subsection{Statement of results}

Before stating our results, we briefly recall the classification of vector bundles on $X$. A more extensive discussion of this and related background will be given in \S \ref{background}. 

\begin{thm}[Fargues-Fontaine, Kedlaya] \label{classification}\cite{FF08}\cite{Ked08} Vector bundles on $X$ enjoy the following properties: 

1) Every vector bundle $\Cal{E}$ admits a canonical Harder-Narasimhan filtration. 

2) For every rational number $\lambda$, there is a unique stable bundle of slope $\lambda$ on $X$, which is denoted $\Cal{O}(\lambda)$. Writing $\lambda = p/q$ in lowest terms, the bundle $\Cal{O}(\lambda)$ has rank $q$ and degree $p$. 

3) Any semistable bundle of slope $\lambda$ is a finite direct sum $\ocal(\lambda)^{d}$, and tensor products of semistable bundles are semistable.

4) For any $\lambda \in \mathbf{Q}$, we have 
\[
H^0(\Cal{O}(\lambda)) = 0 \text{ if and only if } \lambda<0
\]
and 
\[
H^1(\Cal{O}(\lambda)) = 0 \text{ if and only if } \lambda\geq 0.
\]
In particular, any vector bundle $\ecal$ admits a splitting
\[
\Cal{E} \simeq \bigoplus_i \Cal{O}(\lambda_i)
\]
of its Harder-Narasimhan filtration.
\end{thm}

This project began at the 2017 Arizona Winter School, as an investigation of the following question: 

\begin{quotation}
Given two vector bundles $\Cal{F}_1$ and $\Cal{F}_2$ on $X$, which polygons can occur as $\HN(\Cal{E})$ for an extension
\[
0 \rightarrow \Cal{F}_1 \rightarrow \Cal{E} \rightarrow \Cal{F}_2 \rightarrow 0?
\]
\end{quotation}

Given such a short exact sequence, the total rank and degree of $\Cal{E}$ are clearly determined by $\Cal{F}_1$ and $\Cal{F}_2$, and they determine the endpoints of $\HN(\Cal{E})$. By concavity, $\HN(\Cal{E})$ must lie above the straight line between its endpoints, which corresponds to the HN polygon of the unique semistable bundle of the correct degree and rank. Slope theory also imposes a less trivial constraint on $\HN(\ecal)$.  To explain this, let us define a partial order on the set of HN polygons by writing $P \leq P'$ if $P$ lies (nonstrictly) below $P'$ and has the same right endpoint.  It is then not difficult to show that $\HN(\ecal) \leq \HN(\Cal{F}_1 \oplus \Cal{F}_2)$, cf. Corollary \ref{polygon_bound}. 

\begin{center}
\begin{figure}[h]
\begin{tikzpicture}	
        \hspace{-0.5cm}
        \draw[black] (0,0) --node[above]{$\mathcal{F}_1$}  (1.5, 3) --node[above]{$\mathcal{F}_2$} (4,2) -- (0,0)  ;
        \draw[dashed] (0,0)--(1.5,2.2)--node[below]{$\mathcal{E}$} (2,2.4)--(4,2);
%		\draw [green, fill=green!20]  (0,0) -- (1.5,3) -- (12,8) --(0,0);
	%	\draw [orange, fill=orange!20]  (1.5,3) -- (4,5.5) --(12,8) -- (1.5,3);
		%\draw [purple, fill=purple!20]  (4,5.5) -- (7,7) --(12,8) -- (4,5.5);
		
		\draw[step=1cm,thick] (0,0) -- node[left]{} (1.5,3);
		\draw[step=1cm,thick] (1.5,3) -- node[above] {} (4,2);
		
%		\draw[step=1cm,thick] (0,0) -- (0,8);
		
	%	\draw[step=1cm,thick] (0,0) -- (12,0);
		%\draw[step=1cm,thick] (1.5,3) -- (12,8);
		%\draw[step=1cm,thick] (4,5.5) -- (12,8);
		%\draw[step=1cm,thick] (0,0) -- (12,8);
		
		\draw [fill] (0,0) circle [radius=0.05];
		
		\draw [fill] (1.5,3) circle [radius=0.05];
		\draw [fill] (1.5,2.2) circle [radius=0.03];
        \draw [fill] (2,2.4) circle [radius=0.03];
		\draw [fill] (4,2) circle [radius=0.05];
		
		\node at (0-0.8,0-0.2) {};
		\node at (4-0.2,2+0.3) {};

\end{tikzpicture}
\caption{An illustration of the scenario $\HN(\ecal) \leq \HN(\fcal_1 \oplus \fcal_2)$ for semistable $\fcal_i$'s. Here $\HN(\ecal)$ (resp. $\HN(\fcal_1 \oplus \fcal_2)$) is depicted as the dashed (resp. solid) polygonal segment.}
\end{figure}
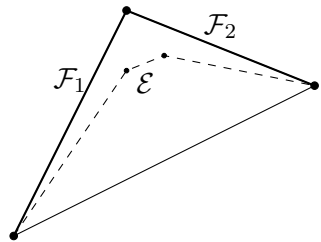
\end{center}

Our first main result is that when $\Cal{F}_1$ and $\Cal{F}_2$ are semistable, every bundle $\ecal$ satisfying these conditions is actually realized. Given a bundle $\ecal$, let $\mu(\ecal)$ be the slope of $\ecal$. 

\begin{thm}\label{main_thm_3} Let $\Cal{F}_1$ and $\Cal{F}_2$ be semistable vector bundles on $X$ such that $\mu(\fcal_1) < \mu(\fcal_2)$. Then any vector bundle $\ecal$ such that $\HN(\ecal) \leq \HN(\Cal{F}_1 \oplus \Cal{F}_2)$ is realized as an extension
\[
0 \rightarrow \Cal{F}_1 \rightarrow \Cal{E} \rightarrow \Cal{F}_2 \rightarrow 0.
\]
\end{thm}

For a more quantitative version of this result, see Theorem \ref{quantitativemainthm}.

In \S \ref{introstrategy} below, we will give a detailed summary of the proof.  For now let us simply remark that even though Theorem \ref{main_thm_3} is purely a statement about vector bundles on a Noetherian scheme, our proof uses \emph{diamonds} in a crucial way.  Furthermore, we believe that any natural proof of this result will make heavy use of diamonds; in our arguments, they arise as the correct framework for constructing moduli spaces of bundle maps with specified properties. We challenge the skeptical reader to produce a short exact sequence $$ 0 \to \ocal\left(-\tfrac{1}{2}\right)^2 \to  \ocal\left(\tfrac{1}{3}\right) \oplus \ocal\left(\tfrac{6}{5}\right) \to \ocal\left(\tfrac{9}{4}\right) \to 0$$by an argument which does not involve diamonds.

This theorem has a natural generalization to multi-step filtrations, which appears as Theorem \ref{main_thm_2} below. Before giving the precise statement, we explain the motivation. As discussed above, the stack $\Bun_n$ of rank $n$ vector bundles on the curve admits a stratification 
\[
\Bun_{n} = \bigsqcup_{P}  \Bun_{n}^{P}
\]
where the stratum $ \Bun_{n}^{P}$ parametrizes vector bundles with fixed Harder-Narasimhan polygon $P$. Immediately before the winter school, DH realized that Theorem \ref{main_thm_2}, if true, could be used to determine the precise closure relations among these strata:

\begin{thm}[Hansen]\label{strataclosure}
Let $P$ be any HN polygon of width $n$. Then we have 
\[
\ol{\Bun_{n}^{P}} = \Bun_{n}^{\geq P}
\]
as substacks of $\Bun_{n}$, where $\Bun_{n}^{\geq P}$ denotes the substack parametrizing vector bundles whose Harder-Narasimhan polygon is $\geq P$.
\end{thm}

This is in some sense a converse to the well-known fact that HN polygons jump up (in our convention) on closed subsets (cf. \cite[Theorem 7.4.5]{KL15}).

We now state our second main result, which is a generalization of Theorem \ref{main_thm_3}:

\begin{thm}\label{main_thm_2}  Let $k\geq 2$ be arbitrary, and let $\Cal{F}_1,\dots,\Cal{F}_k$ be semistable vector bundles on $X$ such that $\mu(\Cal{F}_i)<\mu( \Cal{F}_{i+1})$ for all $1 \leq i \leq k-1$. Let $\Cal{E}$ be any vector bundle on $X$ such that $\HN(\Cal{E}) \leq \HN(\Cal{F}_1 \oplus ... \oplus \Cal{F}_k)$. Then $\Cal{E}$ admits a filtration 
\[
0=\Cal{E}_0 \subset \Cal{E}_1 \subset \cdots \subset \Cal{E}_k=\Cal{E}
\]
such that $\Cal{E}_i/\Cal{E}_{i-1} \simeq \Cal{F}_i$ for all $1\leq i \leq k$.
\end{thm}
Corollary \ref{polygon_bound} again implies that the condition $\HN(\Cal{E}) \leq \HN(\Cal{F}_1 \oplus ... \oplus \Cal{F}_k)$ is necessary for $\ecal$ to admit a filtration with successive gradeds $\fcal_1,\dots,\fcal_k$, so this theorem is optimal.

Although Theorem \ref{main_thm_2} generalizes Theorem \ref{main_thm_3}, in fact the former follows quickly from the latter by induction on $k$, which we explain in 	 \S \ref{ind_step}. The rough idea is to triangulate the desired HN polygon from Theorem \ref{main_thm_2} into pieces which are then handled by Theorem \ref{main_thm_3}.

The bulk of the paper is thus concerned with establishing Theorem \ref{main_thm_3}. We now outline the strategy of our proof.
\subsection{The basic strategy}\label{introstrategy}

Theorem \ref{main_thm_3} asserts that under certain conditions, there exists an extension $\ecal$ of $\Cal{F}_2$ by $\Cal{F}_1$ with a specified HN polygon. 
To prove this, we proceed in two steps (both under the hypotheses of Theorem \ref{main_thm_3}):
\begin{description}
\item[\textbf{Step (1)}] We show that $\Cal{E}$ admits some surjection $\Cal{E} \twoheadrightarrow \Cal{F}_2$. 
\item[\textbf{Step (2)}] We show that if $\Cal{E}$ admits a surjection $\ecal \twoheadrightarrow \fcal_2$, then it admits such a surjection with kernel isomorphic to $\Cal{F}_1$. 
\end{description}

Let us sketch the arguments for these steps.\footnote{Strictly speaking, the following arguments work only after an easy reduction of Theorem \ref{main_thm_3} to the case where $\mu(\fcal_2)$ is strictly larger than the maximal slope of $\ecal$; see \S3.4 for details.}

\noindent \textbf{Step (1).} To carry out Step (1), we construct a moduli space $\Hom(\Cal{E}, \Cal{F}_2)$ parametrizing bundle maps $\ecal \to \fcal_2$, as well as an open subspace $\Surj(\Cal{E}, \Cal{F}_2)$ parametrizing surjective maps. These objects are easily defined as functors on the category of perfectoid spaces over $F$, but they typically are not representable by adic spaces.  We prove that they are \emph{diamonds} in the sense of Scholze \cite{Sch,SW15}; in fact, we show that they are ``locally spatial, equidimensional diamonds''.  We emphasize that our use of diamonds here is genuinely necessary for the arguments.  In particular, we crucially use the fact that a locally spatial diamond $X$ has a naturally associated locally spectral topological space $|X|$, and that the Krull dimension of this space gives rise to a reasonable dimension theory for such $X$'s, cf. \S \ref{diamondbasics}-\ref{dim_formulas}.

In any case, once we have proved that $\Surj(\ecal,\fcal_2)$ is a reasonable diamond, our task is to show that it is non-empty, which we deduce from the following claim.

\begin{claim}\label{claim1}
The dimension of the space of maps $\Cal{E} \rightarrow \Cal{F}_2$ which factor through a proper sub-bundle of $\Cal{F}_2$ is \emph{strictly smaller} than the dimension of $ \Hom(\Cal{E}, \Cal{F}_2)$.
\end{claim}

Roughly speaking, we prove this by stratifying $\Hom(\ecal,\fcal_2)$ according to the isomorphy type of the image of the map $\ecal\to \fcal_2$ and then computing the dimensions of these individual strata. The computation of these dimensions is non-trivial, and relies on the theory developed in \S \ref{dim_formulas}-\ref{bundlemaps}; the final formula is stated in Lemma \ref{maindimlemma1}. Once we have this formula, Claim \ref{claim1} is reduced to a certain finite collection of inequalities, whose verification is a mechanical computation in any given example, but which presents a difficult combinatorial problem in general. We eventually realized that it was natural to interpret the terms appearing in the inequality as sums of \emph{areas} of certain polygons (related to the slopes in the Harder-Narasimhan polygons), and that the inequality could then be seen by a complicated transformation of the polygons after which one area would clearly dominate the other. The precise statement and its proof is given in \S \ref{step_1}. \\

\noindent \textbf{Step (2).} The argument here is similar to that for Step (1). For any given bundle $\kcal$, we define a moduli space $\Surj(\Cal{E}, \Cal{F}_2)^{\Cal{K}}$ parametrizing surjections $\Cal{E} \rightarrow \Cal{F}_2$ with kernel isomorphic to $\Cal{K}$. Again, this is a locally spatial diamond. For $\Cal{K} = \Cal{F}_1$, this object is open in $\Surj(\Cal{E}, \Cal{F}_2)$ for essentially formal reasons (the openness of semistable loci). Again the task is to show $\Surj(\Cal{E}, \Cal{F}_2)^{\Cal{F}_1}$ is non-empty, which we achieve by showing that $\Surj(\Cal{E}, \Cal{F}_2)^{\Cal{K}}$ has strictly smaller dimension for any other isomorphism class of $\kcal$'s. This again amounts to a certain collection of inequalities, which we can again interpret via comparison of areas, and which after some manipulation become geometrically evident. The precise arguments are given in \S \ref{step_2}. 

\subsection*{Acknowledgments}The authors came together around these questions at the 2017 Arizona Winter School, under the umbrella of Kiran Kedlaya's project group.  We would like to heartily thank Kiran for giving us this opportunity.  We would also like to thank the organizers of the Winter School for creating such a wonderful experience, and we gratefully acknowledge NSF support of the Winter School via grant DMS-1504537.

DH is grateful to Christian Johansson for some useful conversations about the material in \S \ref{dim_formulas}, and Peter Scholze for providing early access to the manuscript \cite{Sch} and for some helpful conversations about the results therein. The project group students (CB, TF, SH, QL, AW, and LY) thank DH and Kiran Kedlaya for suggesting the problem. TF gratefully acknowledges the support of an NSF Graduate Fellowship. LY gratefully acknowledges the support of the National Defense Science and Engineering Graduate Fellowship. We would also like to thank David Linus Hamann and the referee for their valuable feedback on the first version of this paper.

\section{Background on the Fargues-Fontaine curve}\label{background}

\subsection{The Fargues-Fontaine curve, and vector bundles}
We begin by recalling the definition of the Fargues-Fontaine curve $X := X_{E,F}$. This is already a slightly subtle issue, as there are several different incarnations of the curve: as a scheme (which was the original definition of Fargues and Fontaine), as an \emph{adic space}, and as a \emph{diamond}. In this paper we will only need to use the classification of vector bundles on $X$ as black box, so we won't actually need any technical details of the construction of $X$ itself. Therefore, we content ourselves with giving just a cursory introduction to the construction. 

\begin{defn}\label{adicFFC} Let $E$ be a finite extension of $\Q_p$, with uniformizer $\pi$, ring of integers $E^\circ$, and residue field $\mathbf{F}_q$ where $q=p^f$, and let $F/ \mathbf{F}_q$ be an algebraically closed perfectoid field, with ring of integers $F^\circ$ and pseudouniformizer $\varpi$. 

Let $W_{E^\circ}(F^\circ)=W(F^\circ) \otimes_{W(\mathbf{F}_q)} E^\circ$ be the ramified Witt vectors of $F^\circ$ with coefficients in $E^\circ$. Define
\[
\Cal{Y}_{E,F}=\Spa(W_{E^\circ}(F^\circ))\setminus\{|p[\varpi]|=0\},
\]
and let $\phi:\Cal{Y}_{E,F}\to \Cal{Y}_{E,F}$ be the Frobenius automorphism of $\Cal{Y}_{E,F}$ induced by the natural $q$-Frobenius $\varphi_q = \varphi^f \otimes 1$ on $W_{E^{\circ}}(F^\circ)$. The (mixed-characteristic) \emph{adic Fargues-Fontaine curve} $\Cal{X}_{E,F}$ is
\[
\Cal{X}_{E,F}=\Cal{Y}_{E,F}/\phi^\Z.
\]
\end{defn}

\begin{prop}[Kedlaya] For any pair $(E,F)$ as above, $\xcal_{E,F}$ is a  Noetherian adic space over $\Spa\,E$.
\end{prop}
\begin{proof}This is one of the main results of \cite{KedNoeth}.
\end{proof}

\begin{remark}
There is also an equal-characteristic Fargues-Fontaine curve associated with a pair $(E,F)$ as above, but where $E$ is now taken to be a finite extension of $\mathbf{F}_p ((t))$.  In this paper, we will largely focus on the mixed-characteristic curve: our main results hold verbatim in the equicharacteristic setting, but the proofs are strictly easier.
\end{remark}

By descent, a vector bundle on $\Cal{X}_{E,F}$ is the same as a $\phi$-equivariant vector bundle on $\Cal{Y}_{E,F}$, that is, a vector bundle $\tilde{\ecal}$ on $\Cal{Y}_{E,F}$ together with an isomorphism $\phi^*\tilde{\ecal}\overset{\sim}{\to}\tilde{\ecal}$. 

\begin{defn}
\label{o-r-over-s}
If $\lambda = r/s$ is a rational number written in lowest terms with $s$ positive, we define a vector bundle $\ocal(\lambda)$ on $\Cal{X}_{E,F}$ as the descent of the trivial rank $s$ vector bundle on $\Cal{Y}_{E,F}$ equipped with the following $\phi$-equivariant structure. Let $v_1,\dotsc,v_s$ be a trivializing basis of $\tilde{\ecal}:=\Cal{O}_{\Cal{Y}_{E,F}}^{\oplus s}$, and by abuse of notation view it as a trivializing basis for $\phi^* \tilde{\ecal}$ as well. Define $\phi^* \tilde{\ecal} \xrightarrow{\sim} \tilde{\ecal}$ by 
\begin{align*}
 v_1 & \mapsto v_2 \\
v_2 & \mapsto v_3 \\
 & \dotsb \\
 v_{s-1} & \mapsto v_s \\
v_s & \mapsto \pi^{-r}v_1. \\
\end{align*}
\end{defn}

As previously discussed, there is also a \emph{scheme-theoretic} Fargues-Fontaine curve $X_{E,F}$. 

\begin{defn}
Let $E$ and $F$ be as in Definition \ref{adicFFC}. We define the \emph{scheme-theoretic Fargues-Fontaine curve $X_{E,F}$} to be 
\[
X_{E,F} := \Proj \left( \bigoplus_{n \geq 0 } H^0(\Cal{X}_{E,F}, \Cal{O}(n) ) \right).
\]
\end{defn}

\begin{remark}
The original definition (\cite{FF08}, \cite{FF14}) of the scheme-theoretic Fargues-Fontaine curve was given in terms of certain period rings of $p$-adic Hodge theory:
\[
X_{E,F} = ``\Proj \left( \bigoplus_{n \geq 0 } B^{\varphi_q = \pi^n}\right)"
\]
(\cite{FF14}, p.22); however, the vector spaces $B^{\varphi_q = \pi^n}$ coincide with $H^0(\Cal{X}_{E,F}, \Cal{O}(n) )$, and so the definitions agree. 
\end{remark}

\begin{prop}The scheme $X_{E,F}$ is Noetherian, connected, and regular of (absolute) dimension one.
\end{prop}

It is known (``GAGA for the Fargues-Fontaine curve'', \cite{KL15}, Theorem 6.3.12) that there is a natural map 
\[
\Cal{X}_{E,F} \rightarrow X_{E,F}
\]
which induces by pullback an equivalence of categories of vector bundles. Therefore, we can and do speak interchangeably about vector bundles on $\Cal{X}_{E,F}$ and $X_{E,F}$. For ease of notation, we henceforth set $X = X_{E,F}$.

\subsection{Harder-Narasimhan filtrations and polygons}

The following fundamental result of Fargues and Fontaine is the key to developing a slope theory for bundles on $X$.

\begin{prop}[\cite{FF08}]The curve $X$ is \emph{complete} in the sense that if $f \in k(X)$ is any nonzero rational function on $X$, then the divisor of $f$ has degree zero.
\end{prop}

This implies, in particular, that if $\lcal$ is a line bundle on $X$ and $s$ is any nonzero meromorphic section of $\lcal$, then $\deg \lcal \overset{\mrm{def}}{=} \deg \mathrm{div}(s)$ is well-defined independently of the choice of $s$.

\begin{defn}
If $\Cal{E}$ is a vector bundle on $X$, we define the \emph{degree} of $\Cal{E}$ by $\deg \ecal := \deg \wedge^{\mathrm{rank}\,\ecal} \ecal$, and the \emph{slope} of $\Cal{E}$ by $\mu(E) := \frac{\deg\Cal{E}}{\rank \Cal{E}}$. 
\end{defn}

\begin{example}\label{tensor_bundles}
As a vector bundle on $X$, $\ocal(r/s)$ has rank $s$ and degree $r$, hence slope $r/s$. One may check by hand from Definition~\ref{o-r-over-s} that
\[
\ocal\left(\frac rs\right)\ten\ocal\left(\frac{r'}{s'}\right)\simeq\ocal\left(\frac rs+\frac{r'}{s'}\right)^{\dsm\gcd(s s',rs'+r's)}.
\]
In particular, $\ocal(r/s)\ten\ocal(r'/s')$ has rank $ss'$, degree $rs'+r's$, and slope $r/s+r'/s'$.
\end{example}

We recall the usual notions of (semi)stability. 

\begin{defn}
We say that a vector bundle $\Cal{E}$ on $X$ is \emph{stable} if it has no proper, non-zero subbundles $\Cal{F} \subset \Cal{E}$ with $\mu(\Cal{F}) \geq \mu(\Cal{E})$. We say that $\Cal{E}$ is \emph{semistable} if it has no proper, non-zero subbundles $\Cal{F} \subset \Cal{E}$ with $\mu(\Cal{F}) > \mu(\Cal{E})$.
\end{defn}

Stable and semistable vector bundles on $X$ turn out to admit a complete classification:

\begin{prop}[\cite{FF08,Ked17}]
For any $\lambda \in \mathbf{Q}$, the vector bundle $\Cal{O}(\lambda)$ is stable, and any finite direct sum $\ocal(\lambda)^{\oplus n}$ is semistable.  Moreover, every semistable vector bundle of slope $\lambda$ is isomorphic to some finite direct sum $\ocal(\lambda)^{\oplus n}$.
\end{prop}
Combining this classification with Example \ref{tensor_bundles}, one immediately deduces that tensor products of semistable bundles are semistable.

\begin{defn}
A \emph{Harder-Narasimhan (HN) filtration} of a vector bundle $\Cal{E}$ is a filtration of $\ecal$ by subbundles
\[
0 = \Cal{E}_0 \subset \Cal{E}_1 \subset \dotsb \subset \Cal{E}_m = \Cal{E}
\]
such that each successive quotient $\Cal{E}_i/\Cal{E}_{i-1}$ is a semistable vector bundle which is of slope $\mu_i$, and such that 
\[
\mu_1 > \mu_2 > \dotsb > \mu_m.
\]Keeping this notation, the \emph{Harder-Narasimhan (HN) polygon} of $\Cal{E}$ is the upper convex hull of the points $(\rank \ecal_i, \deg \ecal_i)$. 
\end{defn}

\begin{remark}
Our convention for HN polygons is opposed to the usual convention for Newton polygons, which are usually lower convex hulls.
\end{remark}

\begin{thm}[Fargues-Fontaine] Every vector bundle on $X$ admits a canonical and functorial Harder-Narasimhan filtration.
\end{thm}
\begin{proof} This follows from \cite[Th\'eor\`eme 6.5.2]{FF08} together with the discussion in \cite[\S 5.5]{FF08}.
\end{proof}

\begin{example}
In Figure \ref{HNpolygon_example_fig} we depict an example of the HN polygon associated to a vector bundle whose Harder-Narasimhan filtration has the form 
\[
0 \subset \Cal{E}_1 \subset \Cal{E}_2 \subset \Cal{E}_3 \subset \Cal{E}_4 = \Cal{E}
\]
with $\mu_i := \mu(\Cal{E}_i/\Cal{E}_{i-1})$.
\begin{center}\begin{figure}[h]\begin{tikzpicture}	
        \hspace{-0.5cm}
        \draw[black, fill = blue!20] (0,0) -- (1.5, 3) -- (4,5.5) --  (7,7) -- (12,8) -- (0,0)  ;
%		\draw [green, fill=green!20]  (0,0) -- (1.5,3) -- (12,8) --(0,0);
	%	\draw [orange, fill=orange!20]  (1.5,3) -- (4,5.5) --(12,8) -- (1.5,3);
		%\draw [purple, fill=purple!20]  (4,5.5) -- (7,7) --(12,8) -- (4,5.5);
		
		\draw[step=1cm,thick] (0,0) -- node[right] {$\mu_1$} node[sloped, above, pos=0.99] {$(\rank \Cal{E}_1, \deg \Cal{E}_1)$} (1.5,3);
		\draw[step=1cm,thick] (1.5,3) -- node[below=3pt] {$\mu_2$} node[sloped, above, pos=0.99] {$(\rank \Cal{E}_2, \deg \Cal{E}_2)$} (4,5.5);
		\draw[step=1cm,thick] (4,5.5) -- node[below] {$\mu_3$} node[sloped, above, pos=0.99]  {$(\rank \Cal{E}_3, \deg \Cal{E}_3)$} (7,7);
		\draw[step=1cm,thick] (7,7) -- node[below] {$\mu_4$} (12,8);
		
%		\draw[step=1cm,thick] (0,0) -- (0,8);
		
	%	\draw[step=1cm,thick] (0,0) -- (12,0);
		%\draw[step=1cm,thick] (1.5,3) -- (12,8);
		%\draw[step=1cm,thick] (4,5.5) -- (12,8);
		%\draw[step=1cm,thick] (0,0) -- (12,8);
		
		\draw [fill] (0,0) circle [radius=0.05];
		
		\draw [fill] (1.5,3) circle [radius=0.05];
		
		\draw [fill] (4,5.5) circle [radius=0.05];
		
		\draw [fill] (7,7) circle [radius=0.05];

		\draw [fill] (12,8) circle [radius=0.05];
		\node at (0-0.8,0-0.2) {$(0,0)$};

		\node at (12-0.2,8+0.3) {$(\rank \Cal{E}_4, \deg \Cal{E}_4)$};
			 \useasboundingbox (-2,0);
\end{tikzpicture}
\caption{The HN polygon of a vector bundle with a 4-step HN filtration.}\label{HNpolygon_example_fig}
\end{figure}
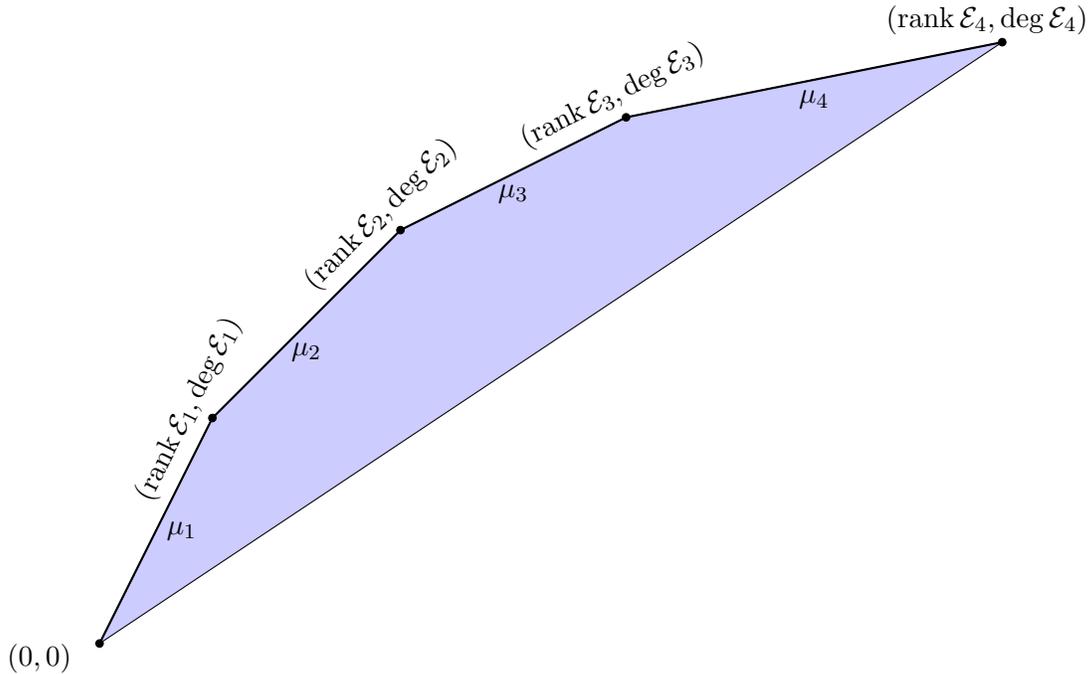
\end{center}
\end{example}

Combining the classification of semistable bundles with an explicit calculation of the cohomology groups $H^i(X,\ocal(\lambda))$, one deduces that the HN filtration of any vector bundle on $X$ is split.  The precise statement is as follows.

\begin{thm}[Kedlaya, Fargues-Fontaine]\cite{FF08},\cite{Ked08}\label{FFBClassification} \hfill

\noindent 1) For any given rational number $\lambda$, we have 
\[
H^0(\Cal{O}(\lambda)) = 0 \text{ if and only if } \lambda<0
\]
and 
\[
H^1(\Cal{O}(\lambda)) = 0 \text{ if and only if } \lambda\geq 0.
\]
2) Any vector bundle $\Cal{E}$ on $X$ admits a direct sum decomposition
\[
\Cal{E} \simeq \bigoplus_i \Cal{O}(\lambda_i)
\]
where the $\lambda_i$s run over the Harder-Narasimhan slopes of $\ecal$ counted with the appropriate multiplicities.
\end{thm}

\begin{cor}\label{HNsplit}
A vector bundle $\ecal$ on $X$ is determined up to isomorphism by its HN polygon $\HN(\ecal)$.
\end{cor}

%\begin{example}
%Figure \ref{HNpoly_order_ex_fig} depicts the order relation between HN polygons. If $\HN_1$ is the polygon shaded in blue and $\HN_2$ is the polygon shaded in green then $\HN_1 \leq \HN_2$. 
%\end{example}
\begin{center}
\begin{figure}[h]\begin{tikzpicture}	[scale=0.6]
        \hspace{-0.5cm}
        \draw[black, fill = green!20] (0,0) -- (1.5, 3) -- (4,5.5) --  (7,7) -- (12,8) -- (0,0)  ;
 	\draw [green, fill=blue!20]  (0,0) -- (1.5,3) -- (12,8) --(0,0);
	%	\draw [orange, fill=orange!20]  (1.5,3) -- (4,5.5) --(12,8) -- (1.5,3);
		%\draw [purple, fill=purple!20]  (4,5.5) -- (7,7) --(12,8) -- (4,5.5);
		
		\draw[step=1cm,thick] (0,0) -- node[left] {} (1.5,3);
		\draw[step=1cm,thick] (1.5,3) -- node[above] {} (4,5.5);
		\draw[step=1cm,thick] (4,5.5) -- node[above] {} (7,7);
		\draw[step=1cm,thick] (7,7) -- node[above] {}(12,8);
		
%		\draw[step=1cm,thick] (0,0) -- (0,8);
		
	%	\draw[step=1cm,thick] (0,0) -- (12,0);
		%\draw[step=1cm,thick] (1.5,3) -- (12,8);
		%\draw[step=1cm,thick] (4,5.5) -- (12,8);
		%\draw[step=1cm,thick] (0,0) -- (12,8);
		
		\draw [fill] (0,0) circle [radius=0.05];
		
		\draw [fill] (1.5,3) circle [radius=0.05];
		
		\draw [fill] (4,5.5) circle [radius=0.05];
		
		\draw [fill] (7,7) circle [radius=0.05];

		\draw [fill] (12,8) circle [radius=0.05];
		\node at (0-0.8,0-0.2) {};
		\node at (4-0.2,5.5+0.3) {};
		\node at (7-0.2,7+0.3){} ;
		\node at (12-0.2,8+0.3) {};
		\node at (1.5-0.2,3+0.3) {};
		
\end{tikzpicture}
\caption{This depicts the order relation between HN polygons. If $\HN_1$ is the polygon shaded in blue and $\HN_2$ is the polygon shaded in green, then $\HN_1 \leq \HN_2$. }
%\caption{Depiction of the order relation between HN polygons: The polygon bounded by the solid black lines is $\geq$ the polygon bounding the blue region.}\label{HNpoly_order_ex_fig}
\end{figure}
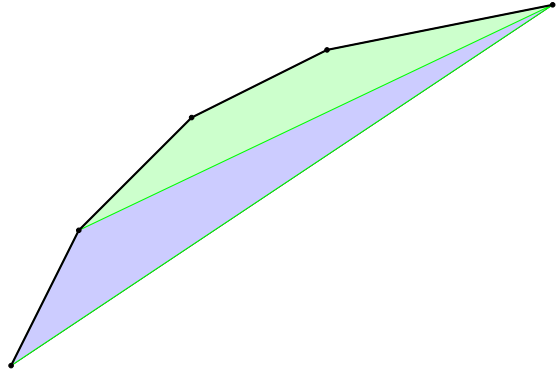
\end{center}

\begin{defn}
We define a partial order on HN polygons as follows: if $\HN_1$ and $\HN_2$ are two polygons, then we say that $\HN_1 \leq \HN_2$ if $\HN_1$ lies on or below $\HN_2$ and the polygons have the same endpoints. 
\end{defn}

Under this partial ordering, the Harder-Narasimhan filtration of a bundle $\ecal$ satisfies the following extremality property among all filtrations with semistable graded pieces:
\begin{cor}\label{polygon_bound}
Let $\ecal$ be a vector bundle on $X$, and let
\[
0 = \Cal{E}_0 \subset \Cal{E}_1 \subset \ldots \subset \Cal{E}_m = \Cal{E}
\]
be any filtration of $\ecal$ such that each graded piece $\Cal{E}_i/\Cal{E}_{i-1} $ is a semistable vector bundle. Then 
\[
\HN(\Cal{E}) \leq \HN(\Cal{E}_1/\Cal{E}_0 \oplus \ldots \oplus \Cal{E}_m/\Cal{E}_{m-1}).
\]
\end{cor}
\begin{proof}This is \cite[Corollary 3.4.18]{Ked17}.
\end{proof}

\begin{defn}\label{slope_leq_part}
Let $\Cal{E}$ be a vector bundle on $X$ with Harder-Narasimhan filtration $$0 = \Cal{E}_0 \subset \Cal{E}_1 \subset \ldots \subset \Cal{E}_m = \Cal{E}.$$ Then for any $\lambda \in \mathbf{Q}$, we define $\ecal^{\geq \lambda}$ (resp. $\ecal^{> \lambda}$) to be the subbundle of $\ecal$ given by $\ecal_i$ for the largest value of $i$ such that $\mu(\ecal_i/\ecal_{i-1}) \geq \lambda$ (resp. such that $\mu(\ecal_i/\ecal_{i-1}) > \lambda$). We also define $\ecal^{<\lambda} = \ecal / \ecal^{\geq \lambda}$ and $\ecal^{\leq \lambda} = \ecal / \ecal^{> \lambda}$.
\end{defn}

\subsection{Geometric interpretation of degrees}

In a number of our later arguments, we will need to understand quantities of the form $\deg\,(\ecal_1 \otimes \ecal_2^\vee)^{\geq 0}$ for some fairly arbitrary bundles $\ecal_1$ and $\ecal_2$.  Here we develop some preliminary language for doing this.  The results in this section will not be necessary until \S4.
    
\begin{defn}\label{order_relation}	For two vectors $v, w \in \mathbb{R}^2$ with nonzero $x$-coordinates, we write $v \prec w$ (resp. $v \preceq w$) if the slope of $v$ is less than (resp. not greater than) the slope of $w$. 

	\end{defn}
	
We can determine this relation between $v,w$ using their (two-dimensional) cross product as follows:

\begin{lemma}\label{relation_area} We have the following characterization of the order $\prec$: 
	\begin{itemize}
		\item if the $x$-coordinates of $v$ and $w$ have the same sign, $v \prec w$ (resp. $v \preceq w$) if and only if $v \times w >0$ (resp. $v \times w \geq 0$); 
		\item if the $x$-coordinates of $v$ and $w$ have different signs, $v \prec w$ (resp. $v \preceq w$) if and only if $v \times w <0$ (resp. $v \times w \leq 0$). 
	\end{itemize}
	\end{lemma}
	
	\begin{proof}
	This is straightforward.
	\end{proof}

		Let $\mathcal{V}$ be a vector bundle on $X$. By Theorem \ref{FFBClassification}, $\mathcal{V}$ decomposes into a direct sum
	\[ \mathcal{V} = \bigoplus_{i=1}^s \mathcal{O}(d_i/h_i)^{m_i}\]
	where $d_1/h_1 > d_2/h_2 > \cdots > d_s/h_s$ and $d_i, h_i$ are relatively prime for each $i= 1, 2, \cdots, s$. Then the Harder-Narasimhan polygon of $\mathcal{V}$ is built out of the sequence of vectors
	\[ v_1 \succ v_2 \succ \cdots \succ v_s\]
	where $v_i = (m_i h_i, m_i d_i)$ is the $i$-th edge in $\HN(\mathcal{V})$. 
	
\begin{defn} \label{HN vector notation}
Keeping the notation of the preceding paragraph, the \emph{HN vectors of $\Cal{V}$} will be denoted by $\HNvec(\mathcal{V}) := (v_i)_{1\leq i \leq s}$. Note that the isomorphism class of $\mathcal{V}$ is uniquely determined by $\HNvec(\mathcal{V})$. 
\end{defn}

	\begin{lemma}\label{diamond lemma for nonnegative degree}
		Let $\mathcal{V}$ and $\mathcal{W}$ be any vector bundles on $X$ with $\HNvec(\mathcal{V}) = (v_i)$ and $\HNvec(\mathcal{W}) = (w_j)$. Then we have 
		\begin{equation}\label{deg_formula}
		\deg(\mathcal{V}^\vee \otimes \mathcal{W}) = \sum_{i, j} v_i \times w_j
		\end{equation}
		and 
		\begin{equation}\label{deg0_formula}
		\deg(\mathcal{V}^\vee \otimes \mathcal{W})^{\geq 0} = \sum_{v_i \preceq w_j} v_i \times w_j.
		\end{equation}
	\end{lemma}

	\begin{proof}
		Let us first consider the case when $\mathcal{V}$ and $\mathcal{W}$ are both semistable. In this case, both $\HNvec(\mathcal{V})$ and $\HNvec(\mathcal{W})$ consist of a single element, namely $(\rank(\mathcal{V}), \deg(\mathcal{V}))$ and $(\rank(\mathcal{W}), \deg(\mathcal{W}))$, respectively. Then \eqref{deg_formula} follows directly from Example~\ref{tensor_bundles}, and \eqref{deg0_formula} follows from \eqref{deg_formula} by Part 1 of Lemma~\ref{relation_area}. %\eqref{deg0_formula} follows from \eqref{deg_formula} using the fact that $\mathcal{V}^\vee \otimes \mathcal{W}$ is semistable (again by Example \ref{tensor_bundles}).

The general case now follows by observing that both sides of \eqref{deg_formula} and \eqref{deg0_formula} are linear in both $\Cal{V}$ and $\Cal{W}$, and using Theorem \ref{classification} to see that every vector bundle decomposes as a direct sum of semistable bundles. \end{proof}

	\begin{prop}\label{geometric diamond lemma for one vector bundle} Let $\mathcal{V}$ be a vector bundle on $X$. Then $\deg(\mathcal{V}^\vee \otimes \mathcal{V})^{\geq 0}$ is equal to twice the area of the region enclosed between $\HN(\mathcal{V})$ and the line segment joining the two endpoints of $\HN(\mathcal{V})$. In particular, $\deg(\mathcal{V}^\vee \otimes \mathcal{V})^{\geq 0}=0$ if and only if $\mathcal{V}$ is semistable. 
	\end{prop}

	\begin{proof}
		Write $\HNvec(\mathcal{V}) = (v_i)_{1\leq i \leq s}$, and let $O=P_0, P_1, P_2, \dotsc, P_s$ be the breakpoints of $\HN(\mathcal{V})$, listed in order of increasing $x$-coordinates (see Figure \ref{diamond_Lemma_fig}). Note that $v_i = \overrightarrow{P_{i-1} P_i}$. 

\begin{center}
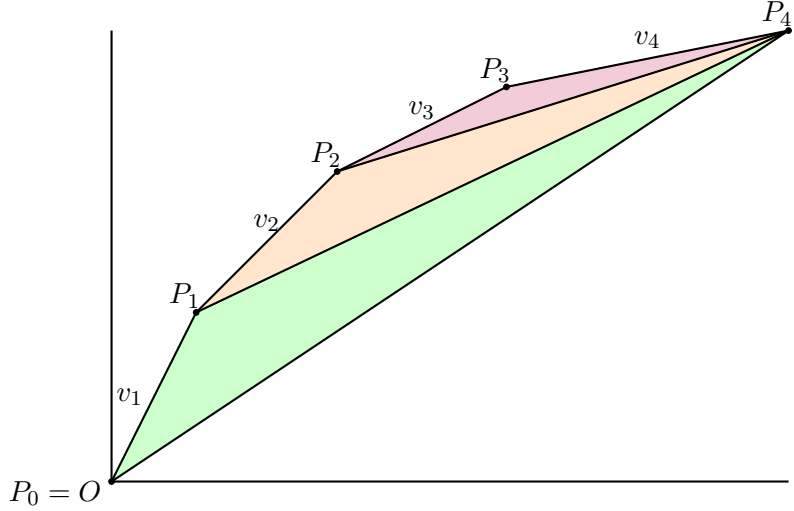
\begin{figure}[h]
\begin{tikzpicture}	[scale=0.75]
        \hspace{-0.5cm}
		\draw [green, fill=green!20]  (0,0) -- (1.5,3) -- (12,8) --(0,0);
		\draw [orange, fill=orange!20]  (1.5,3) -- (4,5.5) --(12,8) -- (1.5,3);
		\draw [purple, fill=purple!20]  (4,5.5) -- (7,7) --(12,8) -- (4,5.5);
		
		\draw[step=1cm,thick] (0,0) -- node[left] {$v_1$} (1.5,3);
		\draw[step=1cm,thick] (1.5,3) -- node[above] {$v_2$} (4,5.5);
		\draw[step=1cm,thick] (4,5.5) -- node[above] {$v_3$} (7,7);
		\draw[step=1cm,thick] (7,7) -- node[above] {$v_4$} (12,8);
		
		\draw[step=1cm,thick] (0,0) -- (0,8);
		
		\draw[step=1cm,thick] (0,0) -- (12,0);
		\draw[step=1cm,thick] (1.5,3) -- (12,8);
		\draw[step=1cm,thick] (4,5.5) -- (12,8);
		\draw[step=1cm,thick] (0,0) -- (12,8);
		
		\draw [fill] (0,0) circle [radius=0.05];
		
		\draw [fill] (1.5,3) circle [radius=0.05];
		
		\draw [fill] (4,5.5) circle [radius=0.05];
		
		\draw [fill] (7,7) circle [radius=0.05];

		\draw [fill] (12,8) circle [radius=0.05];
		\node at (0-1,0-0.2) {$P_0=O$};
		\node at (4-0.2,5.5+0.3) {$P_2$};
		\node at (7-0.2,7+0.3) {$P_3$};
		\node at (12-0.2,8+0.3) {$P_4$};
		\node at (1.5-0.2,3+0.3) {$P_1$};
		\useasboundingbox (-3,0);
		
\end{tikzpicture}
\caption{Visual depiction of Proposition \ref{geometric diamond lemma for one vector bundle}.}\label{diamond_Lemma_fig}
\end{figure}
\end{center}

		By definition, we have $v_1 \succ v_2 \succ \cdots \succ v_s$. Hence Lemma \ref{diamond lemma for nonnegative degree} yields
		\begin{align*} \deg(\mathcal{V}^\vee \otimes \mathcal{V})^{\geq 0} &= \sum_{i\geq j} v_i \times v_j  = \sum_{j=1}^s (v_j + v_{j+1} + \cdots + v_s) \times v_j \\
		&= \sum_{j=1}^s  \big(\overrightarrow{P_{j-1} P_s} \times \overrightarrow{P_{j-1} P_j}\big)\\
		&= \sum_{j=1}^s 2 \cdot \text{Area}(\bigtriangleup P_{j-1}P_j P_s),
		\end{align*}
		which is clearly equal to twice the area of the region enclosed between $\HN(\mathcal{V})$ and the line segment joining $O$ and $P_s$. \end{proof}
	
	\begin{cor}\label{diamond inequality}
		Let $\mathcal{V}$ and $\mathcal{W}$ be any vector bundles on $X$ such that $\HN(\mathcal{V}) \leq \HN(\mathcal{W})$. Then $\deg(\mathcal{V}^\vee \otimes \mathcal{V})^{\geq 0} \leq \deg(\mathcal{W}^\vee \otimes \mathcal{W})^{\geq 0}$. 
	\end{cor}
	\begin{proof}
		This is an immediate consequence of Proposition \ref{geometric diamond lemma for one vector bundle}. 
	\end{proof}

\section{Diamonds, vector bundles, and dimensions}\label{diamond_dim}

\subsection{Recollections on diamonds}\label{diamondbasics}

In this section we (very) briefly introduce the language of diamonds.  Everything here can be found in \cite[\S7-8, \S10-11, \S18, \S21]{Sch}.

\begin{defn}
\begin{enumerate}
\item A map $Y\to X$ of affinoid perfectoid spaces is \emph{affinoid pro-\'etale} if $Y$ is isomorphic to the limit of some cofiltered system of affinoid perfectoid spaces $Y_i$ \'etale over $X$. 

\item A map $f: Y\to X$ of perfectoid spaces is \emph{pro-\'etale} if $X$ and $Y$ can be covered by open affinoid perfectoid subsets $U_i \subset X, V_{ij} \subset Y$ such that $f|_{V_{ij}}: V_{ij} \to X$ factors through the inclusion of $U_i$ and such that the induced map $V_{ij} \to U_i$ is affinoid pro-\'etale.

\item A map $f: Y\to X$ of perfectoid spaces is a \emph{pro-\'etale cover} if $f$ is pro-\'etale and for any qc open subset $U \subset X$, there is some qc open subset $V \subset Y$ with $f(V)=U$.

\item The \emph{big pro-\'etale site} $\mathrm{Perf}$ is the site of perfectoid spaces in characteristic $p$, with covers given by pro-\'etale covers.

\end{enumerate}
\end{defn}

\begin{prop}The site $\mathrm{Perf}$ is subcanonical.
\end{prop}

We slightly abuse notation and write $X=h_X$ for any $X \in \mathrm{Perf}$, i.e. we freely identify a characteristic $p$ perfectoid space $X$ with its Yoneda embedding into sheaves on $\mathrm{Perf}$. By the previous proposition, there is little harm in this.

\begin{defn} A \emph{diamond} is a sheaf $Y$ on $\mathrm{Perf}$ such that $$ Y \simeq \mathrm{Coeq}\left( R\rightrightarrows X \right), $$ where $R \rightrightarrows X$ is a pro-\'etale equivalence relation in characteristic $p$ perfectoid spaces.
\end{defn}

If $Y$ is a diamond with a presentation $Y \simeq \mathrm{Coeq}\left( R\rightrightarrows X \right)$, then $|R| \rightrightarrows |X|$ defines an equivalence relation on $|X|$, and we define $|Y| = |X|/|R|$ as the associated quotient, topologized via the quotient topology.  The topological space $|Y|$ is well-defined independently of the chosen presentation, and is functorial in $Y$.

\begin{defn}A subfunctor $Y' \subset Y$ of a diamond $Y$ is an \emph{open subdiamond} if for any perfectoid space $T$ with a map $T \to Y$, $Y'\times_{Y}T$ is representable and $Y' \times_Y T \to T$ is an open immersion.
\end{defn}
One then checks that if $Y$ is any diamond, the association $Y' \rightsquigarrow |Y'|$ defines an inclusion-preserving bijection between open subdiamonds of $Y$ and open subsets of $|Y|$. 

We also recall the relation of diamonds with adic spaces:

\begin{prop}There is a natural functor $X \mapsto X^\lozenge$ from analytic adic spaces over $\Spa(\mathbf{Z}_p,\mathbf{Z}_p)$ to diamonds, extending the functor $X \mapsto h_{X^\flat}$ on perfectoid spaces and inducing a functorial homeomorphism $|X^\lozenge| \cong |X|$.
\end{prop}

If $\Spa(A,A^+)$ is an affinoid adic space, we'll often write $\Spd(A,A^+):=\Spa(A,A^+)^\lozenge$ and $\Spd\,A := \Spd(A,A^\circ)$.

Next we recall a flexible variant of pro-\'etale morphisms in the setting of diamonds.

\begin{defn}
\begin{enumerate}
\item A perfectoid space $X$ is \emph{strictly totally disconnected} if it is qcqs and every connected component of $X$ is of the form $\Spa(K,K^+)$ for some algebraically closed perfectoid field $K$.
\item A map of diamonds $Y\to X$ is \emph{quasi-pro-\'etale} if for any strictly totally disconnected perfectoid space $T$ with a map $T\to X$, the sheaf $Y \times_X T$ is representable and the induced map of perfectoid spaces $Y \times_X T \to T$ is pro-\'etale. 
\end{enumerate}
\end{defn}

We now recall some useful classes of diamonds and morphisms of diamonds.

\begin{defn}
\begin{enumerate}
\item A diamond $Y$ is  \emph{quasicompact (qc)} if it admits a presentation $Y \simeq \mathrm{Coeq}\left( R\rightrightarrows X \right)$ with $X$ quasicompact. A morphism of diamonds $Y\to X$ is  \emph{quasicompact} if $Y \times_X U$ is quasicompact for all quasicompact diamonds $U$ with a map $U\to X$.

\item A diamond $Y$ is  \emph{quasiseparated (qs)} if for any qc diamonds $U \to Y, V \to Y$, the fiber product $U \times_Y V$ is qc. A morphism of diamonds $Y\to X$ is  \emph{quasiseparated} if the relative diagonal $Y \to Y \times_X Y$ is quasicompact.

\item A diamond $Y$ is \emph{separated} (resp.  \emph{partially proper}) if it is quasiseparated and for all choices of a characteristic $p$ perfectoid Tate ring $R$ with a ring of integral elements $R^+$, the restriction map $$Y(R,R^{+}) \to Y(R,R^\circ)$$ is injective (resp. bijective). A morphism $Y\to X$ of diamonds is  \emph{separated} (resp.  \emph{partially proper}) if it is quasiseparated and for all choices of a characteristic $p$ perfectoid Tate ring $R$ with a ring of integral elements $R^+$, the natural map $$Y(R,R^+) \to Y(R,R^\circ) \times_{X(R,R^\circ)} X(R,R^+)$$ is injective (resp. bijective).

\item A diamond $Y$ is  \emph{spatial} if it is quasicompact and quasiseparated, and if the subsets $|U| \subset |Y|$, for $U \subset Y$ running over arbitrary qc open subdiamonds of $Y$, give a neighborhood basis of $|Y|$.

\item A diamond is  \emph{locally spatial} if it admits a covering by spatial open subdiamonds.

\end{enumerate}
\end{defn}

We will mostly restrict our attention to locally spatial diamonds, which have several favorable properties:

\begin{prop}\label{locallyspatialbasics}
\begin{enumerate}
\item If\/ $Y$ is (locally) spatial, then $|Y|$ is a (locally) spectral topological space.  
\item A locally spatial diamond $Y$ is qc (resp. qs) if and only if $|Y|$ is qc (resp. qs).
\item If $U \to X \leftarrow V$ is a diagram of (locally) spatial diamonds, then $U \times_X V$ is (locally) spatial.
\item If\/ $Y \to X$ is any map of locally spatial diamonds, the associated map of locally spectral spaces $|Y| \to |X|$ is spectral and generalizing.\footnote{Recall that a morphism of topological spaces $f:Y\to X$ is generalizing if for any point $y \in Y$ and any generalization $x'$ of $f(y)$, there is some generalization $y'$ of $y$ such that $f(y')=x'$, cf. \cite[Def. 1.3.3]{Huberbook}.} Moreover, $|Y| \to |X|$ is a quotient map if $Y\to X$ is surjective.
\item If $X$ is an analytic adic space, then $X^\lozenge$ is locally spatial, and $X^\lozenge$ is spatial exactly when $X$ is qcqs.
\end{enumerate}
\end{prop}

Finally, in some of the proofs in \S3.2 we will make use of Scholze's results on ``canonical compactifications'' of diamonds.  Let us briefly explain what we need from this theory; the reader may wish to skip this discussion.

Fix a diamond $S$.  In \cite[\S18]{Sch}, Scholze defines a compactification functor $Y \mapsto \overline{Y}^{/S}$ from separated $S$-diamonds to separated $S$-diamonds satisfying a number of useful compatibilities:

1) The map $\overline{Y}^{/S} \to S$ is partially proper, and the diamond $ \overline{Y}^{/S}$ comes with a functorial injection $Y \hookrightarrow \overline{Y}^{/S}$ which is an isomorphism if $Y$ is partially proper.

2) The functor $Y \mapsto \overline{Y}^{/S}$ commutes with all limits.

3) If $Y \to Z$ is injective (resp. surjective, resp. quasi-pro-\'etale), then $\overline{Y}^{/S} \to \overline{Z}^{/S}$ is injective (resp. surjective, resp. quasi-pro-\'etale).

4) If $S=\Spd(A,A^+)$ and $Y=\Spd(B,B^+)$ are affinoid perfectoid, then $\overline{Y}^{/S} = \Spd(B,B')$ where $B'$ denotes the completed integral closure of $A^+ + B^{\circ \circ}$ in $B$.

5) If $W$ is a profinite set, then $\overline{Y \times \underline{W}}^{/S} \cong \overline{Y}^{/S}\times \underline{W}$.

We also need the following compatibility of compactifications with torsors.

\begin{prop}\label{Gtorsorcompactification}
Let $G$ be a profinite group, and let $X \to Y$ be a $\underline{G}$-torsor in separated $S$-diamonds.  Then $\overline{X}^{/S} \to \overline{Y}^{/S}$ is a $\underline{G}$-torsor.
\end{prop}

\begin{proof}
By property 3) above, the map $\overline{X}^{/S} \to \overline{Y}^{/S}$ is surjective and quasi-pro-\'etale.  Functoriality of the compactification implies that the $\underline{G}$-action on $X$ extends canonically to a $\underline{G}$-action on $\overline{X}^{/S}$.  It now suffices to check that $$ \underline{G} \times \overline{X}^{/S}  \to \overline{X}^{/S} \times_{\overline{Y}^{/S}} \overline{X}^{/S}$$ is an isomorphism. By property 2), the target of this map is canonically isomorphic to $$ \overline{X \times_{Y} X}^{/S} \cong \overline{\underline{G} \times X}^{/S},$$ which in turn is canonically isomorphic to $\underline{G} \times  \overline{X}^{/S}$ by property 5).  This implies the desired result.
\end{proof}

\subsection{Dimension theory} \label{dim_formulas}

In this section we work out some basic dimension theory for diamonds.  Fix a complete nonarchimedean field $K$ of residue characteristic $p$; until further notice, all diamonds are assumed to live over $\mathrm{Spd}(K,K^\circ)$.

\begin{defn} For $X$ a locally spatial diamond, we say $x \in |X|$ is a  \emph{rank one point} if any one of the following equivalent conditions are satisfied:
\begin{enumerate}
\item $x$ has no proper generalizations inside $|X|$;
\item there exists a perfectoid field $L$ and a quasi-pro-\'etale map $\Spd(L,L^\circ) \to X$ with topological image $x$;
\item we have $$x= \bigcap_{\substack{x \in U \subset |X| \\ \,U\,\mathrm{qc\,open}}} U$$as subsets of $|X|$.
\end{enumerate}
(The equivalence of these conditions follows easily from Proposition \ref{locallyspatialbasics}.)
Note that any rank one point $x$ has the structure of a diamond in its own right: more precisely, there is a quasicompact injection from a spatial diamond $x^\lozenge \to X$ with $|x^\lozenge| = x$, which is unique up to unique isomorphism. We will sometimes confuse $x$ and $x^\lozenge$ in what follows. We also note that for any map $Y \to X$ of locally spatial diamonds, the induced map $|Y| \to |X|$ carries rank one points onto rank one points.

\end{defn}
\begin{lemma}\label{dimlemma1}
Let $X$ be a partially proper and locally spatial diamond over $S = \mathrm{Spd}(K,K^\circ)$. Choose a rank one point $x \in |X|$, and choose a perfectoid field $L$ together with a quasi-pro-\'etale map $\mathrm{Spd}(L,L^\circ) \to X$ with topological image $x$.  Then $\dim \overline{\{x \}}=\mathrm{tr.deg}(l/k)$, where $l$ and $k$ denote the residue fields of $L^\circ$ and $K^\circ$, respectively.
\end{lemma}
\begin{proof}
Replacing $L$ with $\widehat{\overline{L}}$ doesn't change the transcendence degree of $l$ over $k$, so without of loss of generality we may assume that $L$ is algebraically closed.  By the proof of \cite[Prop. 21.9]{Sch}, this gives rise to a presentation   $$x^\lozenge \simeq \Spd(L,L^{\circ})/\underline{G},$$ where $G$ is a profinite group acting continuously and faithfully on $L$ by $K$-linear automorphisms.  Let $L^{\mathrm{min}}$ be the completed integral closure of $K^\circ + L^{\circ \circ}$ in $L$. Note that $\Spd(L,L^{\mathrm{min}})$ is simply the canonical compactification $\overline{\Spd(L,L^\circ)}^{/S}$.  Note also that $|\Spd(L,L^{\mathrm{min}})|$ is homeomorphic to the Zariski-Riemann space of $l/k'$, where $k'$ denotes the algebraic closure of $k$ in $l$; in particular, $|\Spd(L,L^{\mathrm{min}})|$ is spectral of dimension $\mathrm{tr.deg}(l/k') = \mathrm{tr.deg}(l/k)$.

Since $x^{\lozenge} \to X$ is an injection of separated $S$-diamonds, with $X$ partially proper and $x^\lozenge$ quasicompact, this extends to an injection $\overline{x^\lozenge}^{/S} \to X$, and the injective map $|\overline{x^\lozenge}^{/S}| \to |X|$ is a homeomorphism onto its (closed) image. We now claim that $\overline{\{ x \} }$ and $|\overline{x^\lozenge}^{/S}|$ are homeomorphic as subsets of $|X|$: this follows from observing that both subsets are closed and generalizing subsets of $|X|$ each containing $x$ as their unique rank one point. Since $\Spd(L,L^\circ) \to x^{\lozenge}$ is a $\underline{G}$-torsor, the map $\Spd(L,L^{\mathrm{min}}) \to \overline{x^\lozenge}^{/S}$ is a $\underline{G}$-torsor by Proposition \ref{Gtorsorcompactification}. Therefore we have natural homeomorphisms $$ \overline{\{ x \} } \cong |\overline{x^\lozenge}^{/S}| \cong |\Spd(L,L^{\mathrm{min}})|/G.$$ To complete the proof, it suffices to show the equality $$\dim |\Spd(L,L^{\mathrm{min}})| = \dim |\Spd(L,L^{\mathrm{min}})|/G.$$ This is a special case of the next lemma.
\end{proof}

\begin{lemma}\label{groupquotientspectralspaces} Let $X$ be a spectral space, and let $G$ be a profinite group acting continuously on $X$. Then $X/G$ is spectral, the map $q: X \to X/G$ is spectral and generalizing, and $\dim X = \dim X/G$.
\end{lemma}

\begin{proof} Without loss of generality, $G$ acts faithfully on $X$.  One easily checks that the image $R$ of the natural map $G \times X \overset{(xg,x)}{\to} X \times X$ is pro-constructible and defines an equivalence relation on $X$, such that the maps $s,t: R \to X$ are quasicompact, generalizing and open.  Spectrality of $X/R=X/G$ then follows from \cite[Lemmas 2.9-2.10]{Sch}; moreover, these same lemmas imply that $X \to X/G$ is spectral and generalizing. 

For the equality of dimensions, note that since $q$ is generalizing, we may lift any finite chain in $X/G$ to a chain of the same length in $X$, and thus $\dim X \geq \dim X/G$. It remains to prove the opposite inequality. Since $q$ sends chains in $X$ to chains in $X/G$, it's enough to show that the for any finite chain $C \subset X$, the length of $C$ coincides with the length of $q(C)$. Suppose otherwise; then we can find two distinct points $x \prec y$ in $X$ with $q(x)=q(y)$, or equivalently with $xG = yG$. Let\footnote{Thanks to Christian Johansson for suggesting the following argument.} $\lambda : |X| \to \mathbf{N} \cup \{ \infty \}$ be the function sending a point $z \in X$ to the maximal length of any chain of generalizations of $z$. Clearly $\lambda(x) \geq \lambda(y)+1$.  On the other hand, the $G$-action on $X$ preserves the relation of one point generalizing another, so $\lambda$ is constant on $G$-orbits. This is a contradiction.
\end{proof}

\begin{lemma}\label{additivefiberdim}
Let $f: X \to Y$ be a map of partially proper and locally spatial diamonds over $\mathrm{Spd}(K,K^\circ)$. Choose any rank one point $x \in |X|$ with image $y=f(x)$, and let $X_{y}=X\times_{Y} y^\lozenge$ denote the fiber of $X$ over $y$ (so in particular $x \in |X_y|$).  Then $$\dim \overline{\{x \}}^{X} = \dim \overline{\{y \}}^{Y} + \dim \overline{\{x \}}^{X_y},$$
where $\overline{\{s\}}^{S}$ denotes the closure of $\{s\}$ in the topological space of $S$.
\end{lemma}

When $X$ and $Y$ are analytic adic spaces, this is exactly Proposition 4.2.21 of \cite{CS}. The proof given below is essentially an adaptation of the argument in loc. cit., although we need to incorporate a rather elaborate pro-\'etale unscrewing.

\begin{proof}
Choose surjective quasi-pro-\'etale maps $\Spd(L,L^\circ) \to X$, $\Spd(M,M^\circ) \to Y$ with topological images $x$ and $y$, respectively; without loss of generality we may assume that $L$ and $M$ are algebraically closed, with residue fields $l$ and $m$, respectively.  Applying Lemma \ref{dimlemma1} twice, we get equalities $\dim \overline{\{x \}}^{X} = \mathrm{tr.deg}(l/k)$ and $\dim \overline{\{y \}}^{Y} = \mathrm{tr.deg}(m/k)$. By the additivity of transcendence degrees, it suffices to prove that there is \emph{some} continuous $K$-algebra map $M \to L$ such that the induced extension $m \to l$ of residue fields satisfies $\mathrm{tr.deg}(l/m) = \dim \overline{\{x \}}^{X_y}$. 

Let $G$ and $H$ be the profinite groups such that $\Spd(L,L^\circ) \to x^\lozenge$ (resp. $\Spd(M,M^\circ) \to y^\lozenge$) is a $G$-torsor (resp. an $H$-torsor) as before.  Arguing as in the proof of Lemma \ref{dimlemma1}, we have a canonical homeomorphism $$ \overline{\{x \}}^{X_y} \cong |\overline{x^\lozenge}^{/y^\lozenge}|. $$ Base change along the map $\Spd(M,M^\circ) \to y^\lozenge$ identifies $\overline{x^\lozenge \times_{y^\lozenge} \Spd(M,M^\circ)}^{/\Spd(M,M^\circ)}$ with $\overline{x^\lozenge}^{/y^\lozenge} \times_{y^\lozenge} \Spd(M,M^\circ)$; in particular, $$\overline{x^\lozenge \times_{y^\lozenge} \Spd(M,M^\circ)}^{/\Spd(M,M^\circ)}$$ is naturally an $\underline{H}$-torsor over $\overline{x^\lozenge}^{/y^\lozenge}$, so \begin{eqnarray*} \dim \overline{\{x \}}^{X_y}  & = & \dim |\overline{x^\lozenge}^{/y^\lozenge}| \\
& = & \dim |\overline{x^\lozenge \times_{y^\lozenge} \Spd(M,M^\circ)}^{/\Spd(M,M^\circ)}|,
\end{eqnarray*}
where the second equality follows from Lemma \ref{groupquotientspectralspaces}. Furthermore, the 
$\underline{G}$-torsor $\Spd(L,L^\circ) \to x^\lozenge$ induces a $\underline{G}$-torsor structure on the map $$ \overline{\Spd(L,L^\circ) \times_{y^\lozenge} \Spd(M,M^\circ)}^{/\Spd(M,M^\circ)} \to \overline{x^\lozenge \times_{y^\lozenge} \Spd(M,M^\circ)}^{/\Spd(M,M^\circ)},$$ so we are reduced to calculating the topological dimension of the source of this map.  Since $\Spd(L,L^\circ) \times_{y^\lozenge} \Spd(M,M^\circ) \to \Spd(L,L^\circ)$ is an $\underline{H}$-torsor (and in particular a pro-finite \'etale map) with target a geometric point, there is an isomorphism $$\Spd(L,L^\circ) \times_{y^\lozenge} \Spd(M,M^\circ) \simeq \Spd(L,L^\circ) \times \underline{W}$$for some profinite set $W$. In particular, any connected component $V$ of $$ \Spd(L,L^\circ) \times_{y^\lozenge} \Spd(M,M^\circ)$$ is isomorphic to $\Spd(L,L^\circ)$; choosing such a component induces a continuous map $M \to L$, and $\dim |\overline{V}^{/\Spd(M,M^\circ)}| = \mathrm{tr.deg}(l/m)$ by our previous discussions.  We now conclude by observing that the compactification functor identifies any given connected component of $|\overline{\Spd(L,L^\circ) \times_{y^\lozenge} \Spd(M,M^\circ)}^{/\Spd(M,M^\circ)}|$ homeomorphically with $ |\overline{V}^{/\Spd(M,M^\circ)}|$ for some uniquely determined connected component $V \subset \Spd(L,L^\circ) \times_{y^\lozenge} \Spd(M,M^\circ)$, so the result now follows by varying $V$.
\end{proof}

The next result conceptualizes the proof of Proposition 4.2.23 of \cite{CS}.

\begin{lemma}\label{imagedimlemma}
Let $f: X \to Y$ be a map of partially proper and locally spatial diamonds over $\mathrm{Spd}(K,K^\circ)$. Suppose there is some integer $d \geq 0$ such that for every rank one point $y \in \mathrm{im}(|X| \to |Y|)$, the fiber $X_y = X \times_{Y} y^{\lozenge}$ is of dimension $d$.  Then $$\dim \mathrm{im}(|X| \to |Y|) = \dim |X| - d.$$
\end{lemma}
\begin{remark}Let $f:X \to Y$ be as in the lemma. If $f$ is quasicompact, then the subspace and quotient topologies on $\mathrm{im}(|X| \to |Y|)$ coincide, and $\mathrm{im}(|X| \to |Y|)$ is locally spectral and pro-constructible inside $|Y|$. For a general $f$, all of these outcomes can fail.  It is true, however, that $\mathrm{im}(|X| \to |Y|)$ is always closed under generalization and specialization inside $|Y|$. In any case, we define $\dim \mathrm{im}(|X| \to |Y|)$ as the maximal length of any chain of specializations of any rank one point of $|Y|$ contained in the image of $|X|$. 
\end{remark}
\begin{proof} Choose a rank one point $y \in \mathrm{im}(|X| \to |Y|)$, and choose some rank one point $x \in X_y$ lifting $y$ such that $\dim \overline{\{x \}}^{X_y} = d$. Then $$ \dim \overline{\{y \}}^{Y} = \dim \overline{\{x \}}^{X}-\dim \overline{\{x \}}^{X_y} = \dim \overline{\{x \}}^{X}-d \leq \dim X -d$$ by Lemma \ref{additivefiberdim}. Taking the supremum over all $y$, we deduce the inequality $$\dim \mathrm{im}(|X| \to |Y|) \leq \dim |X| - d.$$

For the opposite inequality, it clearly suffices to prove that $$ \dim X \leq \dim \mathrm{im}(|X| \to |Y|)+d.$$ Let $|X|^{\mathrm{rk}\,1}$ denote the set of rank one points in $|X|$. Then  \begin{eqnarray*}
\dim X & = & \sup_{x \in |X|^{\mathrm{rk}\,1}} \dim \overline{\{x\}}^{X}\\
 & = & \sup_{x \in |X|^{\mathrm{rk}\,1}} (\dim \overline{\{f(x)\}}^{Y}+ \dim \overline{\{x\}}^{X_{f(x)}}) \\
 & \leq & \sup_{x \in |X|^{\mathrm{rk}\,1}} \dim \overline{\{f(x)\}}^{Y}+ \sup_{x \in |X|^{\mathrm{rk}\,1}} \dim \overline{\{x\}}^{X_{f(x)}}\\
 & = & \dim \mathrm{im}(|X| \to |Y|) + d.\end{eqnarray*}Here the first line is simply the definition of $\dim X$, the second line follows from Lemma \ref{additivefiberdim}, and the final line follows from the relevant definitions together with our assumption on the fibers of $f$.
\end{proof}

\subsection{Diamonds and moduli of bundle maps}\label{bundlemaps}

In this section, we define and study diamonds parametrizing maps between vector bundles on $\xcal$ with various specified properties. We note in passing that every diamond considered here is a diamond in the sense of the original definition proposed by Scholze in 2014, i.e. they each admit a surjective representable pro-\'etale morphism from a perfectoid space. 

Throughout this section, we fix a finite extension $E/\mathbf{Q}_p$ with residue field $\mathbf{F}_q$ and an algebraically closed perfectoid field $F / \mathbf{F}_q$, and we let $\Cal{X}=\Cal{X}_{E,F}$ denote the adic Fargues-Fontaine curve over $\Spa\,F$.  For any perfectoid space $S / \Spa\,F$ we get a relative curve $\mathcal{X}_S=\mathcal{X}_{E,S}$ together with a natural map $\Cal{X}_S \to \Cal{X}$ (cf. \cite[Ch. 7-8]{KL15} for a thorough discussion of relative Fargues-Fontaine curves). If $\Cal{E}$ is any vector bundle on $\Cal{X}$, we denote the vector bundle obtained via pullback along $\Cal{X}_S \to \Cal{X}$ by $\Cal{E}_S$.

\begin{defn} Fix vector bundles $\ecal,\fcal$ on $\xcal$.  Consider the following sheaves of sets on $\mathrm{Perf}_{/\mathrm{Spa}\, F}$.
\begin{enumerate}
\item Let $\mathcal{H}^{0}(\mathcal{E}):\mathrm{Perf}_{/\mathrm{Spa}\, F}\to\mathrm{Sets}$
be the functor sending $f:S\to\mathrm{Spa}\, F$ to the set $H^{0}(\mathcal{X}_{S},\mathcal{E}_{S})$.
\item Let $\mathcal{H}\mathrm{om}(\mathcal{E},\mathcal{F})$ be the functor
sending $f:S\to\mathrm{Spa}\, F$ to the set of $\mathcal{O}_{\mathcal{X}_{S}}$-module
maps $m:\mathcal{E}_{S}\to\mathcal{F}_{S}$. Note that $\mathcal{H}\mathrm{om}(\mathcal{E},\mathcal{F})\cong\mathcal{H}^{0}(\mathcal{E}^{\vee}\otimes\mathcal{F})$. 

\item Let $\mathcal{S}\mathrm{urj}(\mathcal{E},\mathcal{F})\subset\mathcal{H}\mathrm{om}(\mathcal{E},\mathcal{F})$
be the subfunctor of $\Hom(\ecal,\fcal)$ whose $S$-points parametrize surjective $\mathcal{O}_{\mathcal{X}_{S}}$-module
maps. 

\item Let $\mathcal{I}\mathrm{nj}(\mathcal{E},\mathcal{F})\subset\mathcal{H}\mathrm{om}(\mathcal{E},\mathcal{F})$
be the subfunctor of $\Hom(\ecal,\fcal)$ whose $S$-points parametrize ``fiberwise-injective'' $\mathcal{O}_{\mathcal{X}_{S}}$-module
maps. Precisely, this is the functor parametrizing $\mathcal{O}_{\mathcal{X}_{S}}$-module
maps $m:\mathcal{E}_{S}\to\mathcal{F}_{S}$ such that for every geometric
point $\overline{x}=\mathrm{Spa}(C,C^{+})\to S$, the pullback of $m:\mathcal{E}_{S}\to\mathcal{F}_{S}$
along the induced map $\mathcal{X}_{\overline{x}} \to\mathcal{X}_{S}$ gives an injective
$\mathcal{O}_{\mathcal{X}_{\overline{x}}}$-module map. %
\footnote{The condition defining $\mathcal{I}\mathrm{nj}$ is much stronger than
the condition that $m:\mathcal{E}_{S}\to\mathcal{F}_{S}$ be
injective; note that the association sending $S$ to the set of injective
$m$s isn't even a presheaf. Note also that there is a natural transformation
$\mathcal{S}\mathrm{urj}(\mathcal{F}^{\vee},\mathcal{E}^{\vee})\to\mathcal{I}\mathrm{nj}(\mathcal{E},\mathcal{F})$;
this turns out to be an open immersion, although it typically isn't
an isomorphism.%
}

\item Let $\mathcal{H}^{0}(\mathcal{E})^{\times}:=\mathcal{I}\mathrm{nj}(\mathcal{O},\mathcal{E})\subset\mathcal{H}^{0}(\mathcal{E})$
be the functor parametrizing sections of $\mathcal{E}$ which are
not identically zero on any fiber of $\mathcal{X}_{S}\dashrightarrow S$.

\item Let $\acal\mathrm{ut}(\ecal)$ be the functor sending $f:S \to \Spa\,F$ to the group of $\ocal_{\xcal_S}$-module automorphisms of $\ecal_S$.
\end{enumerate}

Our first order of business is to show that these functors are all locally spatial and partially proper diamonds. We remark that the functors $\mathcal{H}^{0}$ and $\mathcal{H}\mathrm{om}$ are Banach-Colmez spaces \cite{Col,leBras}, and in particular can be given some meaningful geometric structure without appealing to the theory of diamonds. However, the other four functors defined above are not Banach-Colmez spaces (indeed, they are not even valued in $\mathbf{Q}_p$-vector spaces), and the theory of diamonds seems essential to our analysis of them.

Let $$\tilde{\mathbf{D}}^n = \Spa(F^\circ [[T_1^{1/p^{\infty}},\dots,T_n^{1/p^{\infty}}]],F^\circ [[T_1^{1/p^{\infty}},\dots,T_n^{1/p^{\infty}}]]) \times_{\Spa(F^\circ,F^\circ)} \Spa(F,F^\circ)$$ denote the $n$-dimensional open perfectoid unit disk over $F$.

\end{defn}

\begin{prop}\label{h0nice} 1) The functor $\Cal{H}^0(\Cal{E})$ is a diamond: if $\Cal{E}$ has only positive slopes, then we can find an isomorphism $$\Cal{H}^0(\Cal{E}) \simeq \tilde{\mathbf{D}}^d/\underline{A}$$ where $d=\deg \ecal$ and where $A \simeq \mathbf{Q}_{p}^m$ is an abelian locally profinite group acting freely on $\tilde{\mathbf{D}}^d$; in general, there is a natural isomorphism $\Cal{H}^0(\Cal{E}) \simeq \underline{\mathbf{Q}_{p}^n} \times \Cal{H}^0(\Cal{E}^{>0})$ for some $n\geq 0$.  

2) The diamond $\hcal^0(\ecal)$ is partially proper and locally spatial, and equidimensional of dimension $\deg \ecal^{\geq 0}$. Furthermore, any nonempty open subfunctor of $\Cal{H}^0(\Cal{E})$ has an $F$-point.
\end{prop}

We point out that the dimension of $\hcal^0(\ecal)$ computed above coincides with the ``principal dimension'' of this object in the language of Banach-Colmez spaces \cite{Col}.

\begin{remark} The partial properness of $\Cal{H}^0(\Cal{E})$ is a formal consequence of the fact that the category of vector bundles on a relative curve $\Cal{X}_{\Spa(R,R^+)}$ is canonically independent of the choice of $R^{+}$, cf. \cite[Theorem 8.7.7]{KL15}.  In fact, every functor defined above is partially proper, for the same reason.
\end{remark}

\begin{proof}
Writing $\ecal \simeq \ecal^{>0} \oplus \ocal^n \oplus \ecal'$ where $\ecal'$ has only negative slopes, the isomorphism $\mathcal{H}^{0}(\mathcal{E}_{1}\oplus\mathcal{E}_{2})\cong\mathcal{H}^{0}(\mathcal{E}_{1})\times_{\mathrm{Spd}\, F}\mathcal{H}^{0}(\mathcal{E}_{2})$ together with the identification $\hcal^0(\ocal) \simeq \underline{E}$ reduce us from the general case to the case of positive slopes. Writing $\ecal \simeq \oplus_i \ocal(\lambda_i)$ and observing that $$ \tilde{\mathbf{D}}^{d_1}/\underline{A_1} \times_{\Spd\,F} \tilde{\mathbf{D}}^{d_2}/\underline{A_2} \cong \tilde{\mathbf{D}}^{d_1 +d_2}/\underline{A_1 \times A_2},$$ we reduce further to the case where $\mathcal{E}=\mathcal{O}(\lambda)$ for some $\lambda = d/h \in\mathbf{Q}_{>0}$. Let $E'/E$ be the unramified extension of degree $h$, so we have a natural finite \'etale map $r: \xcal_{E',S} \to \xcal_S$ such that $r_{\ast} \ocal(d) \simeq \ocal(\lambda)$ functorially in $S$. Choose a short exact sequence $$0 \to \ocal^{d-1} \to \ocal(1)^d \to \ocal(d) \to 0$$of vector bundles on $\xcal_{E',F}$. Pushing forward along $r_\ast$ and applying $\hcal^0(-)$ gives rise to a short exact sequence of abelian group sheaves on $\mathrm{Perf}_{/\Spa\,F}$, giving an identification $$\hcal^0(\ocal(d/h)) \simeq \hcal^0(\ocal(1/h)^d)/\hcal^0(\ocal^{h(d-1)}).$$ It's easy to see that $\hcal^0(\ocal^{h(d-1)}) \simeq \underline{E^{h(d-1)}}$ as abelian group sheaves on $\mathrm{Perf}_{/\Spa\,F}$. Moreover, there is an isomorphism $\hcal^0(\ocal(1/h)) \simeq \tilde{\mathbf{D}}^1$: quite generally, for any given integers $0<d \leq h$, one can exhibit an isomorphism $$\hcal^0(\ocal(d/h)) \simeq \tilde{\mathbf{D}}^{d} $$ by identifying $\hcal^0(\ocal(d/h))$ with the universal cover of an isoclinic $\pi$-divisible $\ocal_E$-module of height $h$ and dimension $d$ (cf. \cite{FF08, SW13}). Putting things together, 1) follows.

For 2), assume for the moment that $\hcal^0(\ecal)$ is locally spatial.  Fix a presentation as in 1).  After choosing an open profinite subgroup $A_0 \subset A$, the map $$ \tilde{\mathbf{D}}^d \to \tilde{\mathbf{D}}^d/\underline{A}$$ factors as $$\tilde{\mathbf{D}}^d \to \tilde{\mathbf{D}}^d/\underline{A_0} \to \tilde{\mathbf{D}}^d/\underline{A}.$$ By \cite[Prop. 4.3.2]{Wei} the first arrow here is pro-finite \'etale surjective, and the second arrow is separated, \'etale and surjective by \cite[Lemma 10.13]{Sch}.  The space $\tilde{\mathbf{D}}^d$ is equidimensional of dimension $d$, so $\tilde{\mathbf{D}}^d/\underline{A_0}$ is equidimensional of dimension $d$ by Lemma \ref{groupquotientspectralspaces}.  We are now reduced to checking that if $f: X \to Y$ is a surjective \'etale map of locally spatial diamonds, then $\dim X = \dim Y$.  Since $f$ is generalizing, the inequality $\dim X \geq \dim Y$ is clear.  For the opposite inequality, one easily reduces to the case where $X \to Y$ is a surjective finite \'etale map with $Y$ connected, by \cite[Lemma 11.31]{Sch}, which we then leave as an exercise for the interested reader.\footnote{Sketch: By a standard argument, one can dominate $X$ by a surjective finite \'etale map $X' \to X$ such that $X' \to Y$ is finite \'etale and Galois for some finite group $G$; but then $\dim X \leq \dim X' = \dim Y$, where the latter equality follows from Proposition \ref{groupquotientspectralspaces}.}

We still need to check that $\hcal^0(\ecal)$ is locally spatial. Fix a presentation as in 1), and choose a quasicompact open subgroup $U_0 \subset \tilde{\mathbf{D}}^d$. Set $\underline{A_0}= U_0 \cap \underline{A}$, so this is a profinite group sheaf acting freely on $U_0$. Then $V=U_0/\underline{A_0}$ is an open  subdiamond of $\hcal^0(\ecal)$, and it is  spatial by the subsequent lemma. Since  the $p^{-n}$-dilates of $V$ cover all of $\hcal^0(\ecal)$, the latter is locally spatial, as desired.

Finally, the statement on $F$-points follows from the explicit presentation in 1) together with the (easy) analogous statement for $\tilde{\mathbf{D}}^n$.
\end{proof}
In the previous proof, we made use of the following lemma.
\begin{lemma}Let $X$ be a spatial diamond with a free $\underline{G}$-action for some profinite group $G$. Then $X/\underline{G}$ is a spatial diamond.
\end{lemma}
\begin{proof}Immediate upon combining Lemma 10.13 and Proposition 11.24 from \cite{Sch}.
\end{proof}

\begin{prop}\label{homnice}The functor $\Hom(\mathcal{E}, \mathcal{F})$ is a locally spatial (and partially proper) diamond over $\Spd\,F$, equidimensional of dimension $\deg(\Cal{E}^\vee \otimes \Cal{F})^{\geq 0}$.
\end{prop}
\begin{proof} Immediate from the identification $\Hom(\mathcal{E}, \mathcal{F}) \cong \Cal{H}^0(\Cal{E}^\vee \otimes \Cal{F})$.
\end{proof}

\begin{prop}\label{surjinjnice}The functors $\Surj(\mathcal{E}, \mathcal{F})$ and $\Inj(\mathcal{E}, \mathcal{F})$ are both open subfunctors of $\Hom(\mathcal{E}, \mathcal{F})$. In particular, $\Surj(\mathcal{E}, \mathcal{F})$ and $\Inj(\mathcal{E}, \mathcal{F})$ are locally spatial (and partially proper) diamonds over $\Spd\,F$, each of which is either empty or equidimensional of dimension $\deg(\Cal{E}^\vee \otimes \Cal{F})^{\geq 0}$.
\end{prop}
\begin{proof}(Proof for $\Surj$.) Choose some $T\in\mathrm{Perf}$ together with a surjective quasi-pro-\'etale
morphism $T\to\mathcal{H}\mathrm{om}(\mathcal{E},\mathcal{F})$.
Over $\mathcal{X}_{T}$, we get a ``universal'' $\mathcal{O}_{\mathcal{X}_{T}}$-module
map $m^{\mathrm{univ}}:\mathcal{E}_{T}\to\mathcal{F}_{T}$; let $\mathcal{Q}_{T}$
be the cokernel of $m^{\mathrm{univ}}$. By a standard argument, the
support of $\mathcal{Q}_{T}$ is Zariski-closed in $\mathcal{X}_{T}$,
and we write $Z\subset|\mathcal{X}_{T}|$ for the associated closed
subset.

Next, we observe that the map $|\mathcal{X}_{T}|\to|T|$ is closed.
Indeed, this is a specializing quasicompact spectral map of locally
spectral spaces, so the image of any closed subset is pro-constructible
(by quasicompactness and spectrality) and stable under specialization,
hence closed by {[}Stacks, Tag 0903{]}. In particular, the subset
$V=\mathrm{im}(|\mathcal{X}_{T}|\to|T|)(Z)\subset|T|$ is closed.
We also observe that a geometric point $x:\mathrm{Spd}(C,C^{+})\to\mathcal{H}\mathrm{om}(\mathcal{E},\mathcal{F})$
defines a point of $\mathcal{S}\mathrm{urj}$ (resp. $\mathcal{H}\mathrm{om}\smallsetminus\mathcal{S}\mathrm{urj}$)
if and only if the preimage of $|x|$ in $|T|$ is disjoint from $V$
(resp. contained in $V$). In particular, the open subset $U=|T|\smallsetminus V\subset|T|$
is the preimage of a subset $W\subset|\mathcal{H}\mathrm{om}(\mathcal{E},\mathcal{F})|$;
since $|T|\twoheadrightarrow|\mathcal{H}\mathrm{om}(\mathcal{E},\mathcal{F})|$
is a quotient map, $W$ is open. But now $\mathcal{S}\mathrm{urj}$
can be identified with the open subdiamond of $\mathcal{H}\mathrm{om}(\mathcal{E},\mathcal{F})$
corresponding to the open subset $W$, so we win.

(Proof for $\Inj$.) Set $r=\mathrm{rank}(\mathcal{E})$; by the formula\[
\mathcal{I}\mathrm{nj}(\mathcal{E},\mathcal{F})\cong\mathcal{H}\mathrm{om}(\mathcal{E},\mathcal{F})\times_{\mathcal{H}\mathrm{om}(\wedge^{r}\mathcal{E},\wedge^{r}\mathcal{F})}\mathcal{I}\mathrm{nj}(\wedge^{r}\mathcal{E},\wedge^{r}\mathcal{F}),\]
we reduce to the case where $\mathcal{E}$ is a line bundle. After
twisting, we reduce further to the case $\mathcal{E}=\mathcal{O}$;
in other words, we need to prove that $\mathcal{H}^{0}(\mathcal{F})^{\times}$
is an open subfunctor of $\mathcal{H}^{0}(\mathcal{F})$. Fix an identification
$\mathcal{F}=\oplus_{1\leq i\leq n}\mathcal{O}(\lambda_{i})$, and
(for brevity) set $\mathcal{H}_{i}=\mathcal{H}^{0}(\mathcal{O}(\lambda_{i}))$
and $\mathcal{H}_{i}^{\times}=\mathcal{H}^{0}(\mathcal{O}(\lambda_{i}))^{\times}$.
Under the identification\[
\mathcal{H}^{0}(\mathcal{F})=\mathcal{H}_{1}\times_{\mathrm{Spd}\, F}\cdots\times_{\mathrm{Spd}\, F}\mathcal{H}_{n},\]
it is easy to see that the subfunctor $\mathcal{H}^{0}(\mathcal{F})^{\times}$
on the left-hand side is covered by the union of the subfunctors\[
\mathcal{U}_{i}=\mathcal{H}_{1}\times\cdots\times\mathcal{H}_{i-1}\times\mathcal{H}_{i}^{\times}\times\mathcal{H}_{i+1}\times\cdots\times\mathcal{H}_{n},\;1\leq i\leq n\]
on the right-hand side (here we have omitted the subscripted $\mathrm{Spd}\, F$'s
for brevity). This, finally, reduces us to showing that $\mathcal{H}^{0}(\mathcal{O}(\lambda))^{\times}$
is an open subfunctor of $\mathcal{H}^{0}(\mathcal{O}(\lambda))$,
for any fixed $\lambda \geq 0$. The case $\lambda =0$ is easy and left to the reader.  For $\lambda = d/h > 0$, writing $\ocal(\lambda)$ as the pushforward of $\ocal(d) / \xcal_{E',F}$ as in the proof of Proposition \ref{h0nice} reduces us to the case $h=1$. In this case, the functors $\hcal(\ocal(d))$ and $\hcal(\ocal(d))^\times$ agree (by definition) with the functors denoted by $\mathbf{B}^{\varphi=\pi^d}_{\Spa\,F}$ and $(\mathbf{B}^{\varphi=\pi^d} \smallsetminus \{0\})_{\Spa\,F}$ in \cite[\S2.2]{Far17}, and our claim is exactly the content of \cite[Lemme 2.10]{Far17}.
\end{proof}

\begin{prop}The functor $\mathcal{A}\mathrm{ut}(\ecal)$ is an open and partially proper subdiamond of $\Hom(\ecal,\ecal)$, equidimensional of dimension $\deg(\ecal^\vee \otimes \ecal)^{\geq 0}$.
\end{prop}
\begin{proof}Immediate from the previous proposition together with the identification $\mathcal{A}\mathrm{ut}(\ecal) \cong \Surj(\ecal,\ecal)$, which holds by consideration of rank and degree.
\end{proof}

We now explain how to reduce Step One of the strategy outlined in the introduction to a combinatorial problem. The key result here is Theorem \ref{mainsurjthm1} below.

\begin{defn} For any vector bundles $\ecal, \fcal, \qcal$, composition of bundle maps induces a natural map of diamonds $$\Surj(\ecal,\qcal) \times_{\Spd\,F} \Inj(\qcal,\fcal) \to \Hom(\ecal,\fcal),$$ and we define $$|\Hom(\mathcal{E}, \mathcal{F})_{\mathcal{Q}}| \subset |\Hom(\mathcal{E}, \mathcal{F})|$$ as the image of the induced map on topological spaces.
\end{defn}

As the notation suggests, $|\Hom(\mathcal{E}, \mathcal{F})_{\mathcal{Q}}|$ is the underlying topological space of a subdiamond $\Hom(\mathcal{E}, \mathcal{F})_{\mathcal{Q}}$ of $\Hom(\mathcal{E}, \mathcal{F})$, which (more or less) parametrizes bundle maps $\ecal \to \fcal$ with image isomorphic to $\qcal$ at all geometric points. However, the diamond $\Hom(\mathcal{E}, \mathcal{F})_{\mathcal{Q}}$ is a little obscure (in particular, it's not entirely clear to us that it's locally spatial), so we avoid making any explicit use of this object.  Fortunately, this doesn't cause any real complication.

\begin{prop}\label{homqpp}For any $\ecal, \fcal, \qcal$, $|\Hom(\mathcal{E}, \mathcal{F})_{\mathcal{Q}}|$ is stable under generalization and specialization inside $|\Hom(\mathcal{E}, \mathcal{F})|$.
\end{prop}
\begin{proof} The subset in question is defined as the image of $|Y| \to |X|$, for some map of partially proper and locally spatial diamonds $Y \to X$. Quite generally, if $Y \to X$ is any map of locally spatial diamonds, the image of $|Y| \to |X|$ is stable under generalization by Proposition \ref{locallyspatialbasics}.iii.  Orthogonally, one immediately checks from the valuative criterion that the image of $|Y| \to |X|$ is stable under specialization for $Y\to X$ any partially proper map of diamonds.  Since any map between diamonds partially proper over a fixed base is automatically partially proper, the claim follows.
\end{proof}

We now have the following crucial lemma.

\begin{lemma}\label{maindimlemma1}For any $\ecal, \fcal, \qcal$ as above, the subset $|\Hom(\mathcal{E}, \mathcal{F})_{\mathcal{Q}}| \subset |\Hom(\mathcal{E}, \mathcal{F})|$ is either empty, or of dimension $$\deg(\Cal{E}^{\vee} \otimes \Cal{Q})^{\geq 0} +  \deg(\Cal{Q}^{\vee} \otimes \Cal{F})^{\geq 0} - \deg (\Cal{Q}^{\vee} \otimes \Cal{Q})^{ \geq 0}.$$
\end{lemma}

\begin{proof} Recall that $| \Hom(\ecal,\fcal)_{\qcal}|$ is defined as the image of the map on topological spaces associated with the map of locally spatial diamonds $$f: \Surj(\ecal,\qcal) \times_{\Spd\,F} \Inj(\qcal,\fcal) \overset{(s,i)\mapsto i \circ s}{\longrightarrow} \Hom(\ecal,\fcal).$$If either of the functors on the left is empty, there is nothing to prove.  If both functors on the left are nonempty, then \begin{eqnarray*} \dim \Surj(\ecal,\qcal) \times_{\Spd\,F} \Inj(\qcal,\fcal) & = & \dim \Surj(\ecal,\qcal) + \dim \Inj(\qcal,\fcal)\\ & = & \deg(\ecal^\vee \otimes \qcal)^{\geq 0} + \deg(\qcal^\vee \otimes \fcal)^{\geq 0}\end{eqnarray*}by Propositions \ref{h0nice} and \ref{surjinjnice}. Moreover, the fiber of $f$ over any rank one point $y$ is an $\mathcal{A}\mathrm{ut}(\qcal)\times_{\Spd\,F}y$-torsor, and hence can be identified with $\mathcal{A}\mathrm{ut}(\qcal) \times_{\Spd\,F} \overline{y}$ after taking a pro-finite \'etale covering of $y$ by some geometric point $\overline{y} = \Spd\,C$. Since $$\dim \mathcal{A}\mathrm{ut}(\qcal) = \deg(\qcal^\vee \otimes \qcal)^{\geq 0},$$the result now follows from Lemma \ref{imagedimlemma}.
\end{proof}

\begin{thm}\label{mainsurjthm1} Suppose $\ecal$ and $\fcal$ are vector bundles admitting a nonzero map $\ecal\to\fcal$, such that for any $\qcal \subsetneq \fcal$ which also occurs as a quotient of $\ecal$ we have a strict inequality $$\deg(\Cal{E}^{\vee} \otimes \Cal{Q})^{\geq 0} +  \deg(\Cal{Q}^{\vee} \otimes \Cal{F})^{\geq 0} < \deg(\Cal{E}^{\vee} \otimes \Cal{F})^{\geq 0} +  \deg (\Cal{Q}^{\vee} \otimes \Cal{Q})^{ \geq 0}.$$ Then there exists a surjective bundle map $\ecal \to \fcal$.
\end{thm}
\begin{proof}
We need to show that $\Surj(\ecal,\fcal)(F)$ is nonempty. By Proposition \ref{surjinjnice} above, $\Surj(\ecal,\fcal)$ is an open subfunctor of $\Hom(\ecal,\fcal)$, and hence admits $F$-points if it is nonempty (by Proposition \ref{h0nice}). Since $\Surj(\ecal,\fcal)$ is nonempty if and only if  $|\Surj(\ecal,\fcal)|$ is nonempty, we can argue on topological spaces. 

By the definitions above, we obtain a decomposition $$ X \overset{\mathrm{def}}{=} |\Hom(\ecal,\fcal)| \smallsetminus |\Surj(\ecal,\fcal)| =  \coprod_{\qcal \in S} |\Hom(\ecal,\fcal)_{\qcal}|,$$ where $S$ denotes the set of isomorphism classes of subbundles $\qcal \subsetneq \fcal$ which also occur as quotients of $\ecal$. By Proposition \ref{homqpp}, any chain in $X$ is entirely contained in $|\Hom(\ecal,\fcal)_{\qcal}|$ for some fixed $\qcal$, so then $$ \dim X = \sup_{\qcal \in S} \dim |\Hom(\ecal,\fcal)_{\qcal}|
\leq \sup_{\qcal \in S}\left(\deg(\Cal{E}^{\vee} \otimes \Cal{Q})^{\geq 0} +  \deg(\Cal{Q}^{\vee} \otimes \Cal{F})^{\geq 0} - \deg (\Cal{Q}^{\vee} \otimes \Cal{Q})^{ \geq 0}\right)$$by the previous lemma.  By the assumptions of the theorem, we then deduce a strict inequality $\dim X < \deg(\ecal^\vee \otimes \fcal)^{\geq 0}$.  If $|\Surj(\ecal,\fcal)|$ were empty, however, $X = |\Hom(\ecal,\fcal)|$ would have dimension $\deg(\ecal^\vee \otimes \fcal)^{\geq 0}$ by Proposition \ref{homnice}, and this is a contradiction.  Thus $|\Surj(\ecal,\fcal)|$ is nonempty as desired.
\end{proof}

We now turn to Step Two of the strategy outlined in the introduction, which we will also reduce to a combinatorial problem. Here, the relevant diamonds are defined as follows.

\begin{defn} For some fixed vector bundles $\Cal{E},\Cal{F},\Cal{K}$ such that $\rank{\kcal}+\rank{\fcal}=\rank{\ecal}$ and $\deg{\kcal}+\deg{\fcal}=\deg{\ecal}$, define $\Surj(\mathcal{E}, \mathcal{F})^{\mathcal{K}}$ as the subfunctor of $\Surj(\mathcal{E}, \mathcal{F})$ whose $S$-points (for a given $S \to \Spd\,F$) parametrize surjective bundle maps $q: \ecal_S \to \fcal_S$ with the property that $\ker q$ is isomorphic to $\kcal$ after pullback along any geometric point $\overline{x} \to S$.
\end{defn}

Equivalently, the map $S \to \Surj(\mathcal{E}, \mathcal{F})$ corresponding to a surjection $q: \ecal_S \to \fcal_S$ defines a point of $\Surj(\mathcal{E}, \mathcal{F})^{\mathcal{K}}$ exactly when the HN polygon of $\ker q$, regarded as a function on $|S|$, is constant and coincides with the HN polygon of $\kcal$.

\begin{prop}\label{surjfixedkernelnice} The functor $\Surj(\mathcal{E}, \mathcal{F})^{\mathcal{K}}$ is a locally spatial (and partially proper) diamond over $\Spd\,F$.
\end{prop}
\begin{proof} Choose a perfectoid space $T$ together with a surjective and quasi-pro-\'etale map $f: T \to \Surj(\mathcal{E}, \mathcal{F})$. Let $\vcal$ be the bundle on $\xcal_T$ defined by the kernel of the ``universal'' quotient map $q^{univ} : \ecal_T \to \fcal_T$. Set $n = \rank(\ecal)-\rank(\fcal)$; for a given HN polygon $P$ of width $n$, let $|T|^{\geq P}$ (resp. $|T|^{\leq P}$) be the set of points $x$ where $\HN(\vcal_x) \geq P$ (resp. where $\HN(\vcal_x) \leq P$). By [KL, Theorem 7.4.5], the locus $|T|^{\geq P}$ (resp. $|T|^{\leq P}$) is closed (resp. open) in $|T|$; moreover, these subsets are stable under generalization.  Since $|f| : |T| \to |\Surj(\mathcal{E}, \mathcal{F})|$ is a quotient map, one checks directly that $|T|^{\geq P}$ (resp. $|T|^{\leq P}$) is the preimage of a generalizing and closed (resp. open) subset $|\Surj(\mathcal{E}, \mathcal{F})|^{\geq P}$ (resp. $|\Surj(\mathcal{E}, \mathcal{F})|^{\leq P}$) of $|\Surj(\mathcal{E}, \mathcal{F})|$. In particular, $$ |\Surj(\mathcal{E}, \mathcal{F})|^{\geq P} \cap |\Surj(\mathcal{E}, \mathcal{F})|^{\leq P}$$is a locally closed generalizing subset of $|\Surj(\mathcal{E}, \mathcal{F})|$, and hence corresponds to a locally spatial subdiamond $$\Surj(\mathcal{E}, \mathcal{F})^{P} \subset \Surj(\mathcal{E}, \mathcal{F})$$ by the proof of \cite[Prop. 11.20]{Sch}. Taking $P = \HN(\kcal)$, the functor $\Surj(\mathcal{E}, \mathcal{F})^{P}$ identifies with $\Surj(\mathcal{E}, \mathcal{F})^{\mathcal{K}}$, and we conclude.
\end{proof}

Again, we have a clean dimension formula in certain cases.

\begin{lemma}\label{maindimlemma2}Let $\ecal,\fcal$, and $\kcal$ be as above, and suppose moreover that $\fcal$ is semistable.  Then $\Surj(\mathcal{E}, \mathcal{F})^{\mathcal{K}}$ is either empty, or equidimensional of dimension $$ \deg(\ecal \otimes \kcal^\vee)^{\geq 0} - \deg(\kcal \otimes \kcal^\vee)^{\geq 0}.$$
\end{lemma}

\begin{proof} Recall that $\Surj(\ecal,\fcal)^{\kcal}$ parametrizes surjective bundle maps $q: \ecal \to \fcal$ with kernel isomorphic to $\kcal$ at all geometric points. Let $$\Surj(\ecal,\fcal)^{\kcal,\heart} \to \Surj(\ecal,\fcal)^{\kcal}$$ be the $\mathcal{A}\mathrm{ut}(\kcal)$-torsor parametrizing isomorphisms $\ker q \simeq \kcal$. Arguing as in the proof of Lemma \ref{maindimlemma1}, one easily checks that \begin{eqnarray*} \dim \Surj(\ecal,\fcal)^{\kcal,\heart} & = & \dim \Surj(\ecal,\fcal)^{\kcal} + \dim \mathcal{A}\mathrm{ut}(\kcal)\\ & = & \dim \Surj(\ecal,\fcal)^{\kcal} + \deg(\kcal^\vee \otimes \kcal)^{\geq 0}\end{eqnarray*}

  Next, we observe that  $$\Surj(\ecal,\fcal)^{\kcal,\heart}$$ can also be described as the functor whose $T$-points parametrize isomorphism classes of short exact sequences $$ 0 \to \kcal_T \to \ecal_T \to \fcal_T \to 0$$ on $\xcal_T$. Dualizing the set of such short exact sequences induces a canonical isomorphism $$\Surj(\ecal,\fcal)^{\kcal,\heart} \cong \Surj(\ecal^\vee,\kcal^\vee)^{\fcal^\vee,\heart}.$$
  
Returning to the matter at hand, we now assume that $\fcal$ is semistable.  Then the proof of Proposition \ref{surjfixedkernelnice} shows that $\Surj(\ecal^\vee,\kcal^\vee)^{\fcal^\vee}=\Surj(\ecal^\vee,\kcal^\vee)^{\le\HN(\fcal^\vee)}$ is an \emph{open} subfunctor of $\Surj(\ecal^\vee,\kcal^\vee)$.  Moreover, the dimension formula above degenerates to the equality $$ \dim \Surj(\ecal^\vee,\kcal^\vee)^{\fcal^\vee,\heart} = \dim \Surj(\ecal^\vee,\kcal^\vee)^{\fcal^\vee}.$$ Putting this together with Proposition \ref{surjinjnice}, we deduce that $$\dim \Surj(\ecal^\vee,\kcal^\vee)^{\fcal^\vee,\heart} = \deg( \ecal \otimes \kcal^\vee)^{\geq 0}$$ for semistable $\fcal$. But then \begin{eqnarray*} \dim \Surj(\ecal,\fcal)^{\kcal} & = & \dim \Surj(\ecal,\fcal)^{\kcal,\heart} - \deg(\kcal^\vee \otimes \kcal)^{\geq 0}\\
& = & \dim \Surj(\ecal^\vee,\kcal^\vee)^{\fcal^\vee,\heart} - \deg(\kcal^\vee \otimes \kcal)^{\geq 0}\\ & = & \deg(\ecal \otimes \kcal^\vee)^{\geq 0}-\deg(\kcal^\vee \otimes \kcal)^{\geq 0},\end{eqnarray*}as desired.
\end{proof}

\subsection{Summary of the strategy}\label{strategy}

The goal of \S \ref{step_1}-\S \ref{step_2} is to establish Theorem \ref{main_thm_3}, which we reproduce below for the reader's convenience: 
\begin{thm*} Let $\Cal{F}_1$ and $\Cal{F}_2$ be semistable vector bundles on $X$ with $\mu(\fcal_1) < \mu(\fcal_2)$. Then  any vector bundle $\Cal{E}$ such that
\[ \HN(\Cal{E}) \leq \HN(\Cal{F}_1 \oplus \Cal{F}_2)
\]
can be realized as an extension
\[
0 \rightarrow \Cal{F}_1 \rightarrow \Cal{E} \rightarrow \Cal{F}_2 \rightarrow 0.
\]
\end{thm*}

Again, the strategy of our proof is as follows. 
We note that it suffices to show the theorem when the slope of $\Cal{F}_2$ is 
	strictly greater than the maximal slope of $\Cal{E}$.
Otherwise we would have $\Cal{F}_2=\Cal O(s)^m$ and $\Cal{E}=\Cal{E'}\oplus \Cal O(s)^n$
	for some $\Cal{E}'$ with maximal slope less than $s$ and $n\leq m$.
Solving the extension problem for 
\[
	0\to \Cal F_1 \to \Cal{E'} \to \Cal O(s)^{m-n} \to 0
\]
	and then direct summing with $0\to 0 \to \Cal{O}(s)^n \to \Cal{O}(s)^n\to 0$
    gives the full theorem.
    
It then clearly suffices to show that the assumptions of Theorem \ref{main_thm_3} 
	plus the additional assumption of strict inequality of slopes 
    imply the following two assertions.
\begin{enumerate}
\item There exists a surjection $\Cal{E} \rightarrow \Cal{F}_2 \rightarrow 0$. 
\item If $\Cal{E}$ admits a surjection $\Cal{E} \rightarrow \Cal{F}_2 \rightarrow 0$, then it admits such a surjection with kernel isomorphic to $\Cal{F}_1$. 
\end{enumerate}

The bulk of \S \ref{step_1} is devoted to showing that the inequality
\[
\deg(\Cal{E}^{\vee} \otimes \Cal{Q})^{\geq 0} +  \deg(\Cal{Q}^{\vee} \otimes \Cal{F}_2)^{\geq 0} < \deg(\Cal{E}^{\vee} \otimes \Cal{F}_2)^{\geq 0} +  \deg (\Cal{Q}^{\vee} \otimes \Cal{Q})^{ \geq 0}.
\]
holds for any proper sub-bundle $\Cal{Q} \subsetneq \Cal{F}_2$ which also occurs as a quotient of $\Cal{E}$. This collection of inequalities implies (1) by  Theorem \ref{mainsurjthm1}.

Assuming $\ecal$ satisfies (1), Lemma \ref{maindimlemma2} and Proposition \ref{surjinjnice} reduce (2) to the verification of the inequality
\[
\deg(\ecal \otimes \kcal^\vee)^{\geq 0} - \deg(\kcal \otimes \kcal^\vee)^{\geq 0} < 	\deg (\Cal{E}^\vee \otimes \Cal{F}_2)^{\geq 0}
\]
for a certain finite collection of $\kcal$'s (this implication will be fleshed out in \S \ref{completion}). The proof of these inequalities forms the subject of \S \ref{step_2}; the precise statement is given in Theorem \ref{degree inequality for kernel}.

    \section{Step one: surjections of vector bundles}\label{step_1}
  
\subsection{The goal}
	
    This section is devoted to establishing Step (1), as outlined in \S \ref{strategy}, wherein we show that the obvious necessary numerical conditions required for a bundle $\Cal{E}$ to surject onto a given semistable bundle $\Cal{F}$ are also sufficient. The precise statement is as follows.
    
\begin{thm}\label{step-1}
Let $\Cal{F}$ be a semistable vector bundle on $X$. Let $\Cal{E}$ be a vector bundle on $X$ such that 
\begin{itemize}
\item $\rank \Cal{E} > \rank \Cal{F}$, and 
\item the maximal slope of $\Cal{E}$ is less than or equal to $\mu(\fcal)$.
\end{itemize}
Then $\Cal{E}$ admits a surjection onto $\Cal{F}$.
\end{thm}

Note that, by an argument similar to the one in \ref{strategy}, it suffices to treat the case where the maximal slope of $\ecal$ is strictly less than $\mu(\fcal)$. 
\begin{comment}
Indeed, if the maximal slope of $\ecal$ coincides with $\mu(\fcal)$, then the induced map $q: \ecal^{\geq \mu(\fcal)} \to \fcal$ is a map of semistable bundles of the same slope (See Definition \ref{slope_leq_part} for the notation $\Cal{E}^{\geq \mu}$), hence has image and kernel both of slope $\mu(\fcal)$. Moreover, the Classification Theorem \ref{classification} implies that $\Ima q$ is a direct summand of $\fcal$, and $\ecal^{\geq \mu(\fcal)} \to \Ima q$ admits a section. Therefore, we restrict our attention to the complement of $\Cal{E}^{\geq \mu}$ in $\Cal{E}$ and the complement of $\Ima q$ in $\fcal$ without affecting the inequality $\rank \ecal > \rank \fcal$, thus establishing the claim.
\end{comment}

By this reduction and Theorem \ref{mainsurjthm1}, we are now reduced to proving the inequality
\begin{equation}\label{step1eq}
\deg(\Cal{E}^{\vee} \otimes \Cal{Q})^{\geq 0} +  \deg(\Cal{Q}^{\vee} \otimes \Cal{F})^{\geq 0} < \deg(\Cal{E}^{\vee} \otimes \Cal{F})^{\geq 0} +  \deg (\Cal{Q}^{\vee} \otimes \Cal{Q})^{ \geq 0}
\end{equation}
for any $\qcal$ which is both a quotient of $\Cal{E}$ and a proper subbundle of $\Cal{F}$, under the assumption that the maximal slope of $\Cal{E}$ is \emph{strictly} less than $\mu(\fcal)$. (The strict inequality is false without this latter strictness assumption, which is why we made the initial reduction to this case.) 

\subsection{Some lemmas} 
%To prove the inequality \eqref{step1eq} we obviously need to use the fact that $\Cal{Q}$ is a quotient of $\Cal{E}$, so we make some observations towards translating this property into numerical properties of slopes.

We begin by translating the assumption that $\Cal{Q}$ is a quotient of $\Cal{E}$ into numerical constraints on the slopes of their HN polygons.

\begin{lemma}\label{quotient-condition-lemma}
Suppose $\Cal{Q}$ is a quotient of $\Cal{E}$.
Then for every slope $\mu$, the bundle $\Cal{Q}^{\leq \mu}$ is a quotient of $\Cal{E}^{\leq \mu}$. (Dually, if $\Cal{K}$ is a subbundle of $\Cal{E}$, for every slope $\mu$, the bundle $\Cal{K}^{\geq \mu}$ is a subbundle of $\Cal{E}^{\geq \mu}$.) 
\end{lemma}

\begin{proof}
 The composite quotient map $\Cal{E} \twoheadrightarrow \Cal{Q}\surj\qcal^{\le\mu}$ necessarily factors through $\Cal{E}^{\leq \mu}$, since $\Cal{E}^{>\mu}$ is a direct sum of stable bundles with slopes greater than $\mu$.
\end{proof}

Since we will prove \eqref{step1eq} by interpreting the terms on both sides as areas of polygons, it is convenient to have a ``geometric'' reformulation of the preceding lemma.

\begin{cor}\label{slopes-condition-corollary}
Let $\Cal{Q}$ be a quotient of $\Cal{E}$. Translate the HN polygons for $\Cal{Q}$ and $\Cal{E}$ in the plane so that both their right endpoints lie at the origin.
%Consider aligning the right endpoints of the HN polygons for $\Cal{Q}$ and $\Cal{E}$ at the origin. 
Then for every integer $i$ from $1$ to $\rank \Cal{Q}$, the slope of the part of the HN polygon for $\Cal{Q}$ lying in the strip $[-i,-i+1]\times \RR$ is greater than or equal to the slope of the part of the HN polygon for $\Cal{E}$ lying in this strip. (See Figure \ref{slopes_cond_fig}.)

In particular, the HN polygon for $\Cal{Q}$ always lies below the HN polygon for $\Cal{E}$ (when aligning right endpoints).
\end{cor}
\begin{center}
\begin{figure}[h]
\begin{tikzpicture}	
		\coordinate (q0) at (0,0);
		\coordinate (q1) at (4,1);
		\coordinate (q2) at (8.5,-0.5);
		\coordinate (end) at (11,-3);
		%coordinates of breakpoints of HN(E)
		
		\coordinate (p0) at (2, -2);
		\coordinate (p1) at (4.5, -1);
		\coordinate (p2) at (6.5, -1);
		\coordinate (p3) at (9.5, -2);
		%coordinates of breakpoints of HN(Q)
				
		\draw[step=1cm,thick] (q0) --  (q1) -- (q2) -- (end);
		\draw[step=1cm,thick] (p0) --  (p1) -- (p2) -- (p3) -- (end);
		
		\draw [fill] (q0) circle [radius=0.05];		
		\draw [fill] (q1) circle [radius=0.05];		
		\draw [fill] (q2) circle [radius=0.05];		
		\draw [fill] (end) circle [radius=0.05];
		
		\draw [fill] (p0) circle [radius=0.05];		
		\draw [fill] (p1) circle [radius=0.05];		
		\draw [fill] (p2) circle [radius=0.05];		
		\draw [fill] (p3) circle [radius=0.05];
		
		\draw[step=1cm,dotted] (6, -3.4) -- (6, 1.4);
       	\draw[step=1cm,dotted] (5.5, -3.4) -- (5.5, 1.4);

		\node at (5.3,-3.8) {\scriptsize $-i$};
		\node at (6.2,-3.8) {\scriptsize $-i+1$};
		
		\path (q0) ++(-0.8, -0.1) node {$\HN({\mathcal{E}})$};
		\path (p0) ++(-0.8, -0.1) node {$\HN({\mathcal{Q}})$};

\end{tikzpicture}
\caption{Illustration of Corollary \ref{slopes-condition-corollary}.}\label{slopes_cond_fig}
\end{figure}
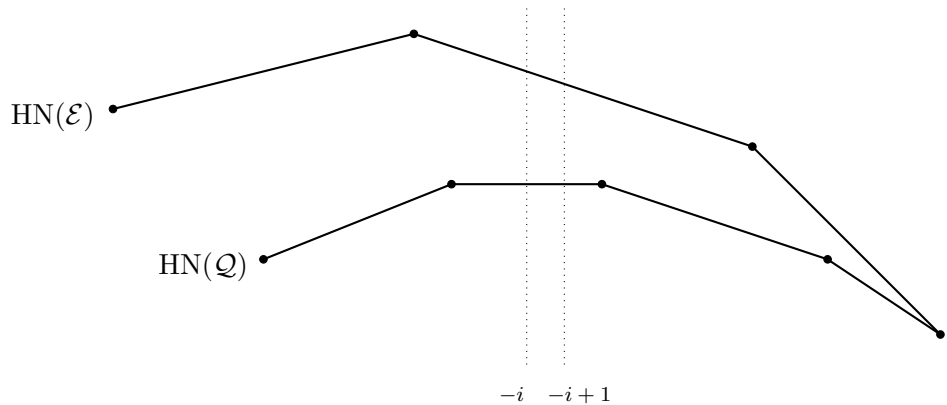
\end{center}
\begin{proof}
Suppose for the sake of contradiction that on some interval $[-i,-i+1]$, the slope of $\HN(\Cal{Q})$ is less than that of $\HN(\Cal{E})$. Let $\mu$ be the slope of the HN polygon for $\Cal{Q}$ on this strip. Then we would have $\rank \Cal{E}^{\leq \mu}< i\leq \rank \Cal{Q}^{\leq \mu}$, contradicting Lemma \ref{quotient-condition-lemma}.
\end{proof}

\begin{remark}
In this section, we will usually find it more convenient to align HN polygons with their \emph{right endpoint} at the origin. (In future sections, it will usually be more convenient to align HN polygons with their \emph{left endpoint} at the origin.) 
\end{remark}

\subsection{Proof of Theorem \ref{step-1}} We are now ready to begin the proof of Theorem \ref{step-1}, the main result of this section. The idea is to interpret each term of the inequality as the area of a certain polygon, and to transform the polygons by area-preserving operations (specifically translations and shears) until they can be easily compared. 

\begin{proof}[Proof of Theorem \ref{step-1}]
Recall that we are trying to show
\begin{equation}\label{ineq_for_step1}
	\deg(\Cal{E}^\vee\otimes \Cal Q)^{\geq 0}
    	-\deg(\Cal{Q}^\vee \otimes \Cal{Q})^{\geq 0}
    < \deg(\Cal{E}^\vee\otimes \Cal{F})^{\geq 0}
    	-\deg(\Cal{Q}^\vee \otimes \Cal{F})^{\geq 0}
\end{equation}
for any $\Cal{Q}$ which can occur as a quotient of $\Cal{E}$ and a proper subbundle of $\Cal{F}$.
%, with equality if and only if $\Cal{Q} \cong \Cal{F}$. 
As discussed before, we assume here that the maximum slope of $\Cal{E}$ is strictly less than the slope of $\Cal{F}$, which is permissible by the initial reduction discussed in \S4.1.

Write $\HNvec(\mathcal{E})=(\vec{\lambda}_j)_{1 \leq j \leq s}$ and $\HNvec(\mathcal{Q})=(\vec{\mu}_i)_{1 \leq i \leq s'}$
(under the notation introduced in Definition \ref{HN vector notation}). We will write $\lambda_j$ for the slope of the vector $\vec{\lambda}_j$, and so on. By Lemma \ref{diamond lemma for nonnegative degree}, we see that 
\begin{align*}
	\deg(\Cal{E}^\vee\otimes \Cal Q)^{\geq 0}
    	-\deg(\Cal{Q}^\vee \otimes \Cal{Q})^{\geq 0}
    &=\sum_{\vec{\lambda}_j \preceq \vec{\mu}_i} \vec{\lambda}_j \times \vec{\mu}_i
    	-\sum_{\vec{\mu}_k\preceq \vec{\mu}_i} \vec{\mu}_k\times \vec{\mu}_i\\
    &=\sum_i \left(\sum_{\vec{\lambda}_j\preceq \vec{\mu}_i}\vec{\lambda}_j-\sum_{\vec{\mu}_k\preceq\vec{\mu}_i} \vec{\mu}_k\right)\times \vec{\mu}_i \\
    &=\sum_i \vec{a}_i\times\vec{\mu}_i
\end{align*}
where
\[
\vec{a}_i=\sum_{\vec{\lambda}_j\preceq \vec{\mu}_i}\vec{\lambda}_j-\sum_{\vec{\mu}_k\preceq\vec{\mu}_i} \vec{\mu}_k.
\]
Then the left side of \eqref{ineq_for_step1} is equal to two times the shaded area in Figure \ref{ineq1_LHS} below.

\begin{center}
\begin{figure}[h]
\begin{tikzpicture}	
		\coordinate (q0) at (0,0);
		\coordinate (q1) at (4,1);
		\coordinate (q2) at (8.5,-0.5);
		\coordinate (end) at (11,-3);
		%coordinates of breakpoints of HN(E)
		
		\coordinate (p0) at (2, -2);
		\coordinate (p1) at (4.5, -1);
		\coordinate (p2) at (6.5, -1);
		\coordinate (p3) at (9.5, -2);
		%coordinates of breakpoints of HN(Q)
		
        \draw[purple, fill=purple!20] (q2)--(end)--(p3)--(q2);
		\draw[purple, fill=purple!20] (q1)--(p1)--(p2)--(p3)--(q1);
		\draw[purple, fill=purple!20] (q0)--(p0)--(p1)--(q0);
		
		\draw[step=1cm,thick] (q2)--(p3);
		\draw[step=1cm,thick] (q1)--node[above] {$\vec{a}_3$} (p3);
		\draw[step=1cm,thick] (q1)--node[below] {$\vec{a}_2$} (p2);
		\draw[step=1cm,thick] (q1)--(p1);
		\draw[step=1cm,thick] (q0)--node[above] {$\vec{a}_1$} (p1);
		\draw[step=1cm,thick] (q0)--(p0);

		\draw[step=1cm,thick] (q0) --node[above] {$\vec{\lambda}_1$}  (q1);
		\draw[step=1cm,thick] (q1) --node[above] {$\vec{\lambda}_2$}  (q2);
		\draw[step=1cm,thick] (q2) --node[right] {$\vec{\lambda}_3=\vec{a}_4$}  (end);
		
		\draw[step=1cm,thick] (p0) --node[below] {$\vec{\mu}_1$}  (p1);
		\draw[step=1cm,thick] (p1) --node[below] {$\vec{\mu}_2$}  (p2);
		\draw[step=1cm,thick] (p2) --node[below] {$\vec{\mu}_3$}  (p3);
		\draw[step=1cm,thick] (p3) --node[below] {$\vec{\mu}_4$}  (end);
		
		\draw [fill] (q0) circle [radius=0.05];		
		\draw [fill] (q1) circle [radius=0.05];		
		\draw [fill] (q2) circle [radius=0.05];		
		\draw [fill] (end) circle [radius=0.05];
		
		\draw [fill] (p0) circle [radius=0.05];		
		\draw [fill] (p1) circle [radius=0.05];		
		\draw [fill] (p2) circle [radius=0.05];		
		\draw [fill] (p3) circle [radius=0.05];
		
		\path (q0) ++(-0.8, -0.1) node {$\HN({\mathcal{E}})$};
		\path (p0) ++(-0.8, -0.1) node {$\HN({\mathcal{Q}})$};
		
\end{tikzpicture}
\caption{Area interpretation of the left side of \eqref{ineq_for_step1} in the case $s=3, s'=4$, with slope vectors $\vec{\mu}_1 \succeq \vec{\lambda}_1 \succ \vec{\mu}_2 \succ \vec{\mu}_3 \succeq \vec{\lambda}_2 \succ \vec{\mu}_4 \succeq \vec{\lambda}_3$.}
\label{ineq1_LHS}
\end{figure}
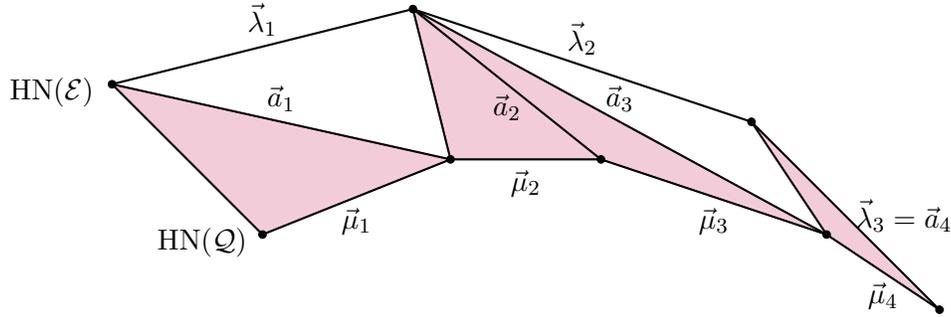
\end{center}
Note that for a fixed $i$, the left endpoint of $\sum_{\vec{\lambda}_j\leq \vec{\mu}_i} \vec{\lambda}_j$ (whose $x$-coordinate is $-\rank \Cal{E}^{\leq \vec{\mu}_j}$)  lies to the left of the left end point of $\sum_{\vec{\mu}_k\leq \vec{\mu}_i} \vec{\mu}_k$ by Lemma \ref{quotient-condition-lemma}.\\

Now we analyze the right side of \eqref{ineq_for_step1}. Let $\HN(\Cal{F})=(\vec{\gamma})$. Note that the slope $\gamma$ of the semistable bundle $\Cal{F}$ must be at least $\mu_i$ for all $i$, and \emph{strictly} greater than $\lambda_j$ for all $j$. Therefore, by similar manipulations to those above, we have
%the right side of \eqref{ineq_for_step1} can be expressed as 
\begin{equation}\label{5.2RHS_area}
\text{RHS of \eqref{ineq_for_step1}} = 	\left(\sum_j \vec{\lambda}_j -\sum_i \vec{\mu}_i\right)\times \vec{\gamma}.
\end{equation}
We note that the vector $\sum_j \vec{\lambda}_j -\sum_i \vec{\mu}_i$ has positive $x$-coordinate because $\rank \Cal{E} > \rank \Cal{Q}$, and has slope $\leq \lambda_1 \le \gamma$ because $\HN(\mathcal{Q})$ lies below $\HN(\mathcal{E})$ by Corollary \ref{slopes-condition-corollary}.

If we align the \emph{left} endpoints of the HN polygons for $\Cal{F}$ and $\Cal{E}$, we then see that the right side of \eqref{ineq_for_step1} is twice the area of the shaded triangle in Figure \ref{ineq1_RHS} below. 

%\lynnelle{I feel like the label $HN(\Cal{F})$ should appear in this diagram, although I am not certain where along the line it should go}

\begin{center}
\begin{figure}[h]
\begin{tikzpicture}	
		\coordinate (q0) at (0,0);
		\coordinate (q1) at (4,1);
		\coordinate (q2) at (8.5,-0.5);
		\coordinate (end) at (11,-3);
		%coordinates of breakpoints of HN(E)
		
		\coordinate (p0) at (2, -2);
		\coordinate (p1) at (4.5, -1);
		\coordinate (p2) at (6.5, -1);
		\coordinate (p3) at (9.5, -2);
		%coordinates of breakpoints of HN(Q)
		
		\coordinate (r) at (10,5);
        %coordinates of breakpoints of HN(F)
        		
		\draw[purple, fill=purple!20] (q0)--(r)--(p0)--(q0);

		\draw[step=1cm,thick] (q0) --node[below] {$\vec{\lambda}_1$}  (q1);
		\draw[step=1cm,thick] (q1) --node[above] {$\vec{\lambda}_2$}  (q2);
		\draw[step=1cm,thick] (q2) --node[above] {$\vec{\lambda}_3$}  (end);
		
		\draw[step=1cm,thick] (p0) --node[below] {$\vec{\mu}_1$}  (p1);
		\draw[step=1cm,thick] (p1) --node[below] {$\vec{\mu}_2$}  (p2);
		\draw[step=1cm,thick] (p2) --node[below] {$\vec{\mu}_3$}  (p3);
		\draw[step=1cm,thick] (p3) --node[below] {$\vec{\mu}_4$}  (end);
		
		\draw[step=1cm,thick] (q0)--node[above]{$\vec{\gamma}$} (r);
        
		\draw[step=1cm,thick] (q0)--(p0);
		\draw[step=1cm,thick] (r)--(p0);
		
		\draw [fill] (q0) circle [radius=0.05];		
		\draw [fill] (q1) circle [radius=0.05];		
		\draw [fill] (q2) circle [radius=0.05];		
		\draw [fill] (end) circle [radius=0.05];
		
		\draw [fill] (p0) circle [radius=0.05];		
		\draw [fill] (p1) circle [radius=0.05];		
		\draw [fill] (p2) circle [radius=0.05];		
		\draw [fill] (p3) circle [radius=0.05];
		
		\draw [fill] (r) circle [radius=0.05];
		
		\path (q0) ++(-0.8, -0.1) node {$\HN({\mathcal{E}})$};
		\path (p0) ++(-0.8, -0.1) node {$\HN({\mathcal{Q}})$};
		\path (r) ++(-1, +.3) node {$\HN(\Cal{F})$};
		
\end{tikzpicture}
\caption{Area interpretation of the right side of \eqref{ineq_for_step1} in the case $s=3, s'=4$ with slope vectors $\vec{\mu}_1 \succeq \vec{\lambda}_1 \succ \vec{\mu}_2 \succ \vec{\mu}_3 \succeq \vec{\lambda}_2 \succ \vec{\mu}_4 \succeq \vec{\lambda}_3$.}
\label{ineq1_RHS}
\end{figure}
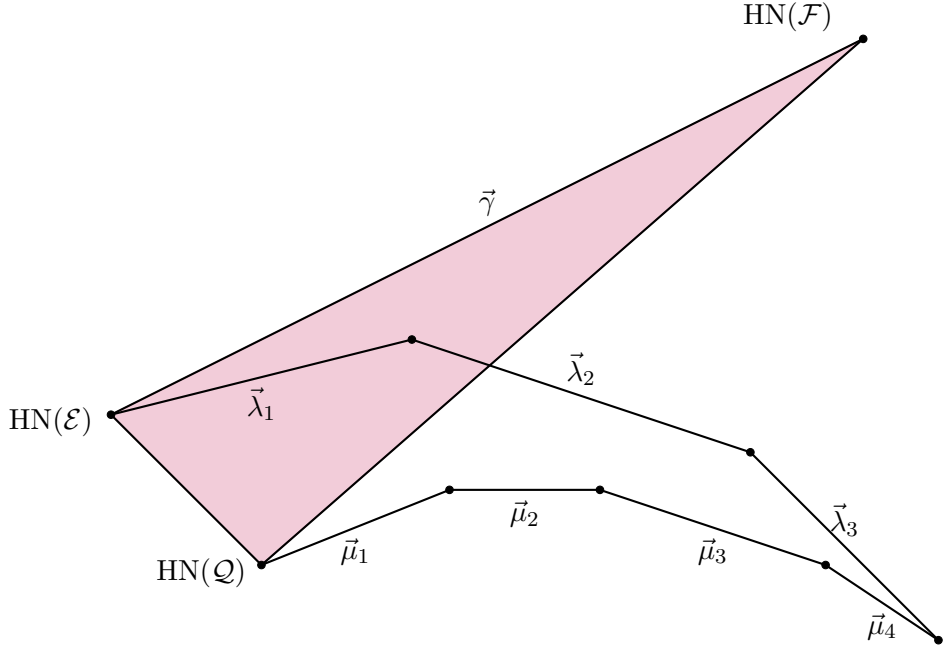
\end{center}

Therefore, we need to show that the shaded area in Figure \ref{ineq1_LHS} (corresponding to the left hand side of \eqref{ineq_for_step1}) is smaller than or equal to the shaded area in Figure \ref{ineq1_RHS} (corresponding to the right hand side of \eqref{ineq_for_step1}), with equality if and only if $\Cal{Q}=\Cal{F}$. The general idea of the proof is to apply shears to each of the triangles in turn so that their bases become horizontal, making it easier to compare the areas.

First, we apply a shear transformation
\[
T = \begin{bmatrix}
	1 & 0 \\
    -\gamma & 1
\end{bmatrix}
\]
to the whole figure, so that the vector $\vec{\gamma}$ becomes horizontal. See Figure \ref{shearT}.

\begin{remark} This shear is nearly equivalent to tensoring by $\mathcal{O}(-\mu(F))$. They are not quite the same because $\Cal{O}(-\mu(F))$ doesn't have rank one. But for the purposes of the inequality \eqref{ineq_for_step1}, tensoring by $\mathcal{O}(-\mu(F))$ would work just as well, as it multiplies both sides by the same factor.
\end{remark}

\begin{remark}\label{shear_observations}
We make some simple observations about the effect of such a shear:
\begin{enumerate}
	\item\label{shear preserves x-coordinates} The $x$-coordinates do not change, so the shear does not change the fact that the horizontal length of the polygon for $\Cal{Q}$ is at most the horizontal length of the polygon for $\Cal{F}$, which is strictly less than the horizontal length of the polygon for $\Cal{E}$. (This fact is simply a reformulation of the inequalities $\rank \Cal{E} > \rank \Cal{F} \geq \rank \Cal{Q}$.) 
    \item\label{shear preserves difference between slopes} The slopes of all line segments decrease by $\gamma$,
    %$s$, so we should take $s$ to be the slope of $\gamma$. 
    so the \emph{difference} between slopes does not change.
    In particular, Corollary \ref{slopes-condition-corollary} still holds, and we may now assume that $\lambda_j<0$  and $\mu_i\leq 0$ for all $i$ and $j$.
    \item\label{shear preserves areas} The determinant of $T$ is $1$, so applying $T$ preserves all areas.
    	Thus, to prove the original inequality, we only need to prove the inequality for the transformed picture.
\end{enumerate}
\end{remark}

\begin{center}
\begin{figure}[h]
\begin{tikzpicture}[scale=0.6]	
		\coordinate (q0) at (0,0);
		\coordinate (q1) at (4,1);
		\coordinate (q2) at (8.5,-0.5);
		\coordinate (end) at (11,-3);
		%coordinates of breakpoints of HN(E)
		
		\coordinate (p0) at (2, -2);
		\coordinate (p1) at (4.5, -1);
		\coordinate (p2) at (6.5, -1);
		\coordinate (p3) at (9.5, -2);
		%coordinates of breakpoints of HN(Q)
		
		\coordinate (r) at (10,5);
        %coordinates of breakpoints of HN(F)
		
        \draw[purple, fill=purple!20] (q2)--(end)--(p3)--(q2);
		\draw[purple, fill=purple!20] (q1)--(p1)--(p2)--(p3)--(q1);
		\draw[purple, fill=purple!20] (q0)--(p0)--(p1)--(q0);
		
		\draw[step=1cm,thick] (q2)--(p3);
		\draw[step=1cm,thick] (q1)--(p3);
		\draw[step=1cm,thick] (q1)--(p2);
		\draw[step=1cm,thick] (q1)--(p1);
		\draw[step=1cm,thick] (q0)--(p1);
		\draw[step=1cm,thick] (q0)--(p0);
		
		\draw[thick] (q0) --node[pos=0.6, below] {\scriptsize $\vec{\lambda}_1$}  (q1);
		\draw[thick] (q1) --node[pos=0.6, above] {\scriptsize $\vec{\lambda}_2$}  (q2);
		\draw[thick] (q2) --node[above] {\scriptsize $\vec{\lambda}_3$}  (end);
		
		\draw[thick] (p0) --node[below] {\scriptsize $\vec{\mu}_1$}  (p1);
		\draw[thick] (p1) --node[below] {\scriptsize $\vec{\mu}_2$}  (p2);
		\draw[thick] (p2) --node[below] {\scriptsize $\vec{\mu}_3$}   (p3);
		\draw[thick] (p3) --node[below] {\scriptsize $\vec{\mu}_4$}   (end);

		\draw[thick] (q0)--node[above]{\scriptsize $\gamma$} (r);
				
		\draw [fill] (q0) circle [radius=0.05];		
		\draw [fill] (q1) circle [radius=0.05];		
		\draw [fill] (q2) circle [radius=0.05];		
		\draw [fill] (end) circle [radius=0.05];
		
		\draw [fill] (p0) circle [radius=0.05];		
		\draw [fill] (p1) circle [radius=0.05];		
		\draw [fill] (p2) circle [radius=0.05];		
		\draw [fill] (p3) circle [radius=0.05];
		
		\draw [fill] (r) circle [radius=0.05];
				
		%\path (q0) ++(-0.8, -0.1) node {$\HN({\mathcal{E}})$};
		%\path (p0) ++(-0.8, -0.1) node {$\HN({\mathcal{Q}})$};
				
\end{tikzpicture} \begin{tikzpicture}[scale=0.4]
        \hspace{-0.2cm}
        \pgfmathsetmacro{\textycoordinate}{6.5}
        %adjust this number to change the position of the arrow
		\draw[->, line width=1pt] (0, \textycoordinate) --node[above]{$T$} (3,\textycoordinate);
		\draw (0,0) circle [radius=0.00];	
\end{tikzpicture} \begin{tikzpicture}[scale=0.6]
		\coordinate (q0) at (0,0);
		\coordinate (q1) at (4,-1);
		\coordinate (q2) at (8.5,-4.75);
		\coordinate (end) at (11,-8.5);
		%coordinates of breakpoints of HN(E) after transformation
        
		\coordinate (p0) at (2, -3);
		\coordinate (p1) at (4.5, -3.25);
		\coordinate (p2) at (6.5, -4.25);
		\coordinate (p3) at (9.5, -6.75);
        %coordinates of breakpoints of HN(Q) after transformation
		
		\coordinate (r) at (10,0);
        %coordinates of breakpoints of HN(F) after transformation
        
        \draw[step=1cm,dashed] (-0.3,-3)--(p0){};
        %helper horizontal line extending HN(Q)
        
        \draw[purple, fill=purple!20] (q2)--(end)--(p3)--(q2);
		\draw[purple, fill=purple!20] (q1)--(p1)--(p2)--(p3)--(q1);
		\draw[purple, fill=purple!20] (q0)--(p0)--(p1)--(q0);
		
		\draw[step=1cm,thick] (q2)--(p3);
		\draw[step=1cm,thick] (q1)--(p3);
		\draw[step=1cm,thick] (q1)--(p2);
		\draw[step=1cm,thick] (q1)--(p1);
		\draw[step=1cm,thick] (q0)--(p1);
		\draw[step=1cm,thick] (q0)--(p0);
		
		\draw[thick] (q0) --node[pos=0.6, below] {\scriptsize $T(\vec{\lambda}_1)$}  (q1);
		\draw[thick] (q1) --node[right] {\scriptsize $T(\vec{\lambda}_2)$}  (q2);
		\draw[thick] (q2) --node[right] {\scriptsize $T(\vec{\lambda}_3)$}  (end);
		
		\draw[thick] (p0) --node[below] {\scriptsize $T(\vec{\mu}_1)$}  (p1);
		\draw[thick] (p1) --node[pos=0.6, below] {\scriptsize $T(\vec{\mu}_2)$}  (p2);
		\draw[thick] (p2) --node[left] {\scriptsize $T(\vec{\mu}_3)$}   (p3);
		\draw[thick] (p3) --node[left] {\scriptsize $T(\vec{\mu}_4)$}   (end);

		\draw[thick] (q0)--node[above]{\scriptsize $T(\gamma)$} (r);
				
		\draw [fill] (q0) circle [radius=0.05];		
		\draw [fill] (q1) circle [radius=0.05];		
		\draw [fill] (q2) circle [radius=0.05];		
		\draw [fill] (end) circle [radius=0.05];
		
		\draw [fill] (p0) circle [radius=0.05];		
		\draw [fill] (p1) circle [radius=0.05];		
		\draw [fill] (p2) circle [radius=0.05];		
		\draw [fill] (p3) circle [radius=0.05];
		
		\draw [fill] (r) circle [radius=0.05];
				
		%\path (q0) ++(-0.8, -0.1) node {$\HN({\mathcal{E}})$};
		%\path (p0) ++(-0.8, -0.1) node {$\HN({\mathcal{Q}})$};

\end{tikzpicture}
\caption{Effect of the shear $T$.}\label{shearT}
\end{figure}
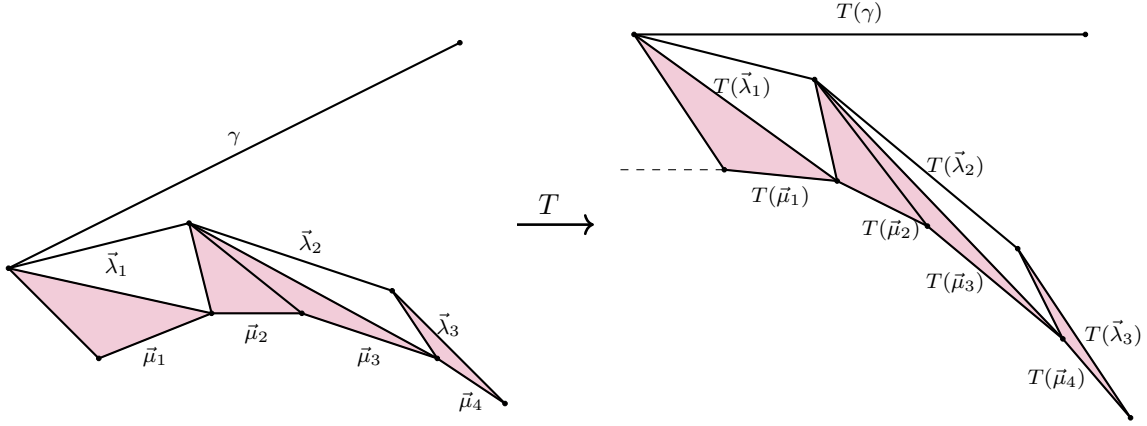
\end{center}

We now want to compute the area of each of the individual small triangles corresponding to terms in the sum for the left hand side.
%so that their base $\vec{\mu}_j$ becomes horizontal. 
Consider Figure \ref{fig_extend} below.

\begin{center}
\begin{figure}[h]
\begin{tikzpicture}[scale=0.7]
		\coordinate (bi+1) at (-1,8.3) {};
		\coordinate (bi) at (0,8) {};
        \coordinate (bix) at (0,0) {};
        \coordinate (bi-1) at (4,5) {};
        \coordinate (bi-2) at (6,3) {};
        \coordinate (ai+1) at (-1,3.3) {};
		\coordinate (ai) at (2,3) {};
		\coordinate (ai-1) at (6,1) {};
        \coordinate (aifoot) at (2, 0) {};
        \coordinate (ai-1foot) at (6,0) {};
		\coordinate (ci) at (0,4) {};
        \coordinate (di) at (2,6.5) {};
        \coordinate (ei) at (0,3.2) {};

        \draw[purple, fill= purple!20] (bi)--(ai)--(ai-1)--(bi);
        
        \draw[step=1cm,thick] (bi)--(bi+1);
        \draw[step=1cm,thick] (bi)--(ai);
		\draw[step=1cm,thick] (bi)--(ai-1);
        \draw[step=1cm,thick] (bi)--(bi-1);
        \draw[step=1cm,thick] (bi-1)--(bi-2);
        \draw[step=1cm,thick] (ai+1)--(ai);
        
        \path (ai) ++(-0.4, -0.5) node {$A_i$};
        \path (ai-1) ++(0.7, -0.2) node {$A_{i-1}$};
        \path (bi) ++(0.2, 0.5) node {$B_i$};
        \path (ci) ++(-0.5, 0.4) node {$C_i$};
        \path (di) ++(0.4, 0.4) node {$D_i$};
        \path (bi) ++(-0.5, -2) node {$h_i$};
        \path (ei) ++(-0.4, -0.4) node {$E_i$};
        \path (aifoot) ++ (2, -0.5) node {$x_i$};
        
        \draw[step=1cm,thick] (-1,0)--(7,0);
        
        \draw[thick] (ai) --node[below] {\small $T(\vec{\mu}_i)$}  (ai-1);
        
        \draw[dashed] (ai)--(ci);
        \draw[dashed] (ai)--(di);
        \draw[dashed] (bi)--(bix);
        \draw[dashed] (ai)--(aifoot);
        \draw[dashed] (ai-1)--(ai-1foot);

		\draw [fill] (bi) circle [radius=0.05];		
		\draw [fill] (bi-1) circle [radius=0.05];		
		\draw [fill] (ai) circle [radius=0.05];		
		\draw [fill] (ai-1) circle [radius=0.05];

\end{tikzpicture}\begin{tikzpicture}[scale=0.4]
        \pgfmathsetmacro{\textycoordinate}{8}
        %adjust this number to change the position of the arrow
		\draw[->, line width=1pt] (0, \textycoordinate) --node[above]{$T_i$} (3,\textycoordinate);
		\draw (0,0) circle [radius=0.00];	
        \hspace{0.2cm}
\end{tikzpicture}\begin{tikzpicture}[scale=0.7]
		\hspace{0.5cm}
		\coordinate (bi+1) at (-1,7.8) {};
		\coordinate (bi) at (0,8) {};
        \coordinate (bix) at (0,0) {};
        \coordinate (bi-1) at (4,7) {};
        \coordinate (bi-2) at (6,6) {};
        \coordinate (ai+1) at (-1,2.8) {};
		\coordinate (ai) at (2,4) {};
		\coordinate (ai-1) at (6,4) {};
        \coordinate (aifoot) at (2, 0) {};
        \coordinate (ai-1foot) at (6,0) {};
		\coordinate (ci) at (0,4) {};
        \coordinate (di) at (2,7.5) {};
        \coordinate (ei) at (0,3.2) {};

        \draw[purple, fill= purple!20] (bi)--(ai)--(ai-1)--(bi);
        
        \draw[step=1cm,thick] (bi)--(bi+1);
        \draw[step=1cm,thick] (bi)--(ai);
		\draw[step=1cm,thick] (bi)--(ai-1);
        \draw[step=1cm,thick] (bi)--(bi-1);
        \draw[step=1cm,thick] (bi-1)--(bi-2);
        \draw[step=1cm,thick] (ai+1)--(ai);
        
        \path (ai) ++(-0.4, -0.6) node {$A_i'$};
        \path (ai-1) ++(0.8, -0.6) node {$A_{i-1}'$};
        \path (bi) ++(0.2, 0.5) node {$B_i'$};
        \path (ci) ++(-0.5, 0.4) node {$C_i'$};
        %\path (di) ++(0.4, 0.4) node {$D_i'$};
        \path (bi) ++(-0.5, -2) node {$h_i$};
        %\path (ei) ++(-0.4, -0.4) node {$E_i$};
        \path (aifoot) ++ (2, -0.5) node {$x_i$};
        
        \draw[step=1cm,thick] (-1,0)--(7,0);
        
        \draw[thick] (ai) --node[below] {\small $T_i (T(\vec{\mu}_i))$}  (ai-1);
        
        \draw[dashed] (ai)--(ci);
        %\draw[dashed] (ai)--(di);
        \draw[dashed] (bi)--(bix);
        \draw[dashed] (ai)--(aifoot);
        \draw[dashed] (ai-1)--(ai-1foot);

		\draw [fill] (bi) circle [radius=0.05];		
		\draw [fill] (bi-1) circle [radius=0.05];		
		\draw [fill] (ai) circle [radius=0.05];		
		\draw [fill] (ai-1) circle [radius=0.05];

\end{tikzpicture}
\caption{}\label{fig_extend}
\end{figure}
\end{center}
Here $A_i$ and $A_{i-1}$ are the two endpoints of $T(\vec{\mu}_i)$, and the vertex $B_i$ is the last vertex of the triangle whose area we are considering.
By Corollary \ref{slopes-condition-corollary}, $B_i$ lies to the left of $A_i$. Extend the line $A_iA_{i-1}$ out leftward so that it meets the vertical line though $B_i$ and call the point of intersection $C_i$.

We claim that $C_i$ must lie below $B_i$.
Indeed, the point $A_i$ lies below the point (call it $D_i$) of $\HN(\Cal{E})$ on the vertical line through $A_i$.
Moreover, by Corollary \ref{slopes-condition-corollary} the slope of the line from $B_i$ to $D_i$ has slope less than or equal to the slope of the line from $C_i$ to $A_i$.
These observations imply that $C_i$ lies below $B_i$ and the distance from $C_i$ to $B_i$ is at least as large as the distance from $A_i$ to $D_i$.

Let $h_i$ denote the distance from $B_i$ to $C_i$, and $x_i$ denote the $x$-coordinate of $\vec{\mu}_i$ (and $T(\vec{\mu}_i)$).	 

\begin{claim}\label{area_claim} We have 
\[
\text{Area}(\Delta A_{i-1} A_i B_i)=\dfrac{1}{2} h_i x_i.
\] 
\end{claim}

\begin{proof}[Proof of Claim \ref{area_claim}] To see this, we apply a shear transformation
\[
T_i = \begin{bmatrix}
	1 & 0 \\
    \gamma-\mu_i & 1
\end{bmatrix}.
\]
%where $s_i$ is the slope of the vector $\vec{\mu}_i$. 
Recall the observations \eqref{shear preserves x-coordinates}, \eqref{shear preserves difference between slopes} and \eqref{shear preserves areas} in Remark \ref{shear_observations} that we made about the effect of such a shear.

Let $A_{i-1}', A_i', B_i', C_i'$ be respectively the image of $A_{i-1}, A_i, B_i, C_i$ under $T_i$ (see Figure \ref{fig_extend}). By observations \eqref{shear preserves x-coordinates} and \eqref{shear preserves difference between slopes} we have the following facts:
\begin{itemize}
\item the line segment $A_{i-1}' A_i'$ is horizontal of length $x_i$,
\item the line segment $B_i' C_i'$ is vertical of length $h_i$.
\end{itemize}
In particular, $\text{Area}(\Delta A_{i-1}' A_i' B_i') = \dfrac{1}{2} h_i x_i$. Now the claim follows since $\text{Area}(\Delta A_{i-1} A_i B_i)= \text{Area}(\Delta A_{i-1}' A_i' B_i')$  by observation \eqref{shear preserves areas}.
\end{proof}

We now let $h$ denote the $y$-coordinate of the vector $\sum_j T(\vec{\lambda}_j)-\sum_i T(\vec{\mu}_i)$. 
Note that $h>0$ since $\HN(\qcal)\leq \HN(\ecal)$, $\rank(\qcal)<\rank(\ecal)$, and the slopes of $T(\vec{\lambda}_i)$ are negative.

\begin{claim}\label{heightclaim}
We have $h_i \leq h$ for all $i$, with equality achieved for all $i$ only if $\mathcal{Q}$ has only one slope $\mu_1$
	and $\mu_1=\gamma$.
\end{claim}

 First let's see that Claim \ref{heightclaim} finishes off the proof of \eqref{ineq_for_step1}. Since $x_i$ is the $x$-coordinate of the vector $\vec{\mu}_i$, the sum $\sum_i x_i$ is equal to $\rank(\mathcal{Q})$. On the other hand, the vector $T(\vec{\gamma})$ is a horizontal vector of length $\rank(\mathcal{F})$. Thus by Claim~\ref{area_claim}, we find that
 \begin{align*}
 \text{LHS of } \eqref{ineq_for_step1} &=   2\sum_i \text{Area}(\Delta A_{i-1} A_i B_i) = \sum_i x_i h_i \\ 
 & \leq h \sum_i x_i = h\cdot \rank(\mathcal{Q}) \\
 & \leq h\cdot\rank(\mathcal{F}) = h \cdot \| T(\vec{\gamma}) \|  =  \text{RHS of } \eqref{ineq_for_step1}.
 \end{align*}
Moreover, equality can only hold if $\rank(\mathcal{Q}) = \rank (\mathcal{F})$ and $\mathcal{Q}$ is semistable of slope $\mu_1=\gamma$, which implies that $\Cal{Q} = \Cal{F}$. 

Thus, we have reduced to proving Claim \ref{heightclaim}.

\begin{proof}[Proof of Claim \ref{heightclaim}] To aid with proving this, we add in a helper horizontal line starting from the left endpoint of $T($HN$(\Cal{Q}))$ extending out to negative infinity to obtain an augmented polygon $T($HN$(\Cal{Q}))'$ (see the dashed line in Figure \ref{shearT}). The resulting figure is convex and retains the property described in Corollary \ref{slopes-condition-corollary}.

Let $E_i$ denote the point directly under $B_i$ lying on this polygon (see Figure \ref{fig_extend}).
By convexity, the length $h_i$ is at most the length of the line from $B_i$ to $E_i$, with equality if and only if the line segment $A_iC_i$
	lies on the augmented polygon $T($HN$(\Cal{Q}))'$.
    
By the analogue of the property described in Corollary \ref{slopes-condition-corollary} for the augmented polygon $T($HN$(\Cal{Q}))'$, as we move left, the vertical distance from $T($HN$(\Cal{E}))$ to this augmented polygon is strictly increasing.
But the distance at the far left when we  reach the end of $T($HN$(\Cal{E}))$ is precisely $h$.
This shows that $|B_iE_i|\leq h$.
Equality occurs precisely when $B_i$ is the left endpoint of $T($HN$(\Cal{E}))$ and $C_i=E_i$ is the point directly below the left endpoint of $T($HN$(\Cal{E}))$, which happens if and only if
	the line segment $A_iC_i$ is horizontal. 
But this occurs precisely when $\mu_j=\gamma$ for all $j\leq i$.
Thus, if the equality is achieved for all $i$, $\mathcal{Q}$ must have only one slope $\mu_1=\gamma$.
\end{proof}

We have finished proving Theorem \ref{step-1}. \end{proof}

	\section{Step two: extensions of semistable bundles}\label{step_2}

\subsection{The key inequality}\label{5 section 1}

%The discussion of \S \ref{strategy}, combined with the results of \S \ref{step_1} (specifically, Theorem \ref{step-1}) reduce the proof of Theorem \ref{main_thm_3} to the following statement.

The bulk of this section is devoted to proving the following result.
	
	\begin{thm}\label{degree inequality for kernel} Let $\Cal{D}$ and $\Cal{F}$ be semistable vector bundles on $X$ such that $\HNvec(\mathcal{D}) = (\vec{\beta})$,
$\HNvec(\mathcal{F})=(\vec{\gamma})$, and $\vec\beta\preceq \vec\gamma$.
	Suppose $\mathcal{E}$ is a vector bundle on $X$ such that 
    	$\HNvec(\Cal E)\leq \HNvec(\Cal{D}\oplus \Cal{F})$ 
        with the same endpoints.
     Suppose also that the maximal slope of $\Cal{E}$ is strictly less than the slope of 	
     	$\Cal{F}$.
	
	 Let $\mathcal{K}$ be a vector bundle on $X$ with the same rank and degree as $\Cal{D}$ satisfying the following two conditions:
		\begin{itemize}
			\item[(i)] the maximal slope of $\mathcal{K}$ is at most the maximal slope of $\mathcal{E}$;
			\item[(ii)] $\mathcal{K}$ is not semistable.
		\end{itemize}
		Then we have a strict inequality
		\begin{equation} 
        \deg(\mathcal{K}^\vee \otimes \mathcal{E})^{\geq 0} < \deg(\mathcal{K}^\vee \otimes \mathcal{K})^{\geq 0} + \deg(\mathcal{E}^\vee \otimes \mathcal{F})^{\geq 0}.
        \label{section 5 inequality}
        \end{equation}
	\end{thm}

\noindent \textbf{Setup and notation.} Write 
\begin{itemize}
\item $\HNvec(\mathcal{E}) = (\vec{\lambda}_j)_{1 \leq j \leq s}$,
\item $\HNvec(\mathcal{K}) = (\vec{\mu}_i)_{1 \leq i \leq s'}$.
\end{itemize}
By assumption $\HN(\mathcal{E}) \leq \HN(\mathcal{D} \oplus \mathcal{F})$, so 
\[
\vec{\lambda}_1 \prec \vec{\gamma} \hspace{1cm} \text{and} \hspace{1cm} \vec{\lambda}_s \succeq \vec{\beta}.
\]
In addition we have $\vec{\mu}_1 \preceq \vec{\lambda}_1$ by assumption (i) of Theorem \ref{degree inequality for kernel}. 
On the other hand, assumption (ii) of Theorem \ref{degree inequality for kernel} says $\vec{\mu}_{s'} \prec \vec{\beta}$, which implies $\vec{\mu}_{s'} \prec \vec{\lambda}_s$ as $\vec{\beta} \preceq \vec{\lambda}_s$.

		We will first sketch the proof of Theorem \ref{degree inequality for kernel} in geometric terms, because we feel that this conveys a better intuition for the formal argument, whose geometric meaning might be obscured by the algebraic manipulations. Then we will carefully present a thorough algebraic proof.\\

		\subsection{Proof sketch of Theorem \ref{degree inequality for kernel}}\label{proof_sketch}
        Let $O$, $P_1$,$\dots$,$P_{s'+1}$ be the vertices of $\HNvec(\Cal{F}\oplus \Cal{K})$,
	and $O$, $Q_1$, $\dots$, $Q_{s+1}=P_{s'+1}$ be the vertices of $\HNvec(\Cal{E})$.
(See Figure \ref{fig_polygon_comparison}.)
		
\begin{comment}
		By rearranging the vectors $\vec{\mu}_i$ and $\vec{\lambda}_j$ in order of decreasing slopes, we get a sequence of the form
		\begin{align*} \vec{\lambda}_1 \succ  \cdots \succ \vec{\lambda}_{j_1} \succeq \vec{\mu}_1 \succ  \cdots \succ \vec{\mu}_{i_1} \succ \vec{\lambda}_{j_1+1} \succ \cdots \succ \vec{\lambda}_{j_2} \succeq \vec{\mu}_{i_1+1} \succ \cdots \succ \vec{\mu}_{i_2} \succ \cdots\\
		\succ \vec{\lambda}_{j_{t-1}+1} \succ \cdots \succ \vec{\lambda}_{j_t = s} \succeq \vec{\mu}_{i_{t-1}+1} \succ \cdots \succ \vec{\mu}_{i_t = s'}. 
		\end{align*}
		For convenience, we set $i_0 = j_0 = 0$. Note that inequalities between $\vec{\lambda}_{j_k}$ and $\vec{\mu}_{i_{k-1}+1}$ are not strict. 
\end{comment}

\begin{center}        
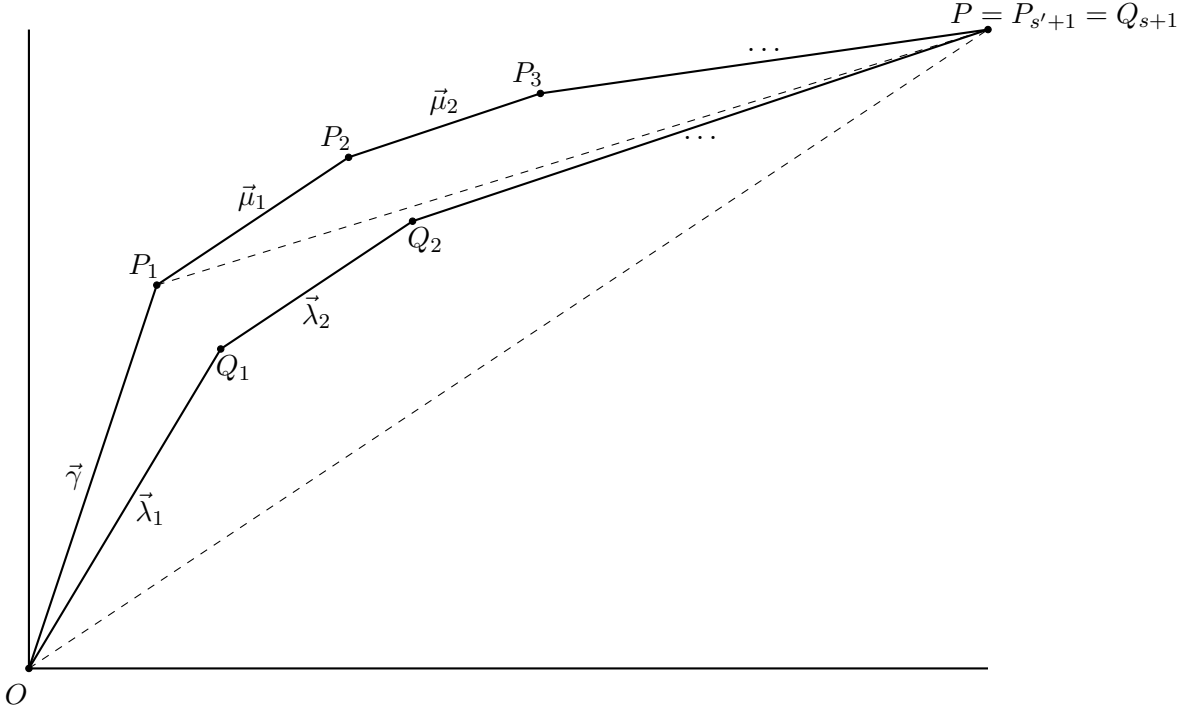
\begin{figure}[h]
		\begin{tikzpicture}[scale=0.85]

		\hspace{-0.5cm}
%		\draw [blue, fill=blue!20]  (0,0) -- (2,4) -- (15,10) --(0,0);
	%	\draw [red, fill=red!20]  (2,4) -- (5,7) --(9,9) -- (15,10) -- (2,4);
		
		\draw[step=1cm,thick] (0,0) -- node[left] {$\vec{\gamma}$} (2,6);
		\draw[step=1cm,thick] (2,6) -- node[above] {$\vec{\mu}_1$} (5,8);
		\draw[step=1cm,thick] (5,8) -- node[above] {$\vec{\mu}_2$} (8,9);
		\draw[step=1cm,thick] (8,9) -- node[above] {$\ldots$} (15,10);
        
        \draw[dashed] (2,6)--(15,10);

		\draw[step=1cm,thick] (0,0) -- node[right] {$\vec{\lambda}_1$} (3,5);
		\draw[step=1cm,thick] (3,5) -- node[below] {$\vec{\lambda}_2$} (6,7);
		\draw[step=1cm,thick] (6,7) -- node[below] {$\ldots$} (15,10);

		\node at (0-0.2,0-0.4) {$O $};
		\node at (5-0.2,8+0.3) {$P_2$};
		\node at (8-0.2,9+0.3) {$P_3$};
		\node at (15+1.2,10.2) {$P=P_{s'+1}=Q_{s+1}$};
		\node at (2-0.2,6+0.3) {$P_1$};

		\node at (3+0.2,5-0.3) {$Q_1$};
		\node at (6+0.2,7-0.3) {$Q_2$};
		%\node at (15+0.2,10-0.3) {$Q_4$};

		\draw[step=1cm,thick] (0,0) -- (0,10);
		
		\draw[step=1cm,thick] (0,0) -- (15,0);

		\draw [fill] (0,0) circle [radius=0.05];
		
		\draw [fill] (2,6) circle [radius=0.05];
		
		\draw [fill] (5,8) circle [radius=0.05];
		
		\draw [fill] (8,9) circle [radius=0.05];

		\draw [fill] (15,10) circle [radius=0.05];
		
		\draw [fill] (3,5) circle [radius=0.05];
		
		\draw [fill] (6,7) circle [radius=0.05];

        \draw[dashed] (0,0)--(15,10);
         \useasboundingbox (-2,0);
        
		\end{tikzpicture}	
		\caption{Comparison of the HN vectors for $\HN(\mathcal{K} \oplus \mathcal{F})$ (the top solid lines) and $\HN(\mathcal{E})$ (the bottom solid lines).}\label{fig_polygon_comparison}
		\end{figure}
\end{center}

        First, we claim that the right hand side of the inequality \eqref{section 5 inequality} is twice the area of the convex polygon $OP_1\dots P_{s'+1}$.
        Indeed, by Proposition \ref{geometric diamond lemma for one vector bundle} we see that $\deg(\Cal{K}^\vee\otimes \Cal{K})^{\geq 0}$ is twice the area of the polygon $P_1 P_2 \dots P_{s'+1}$, while 
        \[
        \deg(\Cal{E}^\vee \otimes \Cal{F})^{\geq 0} =  
        	\sum_{{i=1}}^s \vec\lambda_j \times \vec\gamma
            =\overrightarrow{OP_{s'+1}}\times \overrightarrow{OP_1}
        \]
        is twice the area of the triangle $\Delta OP_1 P_{s'+1}.$
        Summing then gives us the area of polygon $OP_1\dots P_{s'+1}$. (See Figure \ref{area_inequality_1 first}.)
        
\begin{center}        
\begin{figure}[h]
		\begin{tikzpicture}[scale=0.85]

		\hspace{-0.5cm}
		\draw [blue, fill=blue!20]  (0,0) -- (2,6) -- (15,10) --(0,0);
		\draw [red, fill=red!20]  (2,6) -- (5,8) --(8,9) -- (15,10) -- (2,6);
		
		\draw[step=1cm,thick] (0,0) -- node[left] {$\vec{\gamma}$} (2,6);
		\draw[step=1cm,thick] (2,6) -- node[above] {$\vec{\mu}_1$} (5,8);
		\draw[step=1cm,thick] (5,8) -- node[above] {$\vec{\mu}_2$} (8,9);
		\draw[step=1cm,thick] (8,9) -- node[above] {$\vec{\mu}_3$} (15,10);
		\draw[step=1cm,thick] (2,6) -- node[pos=0.4,sloped, above=5pt] {$ \frac{1}{2}\deg(\mathcal{K}^\vee \otimes \mathcal{K})^{\geq 0}$} (15,10);

		\node at (0-0.2,0-0.2) {$O$};
		\node at (5-0.2,8+0.3) {$P_2$};
		\node at (8-0.2,9+0.3) {$P_3$};
		\node at (15-0.2,10+0.3) {$P_4$};
				\node at (2-0.2,6+0.3) {$P_1$};
		
		\draw[step=1cm,thick] (0,0) --  node[sloped, below] {$\vec{\lambda}_1+\vec{\lambda}_2+\vec{\lambda}_3$} (15,10) node[pos=0.5, sloped, above=10pt] {$ \frac{1}{2}\deg(\mathcal{E}^\vee \otimes \mathcal{F})^{\geq 0}$};

		\draw[step=1cm,thick] (0,0) -- (0,10);
		
		\draw[step=1cm,thick] (0,0) -- (15,0);

		\draw [fill] (0,0) circle [radius=0.05];
		
		\draw [fill] (2,6) circle [radius=0.05];
		
		\draw [fill] (5,8) circle [radius=0.05];
		
		\draw [fill] (8,9) circle [radius=0.05];

		\draw [fill] (15,10) circle [radius=0.05];
		
		\useasboundingbox (-2,0);
		
		\end{tikzpicture}	
		\caption{Geometric interpretation of the left hand side of \eqref{section 5 inequality}.}
        \label{area_inequality_1 first}
		\end{figure}
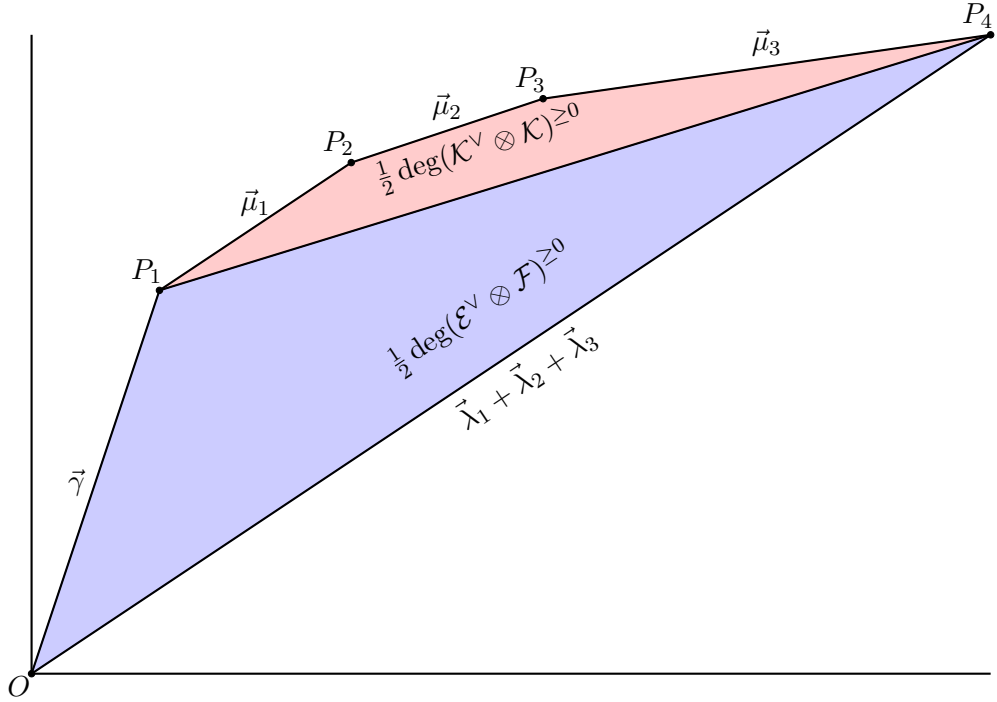
\end{center}

        Now consider the left side of the inequality \eqref{section 5 inequality}.
        By Lemma \ref{diamond lemma for nonnegative degree}, we have
        \begin{align*}
        	\deg(\Cal{K}^\vee\otimes \Cal{E})^{\geq 0} 
            	&= \sum_{\vec{\mu}_i\preceq \vec{\lambda}_j} \vec{\mu}_i\times \vec{\lambda}_j \\
                &= \sum_{i=1}^{s'} \vec{\mu}_i \times \left(\sum_{\vec\lambda_j\succeq \vec \mu_i} \vec\lambda_j \right) \\
                &= \sum_{i=1}^{s'} \overrightarrow{P_i P_{i+1}} \times \overrightarrow{OQ_{j_i}},
        \end{align*}
        where $Q_{j_i}$ is some vertex of $\HN(\Cal{E})$.
        Now we claim that $\overrightarrow{P_iP_{i+1}}\times \overrightarrow{OQ_{j_i}}$
        	is at most twice the area of the triangle $OP_iP_{i+1}$, with equality if 
            and only if $Q_{j_i}$ lies on the line segment $P_iP_{i+1}$.
        Indeed, twice the area of $OP_iP_{i+1}$ is just 
        	$\overrightarrow{P_iP_{i+1}}\times \overrightarrow{OP_i}$,
            so
        \[
        	\overrightarrow{P_iP_{i+1}} \times \overrightarrow{OP_i}
            -\overrightarrow{P_iP_{i+1}} \times \overrightarrow{OQ_{j_i}}
            =\overrightarrow{P_iP_{i+1}} \times \overrightarrow{P_iQ_{j_i}}
        \]
        	is by convexity considerations just twice the area of $\Delta P_iP_{i+1}Q_{j_i}$, which is nonnegative.
        By convexity, it is $0$ if and only if $Q_{j_i}$ lies on the segment $P_iP_{i+1}$.  Then we have
\begin{align*}
	\deg(\Cal{K}^\vee\otimes \Cal E)^{\geq 0} 
    	&\leq \sum_{i=1}^{s'} 2\cdot\text{Area}(\Delta OP_iP_{i+1})\\
        &=2\cdot\text{Area}(OP_1\dots P_{s'+1}) \\
        &= \deg(\mathcal{K}^\vee \otimes \mathcal{K})^{\geq 0} + \deg(\mathcal{E}^\vee \otimes \mathcal{F})^{\geq 0}.
\end{align*}
To see equality cannot hold, observe that, as by assumption $\HNvec(\Cal{E})\leq \HNvec(\Cal{F}\oplus \Cal{D})$, all $Q_{j_i}$ lie inside the triangle $\Delta OP_1P_{s'+1}$.
But since the largest slope of $\Cal{E}$ is strictly less than the slope of
	$\Cal{F}$, we see that $Q_{j_1}$ cannot be the point $P_1$, so cannot lie on
    the segment $P_1P_2$ as $\Cal{K}$ is not semistable.
This gives the theorem. \qed

\subsection{Formal proof of Theorem \ref{degree inequality for kernel}}
We now carefully go through a formal argument for Theorem \ref{degree inequality for kernel}. This is essentially just a thorough formalization of \S \ref{proof_sketch}, so the reader who is already convinced by \S \ref{proof_sketch} can skip ahead to the next subsection.

Maintain the notation of \S \ref{5 section 1}. By rearranging the vectors $\vec{\mu}_i$ and $\vec{\lambda}_j$ in order of decreasing slopes, we get a sequence of the form
		\begin{align*} \vec{\lambda}_1 \succ  \cdots \succ \vec{\lambda}_{j_1} \succeq \vec{\mu}_1 \succ  \cdots \succ \vec{\mu}_{i_1} \succ \vec{\lambda}_{j_1+1} \succ \cdots \succ \vec{\lambda}_{j_2} \succeq \vec{\mu}_{i_1+1} \succ \cdots \succ \vec{\mu}_{i_2} \succ \cdots\\
		\succ \vec{\lambda}_{j_{t-1}+1} \succ \cdots \succ \vec{\lambda}_{j_t = s} \succeq \vec{\mu}_{i_{t-1}+1} \succ \cdots \succ \vec{\mu}_{i_t = s'}. 
		\end{align*}
		For convenience, we set $i_0 = j_0 = 0$. Note that inequalities between $\vec{\lambda}_{j_k}$ and $\vec{\mu}_{i_{k-1}+1}$ are not strict. Let $\vec{\beta}=(r,d)$ and $\vec{\gamma}=(r',d')$. 
        
		\subsubsection{A simple special case}

		Let us first consider the case when $t=1$, which occurs precisely when $\vec{\mu}_i \preceq \vec{\lambda}_j$ for all $1 \leq i \leq s'$ and $1 \leq j \leq s$. By Lemma \ref{diamond lemma for nonnegative degree}, we compute
		\begin{align*}
		\deg(\mathcal{K}^\vee \otimes \mathcal{E})^{\geq 0} &= \sum_{\vec{\mu}_i \preceq \vec{\lambda}_j} \vec{\mu}_i \times \vec{\lambda}_j = \sum_{i, j} \vec{\mu}_i \times \vec{\lambda}_j = \deg(\mathcal{K}^\vee \otimes \mathcal{E})\\
		&= \rank(\mathcal{K}) \deg(\mathcal{E}) - \rank(\mathcal{E})\deg(\mathcal{K}) \\
		& = r(d+d') - (r+r') d = rd' - r'd, \\
		\deg(\mathcal{E}^\vee \otimes \mathcal{F})^{\geq 0} &= \sum_{\vec{\lambda}_j \preceq \vec{\gamma}} \vec{\lambda}_j \times \vec{\gamma} = \sum_j \vec{\lambda}_j \times \vec{\gamma} = \deg(\mathcal{E}^\vee \otimes \mathcal{F})\\
		&= \rank(\mathcal{E}) \deg(\mathcal{F}) - \rank(\mathcal{F}) \deg(\mathcal{E})\\
		&= (r+r') d' - r'(d+d') = rd' - r'd. 
		\end{align*}
		In particular, we have $\deg(\mathcal{K}^\vee \otimes \mathcal{E})^{\geq 0}  = \deg(\mathcal{E}^\vee \otimes \mathcal{F})^{\geq 0}$. On the other hand, Proposition \ref{geometric diamond lemma for one vector bundle} yields $\deg(\mathcal{K}^\vee \otimes \mathcal{K})^{\geq 0} > 0$, as $\mathcal{K}$ is not semistable. Hence we deduce the desired inequality, in this case.
		
		\subsubsection{Reduction to simpler slope relations}
		We now assume that $t \geq 2$. We now reduce to the case where the slope vectors satisfy a simpler 	``intertwining'' relation. 
		
		Define, for each index $k=1, 2, \dotsc, t$,
		\begin{align*}
		 \tilde{\vec{\lambda}}_k & := \vec{\lambda}_{j_{k-1}+1} + \cdots + \vec{\lambda}_{j_k}, \\
		 \tilde{\vec{\mu}}_k &:= \vec{\mu}_{i_{k-1}+1} + \cdots + \vec{\mu}_{i_k}.
		 \end{align*}
		  
		  Then we clearly have
		\begin{equation}\label{slope inequalities after reduction} \tilde{\vec{\lambda}}_1 \succeq \tilde{\vec{\mu}}_1 \succ \tilde{\vec{\lambda}}_2 \succeq \tilde{\vec{\mu}}_2 \succ \cdots \succ \tilde{\vec{\lambda}}_t \succeq \tilde{\vec{\mu}}_t.\end{equation}
		Define two vector bundles $\tilde{\mathcal{K}}$ and $\tilde{\mathcal{E}}$ on $\adicff_C$ by $\HNvec(\tilde{\mathcal{K}}) = (\tilde{\vec{\mu}}_k)_{1 \leq k \leq t}$ and  $\HNvec(\tilde{\mathcal{E}}) = (\tilde{\vec{\lambda}}_k)_{1 \leq k \leq t}$. Then we have the following properties of $\tilde{\mathcal{K}}$ and $\tilde{\mathcal{E}}$:
		\begin{itemize}
			\item $\deg(\tilde{\mathcal{K}}) = \deg(\mathcal{K}) = d, ~\rank(\tilde{\mathcal{K}}) = \rank(\mathcal{K})= r$;
			\item $\deg(\tilde{\mathcal{E}}) = \deg(\mathcal{E}) = d+d', ~\rank(\tilde{\mathcal{E}}) = \rank(\mathcal{E})= r+r'$;
			\item $\tilde{\mathcal{K}}$ is not semistable since $t \geq 2$; 
			\item the maximal slope of $\tilde{\mathcal{K}}$ is less than or equal to the maximal slope of $\tilde{\mathcal{E}}$, as $\tilde{\vec{\mu}}_1 \preceq \tilde{\vec{\lambda}}_1$; 
            \item the maximal slope of $\tilde{\Cal{E}}$ is strictly less than the slope of	
     	$\Cal{F}$, as $\tilde{\vec{\lambda}}_1 \preceq \vec{\lambda}_1 \prec \vec{\gamma}$;
			\item $\HN(\tilde{\mathcal{E}}) \leq \HN(\mathcal{E})$, and hence $\HN(\tilde{\mathcal{E}}) \leq \HN(\mathcal{D} \oplus \mathcal{F})$. 
		\end{itemize}
		We claim that it suffices to prove Theorem \ref{degree inequality for kernel} for $\tilde{\mathcal{K}}$ and $\tilde{\mathcal{E}}$. 
		
		Indeed, $\tilde{\mathcal{K}}$ and $\tilde{\mathcal{E}}$ satisfy the same assumptions made on $\mathcal{K}$ and $\mathcal{E}$ in Theorem \ref{degree inequality for kernel}. It is also straightforward to check the identities 
		\[\deg(\mathcal{K}^\vee \otimes \mathcal{E})^{\geq 0} = \deg(\tilde{\mathcal{K}}^\vee \otimes \tilde{\mathcal{E}})^{\geq 0} \quad \text{ and } \quad \deg(\mathcal{E}^\vee \otimes \mathcal{F})^{\geq 0} = \deg(\tilde{\mathcal{E}}^\vee \otimes \mathcal{F})^{\geq 0}\] 
		using Lemma \ref{diamond lemma for nonnegative degree} and the inequalities in \eqref{slope inequalities after reduction}. In addition, since $\HN(\tilde{\mathcal{K}}) \leq \HN(\mathcal{K})$, Corollary \ref{diamond inequality} yields
		\[ \deg(\tilde{\mathcal{K}}^\vee \otimes \tilde{\mathcal{K}})^{\geq 0} \leq \deg(\mathcal{K}^\vee \otimes \mathcal{K})^{\geq 0}.\]
		Hence for the purpose of establishing Theorem \ref{degree inequality for kernel} it suffices to prove the inequality 
		\begin{equation}\label{reduced degree inequality for kernel} \deg(\tilde{\mathcal{K}}^\vee \otimes \tilde{\mathcal{E}})^{\geq 0} < \deg(\tilde{\mathcal{K}}^\vee \otimes \tilde{\mathcal{K}})^{\geq 0} + \deg(\tilde{\mathcal{E}}^\vee \otimes \mathcal{F})^{\geq 0}.\end{equation}
		
		\subsubsection{Some algebraic reductions} 
		We now complete the proof of Theorem \ref{degree inequality for kernel}. 
		
		Using Lemma \ref{diamond lemma for nonnegative degree} and the inequalities in \eqref{slope inequalities after reduction}, we find
		\begin{align*}
		\deg(\tilde{\mathcal{K}}^\vee \otimes \tilde{\mathcal{E}})^{\geq 0} &= \sum_{i \geq j} \tilde{\vec{\mu}}_i \times \tilde{\vec{\lambda}}_j = \sum_{i=1}^t \left (\tilde{\vec{\mu}}_i \times \sum_{j=1}^i \tilde{\vec{\lambda}}_j \right)\\
		\deg(\tilde{\mathcal{K}}^\vee \otimes \tilde{\mathcal{K}})^{\geq 0} &= \sum_{i \geq j} \tilde{\vec{\mu}}_i \times \tilde{\vec{\mu}}_j = \sum_{i=1}^t \left(\tilde{\vec{\mu}}_i \times \sum_{j=1}^i \tilde{\vec{\mu}}_j \right), \\
		\deg(\tilde{\mathcal{E}}^\vee \otimes \mathcal{F})^{\geq 0} &=  \sum_{j=1}^t \tilde{\vec{\lambda}}_j \times \vec{\gamma} = \left( \sum_{j=1}^t \tilde{\vec{\lambda}}_j \right) \times \vec{\gamma}.
		\end{align*}
		
		Note that $\vec{\gamma} + (\tilde{\vec{\mu}}_1 + \cdots + \tilde{\vec{\mu}}_t) = \tilde{\vec{\lambda}}_1 + \cdots + \tilde{\vec{\lambda}}_t$, so we may write
		\[\deg(\tilde{\mathcal{E}}^\vee \otimes \mathcal{F})^{\geq 0} =  \left( \vec{\gamma} + \sum_{i=1}^t \tilde{\vec{\mu}}_i \right) \times \vec{\gamma} = \sum_{i=1}^t \tilde{\vec{\mu}}_i  \times \vec{\gamma}.\]
		Substituting this above yields
		\begin{align}\begin{split}\label{right hand side for kernel inequality}
		\deg(\tilde{\mathcal{K}}^\vee \otimes \tilde{\mathcal{K}})^{\geq 0} + \deg(\tilde{\mathcal{E}}^\vee \otimes \mathcal{F})^{\geq 0} &= \sum_{i=1}^t \left[\tilde{\vec{\mu}}_i \times \left( \vec{\gamma} + \sum_{j=1}^i \tilde{\vec{\mu}}_j \right) \right]\\
		&= \sum_{i=1}^t \left[\tilde{\vec{\mu}}_i \times \left( \sum_{j=1}^t \tilde{\vec{\lambda}}_j - \sum_{j=i+1}^t \tilde{\vec{\mu}}_j \right) \right].
		\end{split}\end{align}

		We thus have
		\begin{align*} 
		&\deg(\tilde{\mathcal{K}}^\vee \otimes \tilde{\mathcal{K}})^{\geq 0} + \deg(\tilde{\mathcal{E}}^\vee \otimes \mathcal{F})^{\geq 0} - \deg(\tilde{\mathcal{K}}^\vee \otimes \tilde{\mathcal{E}})^{\geq 0} \\
		&= \sum_{i=1}^t \left[\tilde{\vec{\mu}}_i \times \left( \sum_{j=1}^t \tilde{\vec{\lambda}}_j - \sum_{j=i+1}^t \tilde{\vec{\mu}}_j \right) - \tilde{\vec{\mu}}_i \times \sum_{j=1}^i \tilde{\vec{\lambda}}_j \right]\\
		&=  \sum_{i=1}^t \left[  \tilde{\vec{\mu}}_i \times \left( \sum_{j=i+1}^t \tilde{\vec{\lambda}}_j - \sum_{j=i+1}^t \tilde{\vec{\mu}}_j \right) \right]\\
		&=  \sum_{i=1}^{t-1} \left[  \tilde{\vec{\mu}}_i \times \left( \sum_{j=i+1}^t \tilde{\vec{\lambda}}_j - \sum_{j=i+1}^t \tilde{\vec{\mu}}_j \right) \right].
		\end{align*}
		Hence, to prove \eqref{reduced degree inequality for kernel}, it suffices to prove 
		\begin{equation}\label{vector representation for kernel inequality} 
		0<  \tilde{\vec{\mu}}_i \times \left( \sum_{j=i+1}^t \tilde{\vec{\lambda}}_j - \sum_{j=i+1}^t \tilde{\vec{\mu}}_j \right) \quad \text{ for } i=1, 2, \dotsc, t-1. 
		\end{equation}

		\subsubsection{Completion of the proof} 
		
For the conclusion, it will be useful to refer to the geometry of the HN polygons. Let $O, P_1, P_2, \dotsc, P_{t+1}= P$ be the breakpoints of $\HN(\tilde{\mathcal{K}} \oplus \mathcal{F})$, listed in order of increasing $x$-coordinates. Similarly, let $O= Q_1, Q_2, \dotsc, Q_{t+1} = P$ be the breakpoints of $\HN(\tilde{\mathcal{E}})$, also listed in order of increasing $x$-coordinates. See Figure \ref{fig_polygon_comparison_again}.
		
\begin{center}        
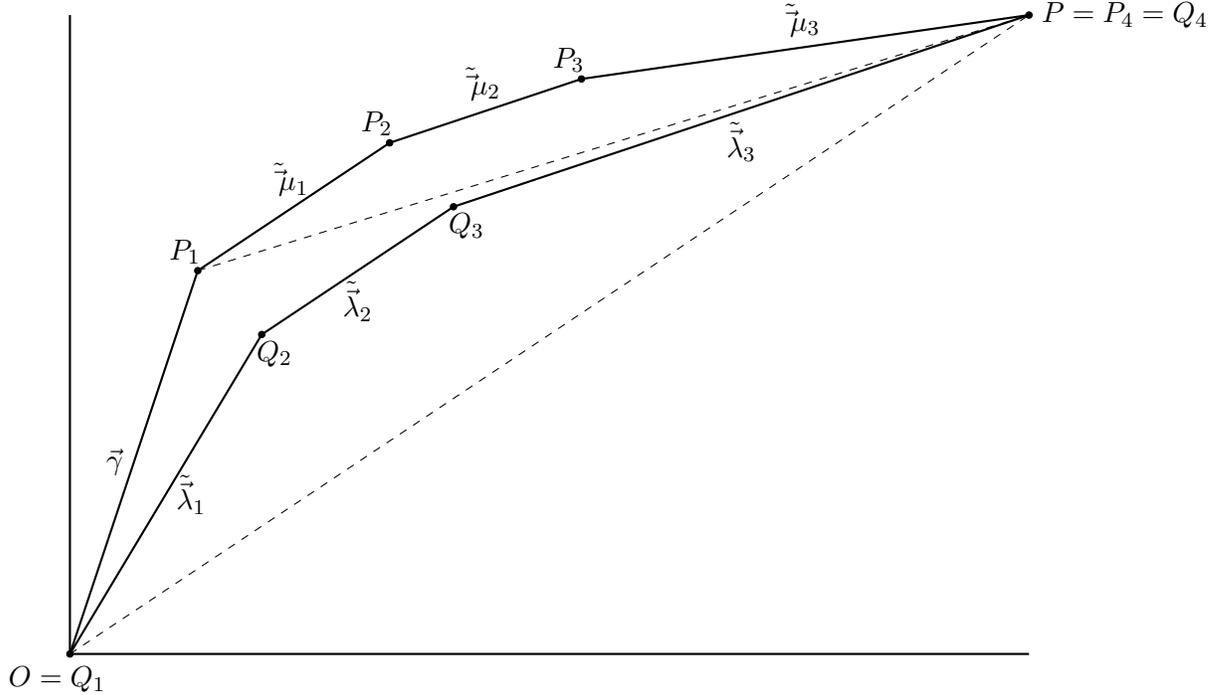
\begin{figure}[h]
		\begin{tikzpicture}[scale=0.85]

		\hspace{-0.5cm}
%		\draw [blue, fill=blue!20]  (0,0) -- (2,4) -- (15,10) --(0,0);
	%	\draw [red, fill=red!20]  (2,4) -- (5,7) --(9,9) -- (15,10) -- (2,4);
		
		\draw[step=1cm,thick] (0,0) -- node[left] {$\vec{\gamma}$} (2,6);
		\draw[step=1cm,thick] (2,6) -- node[above] {$\tilde\vec{\mu}_1$} (5,8);
		\draw[step=1cm,thick] (5,8) -- node[above] {$\tilde\vec{\mu}_2$} (8,9);
		\draw[step=1cm,thick] (8,9) -- node[above] {$\tilde\vec{\mu}_3$} (15,10);
        
        \draw[dashed] (2,6)--(15,10);

		\draw[step=1cm,thick] (0,0) -- node[right] {$\tilde{\vec{\lambda}}_1$} (3,5);
		\draw[step=1cm,thick] (3,5) -- node[below] {$\tilde{\vec{\lambda}}_2$} (6,7);
		\draw[step=1cm,thick] (6,7) -- node[below] {$\tilde{\vec{\lambda}}_3$} (15,10);

		\node at (0-0.2,0-0.4) {$O =Q_1 $};
		\node at (5-0.2,8+0.3) {$P_2$};
		\node at (8-0.2,9+0.3) {$P_3$};
		\node at (15+1.5,10) {$P=P_4=Q_4$};
		\node at (2-0.2,6+0.3) {$P_1$};

		\node at (3+0.2,5-0.3) {$Q_2$};
		\node at (6+0.2,7-0.3) {$Q_3$};
		%\node at (15+0.2,10-0.3) {$Q_4$};

		\draw[step=1cm,thick] (0,0) -- (0,10);
		
		\draw[step=1cm,thick] (0,0) -- (15,0);

		\draw [fill] (0,0) circle [radius=0.05];
		
		\draw [fill] (2,6) circle [radius=0.05];
		
		\draw [fill] (5,8) circle [radius=0.05];
		
		\draw [fill] (8,9) circle [radius=0.05];

		\draw [fill] (15,10) circle [radius=0.05];
		
		\draw [fill] (3,5) circle [radius=0.05];
		
		\draw [fill] (6,7) circle [radius=0.05];

        \draw[dashed] (0,0)--(15,10);
          \useasboundingbox (-2,0);
		\end{tikzpicture}	
		\caption{Comparison of the HN vectors for $\HN(\tilde{\mathcal{K}} \oplus \mathcal{F})$ (the top solid lines) and $\HN(\tilde{\mathcal{E}})$ (the bottom solid lines).}\label{fig_polygon_comparison_again}
		\end{figure}
\end{center}
Then we have $\tilde{\vec{\mu}}_i = \overrightarrow{P_i P_{i+1}}$ and $\tilde{\vec{\lambda}}_j = \overrightarrow{Q_j Q_{j+1}}$, which yields
\begin{align*}
\sum_{j=i+1}^t \tilde{\vec{\lambda}}_j - \sum_{j=i+1}^t \tilde{\vec{\mu}}_j & = \sum_{j=i+1}^t \overrightarrow{Q_j Q_{j+1}} - \sum_{j=i+1}^t \overrightarrow{P_j P_{j+1}} \\
& = \overrightarrow{Q_{i+1} Q_{t+1}} - \overrightarrow{P_{i+1} P_{t+1}} \\
&= \overrightarrow{Q_{i+1} P_{i+1}}.
\end{align*}
		Thus we may rewrite \eqref{vector representation for kernel inequality} as
		\begin{equation}\label{geometric representation for kernel inequality} 
		0 < \overrightarrow{P_i P_{i+1}} \times \overrightarrow{Q_{i+1} P_{i+1}}.
		\end{equation}

		We will prove the inequality \eqref{geometric representation for kernel inequality}  by dividing into three cases depending on the sign of the $x$-coordinate of $\overrightarrow{Q_{i+1} P_{i+1}}$. Note that the $x$-coordinate of $\overrightarrow{P_i P_{i+1}}$ is always positive. 
		
\begin{enumerate}
\item Suppose first that the $x$-coordinate of $\overrightarrow{Q_{i+1} P_{i+1}}$ is positive. In this case, \eqref{geometric representation for kernel inequality} is equivalent by Lemma \ref{relation_area} to
		\begin{align*}
&		  \overrightarrow{P_i P_{i+1}} \prec  \overrightarrow{Q_{i+1} P_{i+1}} \\
		  \iff &\overrightarrow{P_{i+1}P_i} \succ  \overrightarrow{P_{i+1}Q_{i+1}}
		\end{align*}
		which clearly holds by the fact that $\HN(\tilde{\mathcal{E}}) \leq \HN(\tilde{\mathcal{K}} \oplus \mathcal{F})$ and $\tilde{\vec{\lambda}}_1 \prec \vec{\gamma}$. 
		
\item Next consider the case when the $x$-coordinate of $\overrightarrow{Q_{i+1} P_{i+1}}$ is negative. In this case, \eqref{geometric representation for kernel inequality} is equivalent to
		\[  \overrightarrow{P_i P_{i+1}} \succ  \overrightarrow{Q_{i+1} P_{i+1}}.\]
		In fact, one has a stronger inequality 
		\[  \overrightarrow{P_{i+1} P_{i+2}} \succeq \overrightarrow{Q_{i+1} P_{i+1}}\]
		(this is stronger because $ \overrightarrow{P_i P_{i+1}}  \succ \overrightarrow{P_{i+1} P_{i+2}} $ by definition), which can be easily seen by the fact that $\HN(\tilde{\mathcal{E}}) \leq \HN(\tilde{\mathcal{K}} \oplus \mathcal{F})$.
		
\item	It remains to consider the case when the $x$-coordinate of $\overrightarrow{Q_{i+1} P_{i+1}}$ is zero. In this case, \eqref{geometric representation for kernel inequality} is equivalent to saying that the $y$-coordinate of $\overrightarrow{Q_{i+1} P_{i+1}}$ is positive. In fact, this is clearly nonnegative by the fact that $\HN(\tilde{\mathcal{E}}) \leq \HN(\tilde{\mathcal{K}} \oplus \mathcal{F})$, so it remains to prove that this coordinate is never zero. Suppose for contradiction that the $y$-coordinate of $\overrightarrow{Q_{i+1} P_{i+1}}$ is zero, which means that $P_{i+1} = Q_{i+1}$. However, since $\HN(\tilde{\mathcal{E}}) \leq \HN(\tilde{\mathcal{D}} \oplus \mathcal{F})$, the only point on $\tilde{\mathcal{K}}$ which can also lie on $\tilde{\mathcal{E}}$ is $P_1$. Hence we must have $P_1 = Q_1=O$, which is clearly a contradiction. 
		\end{enumerate}
\qed

	\begin{remark} We explain a geometric perspective of \eqref{right hand side for kernel inequality}, as an example of the translation from the geometric argument sketched in \S \ref{proof_sketch}. As in the proof, let $O=P_0, P_1, P_2, \dotsc, P_{t+1}$ be the breakpoints of $\HN(\tilde{\mathcal{K}} \oplus \mathcal{F})$, listed in the order of increasing $x$-coordinates. Note that $\vec{\gamma} = \overrightarrow{O P_1}, ~\tilde{\vec{\mu}}_i = \overrightarrow{P_i P_{i+1}}$ and $\tilde{\vec{\lambda}}_1 + \cdots + \tilde{\vec{\lambda}}_t = \overrightarrow{O P_{t+1}}$. 
		
%		\lynnelle{the explanation below that $\deg(\tilde{\mathcal{K}}^\vee \otimes \tilde{\mathcal{K}})^{\geq 0} + \deg(\tilde{\mathcal{E}}^\vee \otimes \mathcal{F})^{\geq 0}$ is the area of the big polygon was already given in 5.2 and seems redundant. Should we consider rewriting?}

		By Proposition \ref{geometric diamond lemma for one vector bundle}, $\deg(\tilde{\mathcal{K}}^\vee \otimes \tilde{\mathcal{K}})^{\geq 0} = 2 \cdot \text{Area}(P_1 P_2 \cdots P_{t+1})$. In addition, we have 
		\[\deg(\tilde{\mathcal{E}}^\vee \otimes \mathcal{F})^{\geq 0}  = (\tilde{\vec{\lambda}}_1 + \cdots +\tilde{\vec{\lambda}}_t) \times \vec{\gamma} = \overrightarrow{OP_{t+1}} \times \overrightarrow{OP_1} = 2 \cdot \text{Area}(\bigtriangleup OP_1 P_{t+1}).\]
		We may thus write (please refer to Figure \ref{area_inequality_1})
		\begin{align*}
		\deg(\tilde{\mathcal{K}}^\vee \otimes \tilde{\mathcal{K}})^{\geq 0} + \deg(\tilde{\mathcal{E}}^\vee \otimes \mathcal{F})^{\geq 0} &= 2 \cdot \text{Area}(P_1 P_2 \cdots P_{t+1}) + 2 \cdot \text{Area}(\bigtriangleup OP_1 P_{t+1}) \\
		& = 2 \cdot \text{Area}(O P_1 P_2 \cdots P_{t+1}) \\
		&= 2 \sum_{i=1}^t \text{Area}(\bigtriangleup OP_i P_{i+1})=2 \sum_{i=1}^t \Delta_i
		\end{align*}
        
        where we set $\Delta_i=\text{Area}(\bigtriangleup OP_i P_{i+1})$.
		
\begin{center}        
\begin{figure}[h]
		\begin{tikzpicture}[scale=0.7]

		\hspace{-0.5cm}
		\draw [blue, fill=blue!20]  (0,0) -- (2,4) -- (15,10) --(0,0);
		\draw [red, fill=red!20]  (2,4) -- (5,7) --(9,9) -- (15,10) -- (2,4);
		
		\draw[step=1cm,thick] (0,0) -- node[above] {$\vec{\gamma}$} (2,4);
		\draw[step=1cm,thick] (2,4) -- node[above] {$\tilde\vec{\mu}_1$} (5,7);
		\draw[step=1cm,thick] (5,7) -- node[above] {$\tilde\vec{\mu}_2$} (9,9);
		\draw[step=1cm,thick] (9,9) -- node[above] {$\tilde\vec{\mu}_3$} (15,10);
		\draw[step=1cm,thick] (2,4) -- node[sloped, above=5pt] {$ \frac{1}{2}\deg(\tilde{\mathcal{K}}^\vee \otimes \tilde{\mathcal{K}})^{\geq 0}$} (15,10);

		\node at (0-0.2,0-0.2) {$O$};
		\node at (5-0.2,7+0.3) {$P_2$};
		\node at (9-0.2,9+0.3) {$P_3$};
		\node at (15-0.2,10+0.3) {$P_4$};
				\node at (2-0.2,4+0.3) {$P_1$};
		
		\draw[step=1cm,thick] (0,0) --  node[sloped, below] {$\tilde\vec{\lambda}_1+\tilde\vec{\lambda}_2+\tilde\vec{\lambda}_3$} (15,10) node[pos=0.3, sloped, above=10pt] {$ \frac{1}{2}\deg(\tilde{\mathcal{E}}^\vee \otimes \mathcal{F})^{\geq 0}$};

		\draw[step=1cm,thick] (0,0) -- (0,10);
		
		\draw[step=1cm,thick] (0,0) -- (15,0);

		\draw [fill] (0,0) circle [radius=0.05];
		
		\draw [fill] (2,4) circle [radius=0.05];
		
		\draw [fill] (5,7) circle [radius=0.05];
		
		\draw [fill] (9,9) circle [radius=0.05];

		\draw [fill] (15,10) circle [radius=0.05];
		
			\useasboundingbox (-2,0);
		
		\end{tikzpicture}	
		\caption{Geometric interpretation of the left hand side of \eqref{right hand side for kernel inequality}.}\label{area_inequality_1}
		\end{figure}
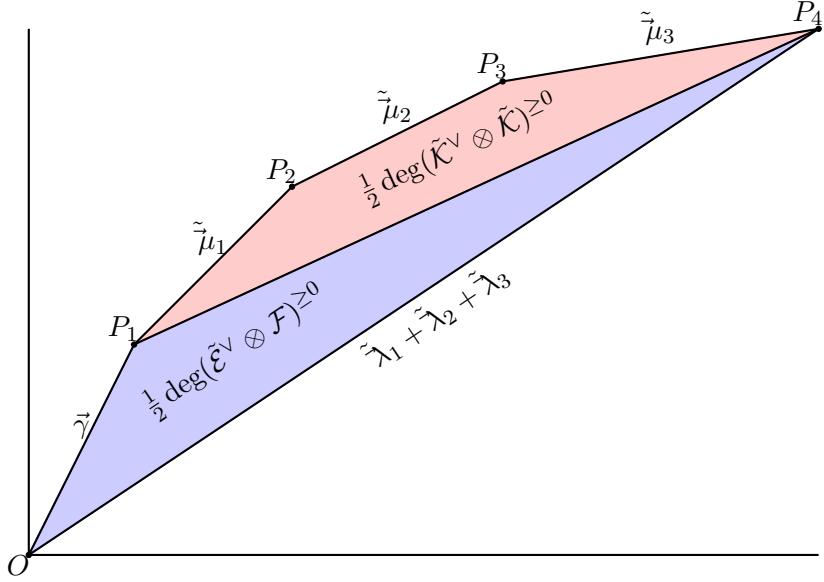
\end{center}

		On the other hand, for each $i=1, 2, \dotsc, t$ we have (please refer to Figure \ref{area_inequality_2}):
		\begin{align*} 
		2\Delta_i &= \overrightarrow{P_i P_{i+1}} \times \overrightarrow{OP_{i+1}}= \overrightarrow{P_i P_{i+1}} \times \big(\overrightarrow{OP_{t+1}} - \overrightarrow{P_{i+1}P_{t+1}}\big)\\
		&= \tilde{\vec{\mu}}_i \times \left( \sum_{j=1}^t \tilde{\vec{\lambda}}_j - \sum_{j=i+1}^t \tilde{\vec{\mu}}_j \right).
		\end{align*} Hence we deduce \eqref{right hand side for kernel inequality}.

\begin{center}
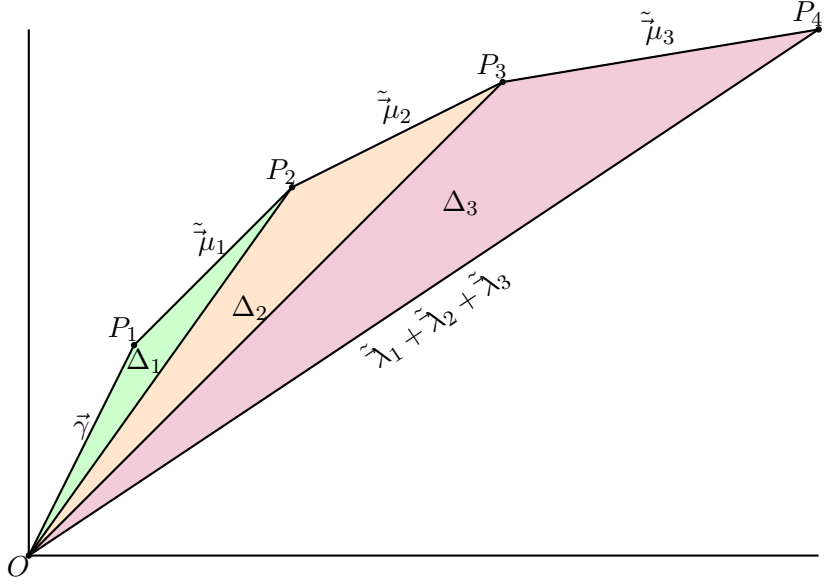
\begin{figure}[h]
		\begin{tikzpicture}[scale=0.7]
        \hspace{-0.5cm}
		\draw [green, fill=green!20]  (0,0) -- (2,4) -- (5,7) --(0,0);
		\draw [orange, fill=orange!20]  (0,0) -- (5,7) --(9,9) -- (0,0);
		\draw [purple, fill=purple!20]  (0,0) -- (9,9) --(15,10) -- (0,0);
		
		\draw[step=1cm,thick] (0,0) -- node[above] {$\vec{\gamma}$} (2,4);
		\draw[step=1cm,thick] (2,4) -- node[above] {$\tilde\vec{\mu}_1$} (5,7);
		\draw[step=1cm,thick] (5,7) -- node[above] {$\tilde\vec{\mu}_2$} (9,9);
		\draw[step=1cm,thick] (9,9) -- node[above] {$\tilde\vec{\mu}_3$} (15,10);
		
		\node at (2+0.2,4-0.3) {$\Delta_1$};
		\node at (4+0.2,5-0.3) {$\Delta_2$};
		\node at (8+0.2,7-0.3) {$\Delta_3$};	
		
		\draw[step=1cm,thick] (0,0) --  node[sloped, below]  {$\tilde\vec{\lambda}_1+\tilde\vec{\lambda}_2+\tilde\vec{\lambda}_3$} (15,10);
		\draw[step=1cm,thick] (0,0) -- (0,10);
		
		\draw[step=1cm,thick] (0,0) -- (15,0);
		\draw[step=1cm,thick] (0,0) -- (5,7);
		\draw[step=1cm,thick] (0,0) -- (9,9);
		
		\draw [fill] (0,0) circle [radius=0.05];
		
		\draw [fill] (2,4) circle [radius=0.05];
		
		\draw [fill] (5,7) circle [radius=0.05];
		
		\draw [fill] (9,9) circle [radius=0.05];

		\draw [fill] (15,10) circle [radius=0.05];
		\node at (0-0.2,0-0.2) {$O$};
		\node at (5-0.2,7+0.3) {$P_2$};
		\node at (9-0.2,9+0.3) {$P_3$};
		\node at (15-0.2,10+0.3) {$P_4$};
		\node at (2-0.2,4+0.3) {$P_1$};
		
				  \useasboundingbox (-2,0);
		\end{tikzpicture}
		\caption{Geometric interpretation of the right hand side of \eqref{right hand side for kernel inequality}.}\label{area_inequality_2}
		
		\end{figure}
		
\end{center}		
		
	\end{remark}

	\subsection{The extension theorem}\label{completion}
We can now complete the proof of the extension theorem.

\begin{proof}[Proof of Theorem \ref{main_thm_3}]  Let the notation and assumptions be as in the theorem.  
By the reduction in \ref{strategy}, it suffices to treat the case where the maximal slope of $\Cal{E}$ is strictly less than the slope of $\Cal F_2$. By Theorem \ref{step-1}, there exists a surjection $\ecal \to \fcal_2$; equivalently, $\Surj(\ecal,\fcal_2)$ is nonempty. It now clearly suffices to show that $\Surj(\ecal,\fcal_2)^{\fcal_1}$ is nonempty.

Let $\mathcal{K}$ be a vector bundle on $X$ which is isomorphic to the kernel of a surjective map from $\mathcal{E}$ to $\mathcal{F}_2$. If $\mathcal{K}$ is not semistable, it clearly satisfies the assumptions in Theorem \ref{degree inequality for kernel}, so we obtain an inequality
	\[ \deg(\mathcal{K}^\vee \otimes \mathcal{E})^{\geq 0} < \deg(\mathcal{K}^\vee \otimes \mathcal{K})^{\geq 0} + \deg(\mathcal{E}^\vee \otimes \mathcal{F}_2)^{\geq 0}.\]
Since $\Surj(\mathcal{E}, \mathcal{F}_2)$ is not empty, we can insert this inequality into the dimension formulas for $\Surj(\mathcal{E}, \mathcal{F}_2)^\mathcal{K}$ and $\Surj(\mathcal{E}, \mathcal{F}_2)$ established in Proposition \ref{surjinjnice} and Lemma \ref{maindimlemma2}, respectively, yielding the strict inequality
	\[ \dim \Surj(\mathcal{E}, \mathcal{F}_2)^\mathcal{K} < \dim \Surj(\mathcal{E}, \mathcal{F}_2)\]for any $\kcal$ satisfying the assumptions of Theorem \ref{degree inequality for kernel}.  Let $S$ denote the set of isomorphism classes of such $\kcal$s.  Now, arguing as the proof of Theorem \ref{mainsurjthm1}, we deduce that the dimension of $X = |\Surj(\mathcal{E}, \mathcal{F}_2)| \smallsetminus |\Surj(\mathcal{E}, \mathcal{F}_2)^{\fcal_1}|$ satisfies the inequality $$\dim X \leq \sup_{\kcal \in S} \deg(\mathcal{K}^\vee \otimes \mathcal{E})^{\geq 0} - \deg(\mathcal{K}^\vee \otimes \mathcal{K})^{\geq 0} < \deg(\mathcal{E}^\vee \otimes \mathcal{F}_2)^{\geq 0},$$ which again yields a contradiction if $|\Surj(\mathcal{E}, \mathcal{F}_2)^{\fcal_1}|$ is empty.
\end{proof}
\begin{remark}There is a more quantitative form of Theorem \ref{main_thm_3} which we would like to explain. To describe this result, fix bundles $\fcal_i$ as in the statement of the theorem, and let $\mathcal{E}\mathrm{xt}(\fcal_2,\fcal_1)$ be the sheafification of the presheaf sending $S \in \mathrm{Perf}_{/\Spa\,F}$ to $H^1(\xcal_S,\fcal_{2,S}^\vee \otimes \fcal_{1,S})$. By arguments similar to those in \S \ref{bundlemaps}, one checks that $\mathcal{E}\mathrm{xt}(\fcal_2,\fcal_1)$ is a locally spatial and partially proper diamond over $\Spd\,F$, which (by construction) parametrizes all isomorphism classes of extensions of $\fcal_2$ by $\fcal_1$. For any given vector bundle $\ecal$, let $\mathcal{E}\mathrm{xt}(\fcal_2,\fcal_1)^{\ecal}$ denote the locally closed subdiamond of $\mathcal{E}\mathrm{xt}(\fcal_2,\fcal_1)$ parametrizing extensions which are isomorphic to $\ecal$ at all geometric points.  According to Theorem \ref{main_thm_3}, $\mathcal{E}\mathrm{xt}(\fcal_2,\fcal_1)^\ecal$ is nonempty if and only if $\HN(\ecal) \leq \HN(\fcal_1 \oplus \fcal_2)$. We then have the following result.
\begin{thm}\label{quantitativemainthm}
For any $\ecal$ such that $\HN(\ecal) \leq \HN(\fcal_1 \oplus \fcal_2)$, the diamond $\mathcal{E}\mathrm{xt}(\fcal_2,\fcal_1)^{\ecal}$ is equidimensional of dimension $$\deg (\fcal_2 \otimes \fcal_1^\vee) - \deg (\ecal \otimes \ecal^\vee)^{\geq 0}.$$
\end{thm}
Using Proposition \ref{geometric diamond lemma for one vector bundle}, one easily checks that the difference of degrees appearing here is twice the area of the region enclosed between $\HN(\ecal)$ and $\HN(\fcal_1 \oplus \fcal_2)$.

Roughly, the idea of the proof is that after rigidifying the extension, we obtain an $\mathcal{A}\mathrm{ut}(\ecal)$-torsor $$\mathcal{E}\mathrm{xt}(\fcal_2,\fcal_1)^{\ecal,\heart} \to \mathcal{E}\mathrm{xt}(\fcal_2,\fcal_1)^{\ecal}$$which can also be identified (more or less) with $\Surj(\ecal,\fcal_2)^{\fcal_1,\heart}$. One then argues using dimension theory as before.

\end{remark}

\section{Multi-step filtrations}\label{ind_step}

We now explain how to deduce Theorem \ref{main_thm_2} from Theorem \ref{main_thm_3}; note that the latter is simply the special case $k=2$ of the former.  

As in the assumptions of Theorem \ref{main_thm_2}, suppose vector bundles $\Cal{F}_1, \Cal{F}_2, \ldots, \Cal{F}_{k}$ are given, along with $\Cal{E}$ such that 
\[
\HN(\Cal{E}) \leq \HN(\Cal{F}_1 \oplus ... \oplus \Cal{F}_k).
\]
We want to find a filtration on $\Cal{E}$ whose subquotients are $\Cal{F}_i$. We induct on $k$, with the case $k=1$ being trivial; we shall need to use the case $k=2$, which is Theorem \ref{main_thm_3}, in the inductive step. \\

\begin{proof}[Proof of Theorem \ref{main_thm_2}] Consider graphing the HN polygons of $\Cal{E}$ and $\Cal{F}_1 \oplus ... \oplus \Cal{F}_k$. By assumption, we have 
\[
\HN(\Cal{F}_1 \oplus ... \oplus \Cal{F}_k) \geq \HN(\Cal{E}).
\]
Since $\Cal{F}_k$ is semistable and has the largest slope, it lies above $\HN(\Cal{E})$. (See Figure \ref{step1_fig}.)
		
		\begin{center}
		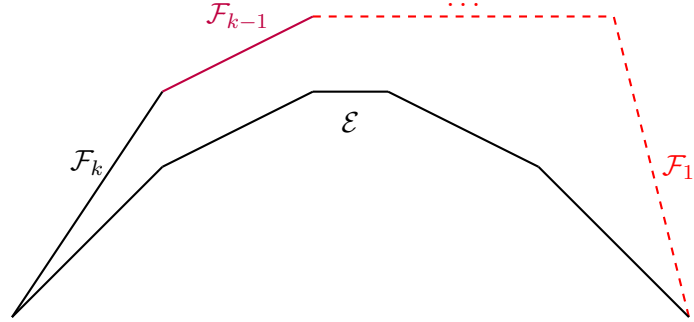
\begin{figure}[h]
	\begin{tikzpicture}	
	
	\coordinate (e0) at (0,0) {};
	\coordinate (e1) at (2,2) {};
	\coordinate (e2) at (4,3) {};
	\coordinate (e3) at (5,3) {};
	\coordinate (e4) at (7,2) {};
	\coordinate (e5) at (9,0) {};
	\coordinate (f1) at (2,3) {};
	\coordinate (f2) at (4,4) {};	     
	\coordinate (f3) at (8,4) {};

	\node[black] at (4.5,3-0.4) {$\Cal{E}$};

	\draw[step=1cm,thick,purple] (f1)--node[above=5pt]{$\Cal{F}_{k-1}$}(f2);
	\draw[step=1cm,thick,red,dashed] (f2)--node[above]{$\ldots$}(f3)--node[above,right]{$\Cal{F}_1$}(e5);
	\draw[step=1cm,thick] (e0) --(e1);
	\draw[step=1cm,thick] (e1) -- (e2);
	\draw[step=1cm,thick] (e2) -- (e3);
	\draw[step=1cm,thick] (e3) -- (e4);
	\draw[step=1cm,thick] (e4) -- (e5);	
	\draw[step=1cm,thick] (e0)--node[above=7 pt]{$\Cal{F}_k$}(f1);

	\end{tikzpicture}	
	\caption{Depiction of the HN polygons for $\Cal{E}$ and $\Cal{F}_1 \oplus \ldots \oplus \Cal{F}_k$.}\label{step1_fig}
		\end{figure}
		\end{center}

Now take the upper convex hull of $\HN(\Cal{F}_k)$ and $\HN(\Cal{E})$.  This gives a polygon which can be written as $\HN(\Cal{E}')$ for a bundle $\Cal{E}' \simeq \Cal{F}_k \oplus \Cal{E}_{k-1}$ for some $\Cal{E}_{k-1}$ (shown in blue below). We have not yet shown that $\Cal{E}_{k-1}$ is a sub-bundle of $\Cal{E}$, but we will do so shortly, and then $\Cal{E}_{k-1}$ will indeed be the first step of the filtration claimed in Theorem \ref{main_thm_2}.

 Note that $\HN(\Cal{E}_{k-1})$ will consist of a subset $\HN(\Cal{H}_{k-1})$ of $\HN(\Cal{E})$ together with a line segment of maximal slope, which is $\HN(\Cal{G}_{k-1})$ for some $\Cal{G}_{k-1}$, connecting $\Cal{F}_k$ to some vertex of $\Cal{E}$. (See Figure \ref{step2_fig}.) Finally, $\HN(\Cal{E})$ is the union of $\HN(\Cal{H}_{k-1})$ and $\HN(\Cal{L}_{k-1})$ for some $\Cal{L}_{k-1}$.

	\begin{center}
	\begin{figure}[h]
		\begin{tikzpicture}	
	
	\coordinate (e0) at (0,0) {};
	\coordinate (e1) at (2,2) {};
	\coordinate (e2) at (4,3) {};
	\coordinate (e3) at (5,3) {};
	\coordinate (e4) at (7,2) {};
	\coordinate (e5) at (9,0) {};
	\coordinate (f1) at (2,3) {};
	\coordinate (f2) at (4,4) {};	     
	\coordinate (f3) at (8,4) {};

	\node[blue] at (4,3-0.4) {$\Cal{E}_{k-1}$};	
	\node[black! 50! green] at (2+0.2,2-0.5) {$\Cal{L}_{k-1}$};	
	\node  at (6.7,1.7) {$\Cal{H}_{k-1}$};
	
	\draw[step=1cm,thick,purple] (f1)--node[above=5pt]{$\Cal{F}_{k-1}$}(f2);
	\draw[step=1cm,thick,red,dashed] (f2)--node[above]{$\ldots$}(f3)--node[above,right]{$\Cal{F}_1$}(e5);
	\draw[step=1cm,thick, black! 50! green] (e0) --(e1);
	\draw[step=1cm,thick, black! 50! green] (e1) -- (e2);
	\draw[step=1cm,thick] (e2) -- (e3);
	\draw[step=1cm,thick] (e3) -- (e4);
	\draw[step=1cm,thick] (e4) -- (e5);	
	\draw[step=1cm,thick] (e0)--node[above=7 pt]{$\Cal{F}_k$}(f1);
	\draw[step=1cm,thick,blue]  (f1)--node[below=7pt,left,black]{$\Cal{G}_{k-1}$}(e2)--(e3)--(e4)--(e5);	
	
	\end{tikzpicture}
	\caption{Construction of the first step of the filtration.}\label{step2_fig}
	\end{figure}
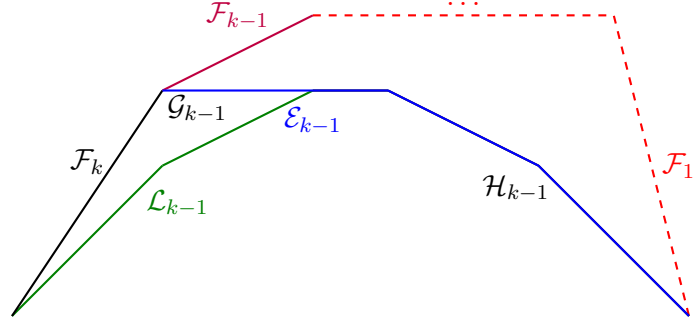
	\end{center}
	
Since all HN filtrations split (Corollary \ref{HNsplit}), we know that 
\begin{equation}\label{multi_eq1}
\Cal{E}_{k-1} \simeq \Cal{G}_{k-1} \oplus \Cal{H}_{k-1}
\end{equation}
and 
\begin{equation}\label{multi_eq2}
\Cal{E}\simeq \Cal{L}_{k-1} \oplus \Cal{H}_{k-1}.
\end{equation}
Now, by the $k=2$ case, which is a reformulation of Theorem \ref{main_thm_2}, we know that there exists an exact sequence 
\[
0 \to \Cal{G}_{k-1} \to \Cal{L}_{k-1} \to \Cal{F}_k \to 0
\]
Pushing out this sequence with respect to $\Cal{G}_{k-1} \rightarrow \Cal{G}_{k-1} \oplus \Cal{H}_{k-1}$ (i.e. taking the direct sum of the first two terms with $\Cal{H}_{k-1} \xrightarrow{\Id} \Cal{H}_{k-1}$), we obtain an exact sequence (using \eqref{multi_eq1} and \eqref{multi_eq2})
\[
0 \to \Cal{E}_{k-1} \to \Cal{E} \to \Cal{F}_k \to 0.
\]
Thus we have constructed a subbundle $\Cal{E}_{k-1}$ with $\Cal{E}/\Cal{E}_{k-1}=\Cal{F}_k$.\\

	\begin{center}
	\begin{figure}[h]
		\begin{tikzpicture}	
	
	\coordinate (e0) at (0,0) {};
	\coordinate (e1) at (2,2) {};
	\coordinate (e2) at (4,3) {};
	\coordinate (e3) at (5,3) {};
	\coordinate (e4) at (7,2) {};
	\coordinate (e5) at (9,0) {};
	\coordinate (f1) at (2,3) {};
	\coordinate (f2) at (4,4) {};	     
	\coordinate (f3) at (8,4) {};

	\node[blue] at (4,3-0.4) {$\Cal{E}_{k-1}$};

	\draw[step=1cm,thick] (e3) -- (e4);
	\draw[step=1cm,thick] (e4) -- (e5);	
	
	\draw[step=1cm,thick,blue]  (f1)--(e2)--(e3)--(e4)--(e5);	
	\draw[step=1cm,thick,purple] (f1)--node[above=5pt]{$\Cal{F}_{k-1}$}(f2); 
	\draw[step=1cm,thick,purple,dashed] (f1)--node[above=5pt]{$\Cal{F}_{k-1}$}(f2)--(e4)--(e5);
	\draw[step=1cm,thick,red,dashed] (f2)--node[above]{$\ldots$}(f3)--node[above,right]{$\Cal{F}_1$}(e5);
\end{tikzpicture}
	\caption{The proof is completed by induction.} \label{induction is awesome}
	\end{figure}
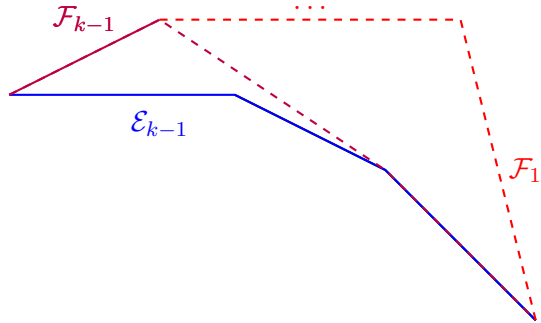
	\end{center}

By convexity of the HN polygons, we have $\HN(\Cal{E}_{k-1}) \leq \HN(\Cal{F}_1 \oplus \ldots \oplus \Cal{F}_{k-1})$ with the same endpoints. We can then conclude by induction (Figure \ref{induction is awesome}).
	
\end{proof}

\bibliographystyle{amsalpha}% Use the "unsrtnat" BibTeX style for formatting the Bibliography

\bibliography{Bibliography}

\providecommand{\bysame}{\leavevmode\hbox to3em{\hrulefill}\thinspace}
\providecommand{\MR}{\relax\ifhmode\unskip\space\fi MR }
\providecommand{\MRhref}[2]{%
  \href{http://www.ams.org/mathscinet-getitem?mr=#1}{#2}
}
\providecommand{\href}[2]{#2}
\begin{thebibliography}{Hub96}

\bibitem[Col02]{Col}
Pierre Colmez, \emph{Espaces de {B}anach de dimension finie}, J. Inst. Math.
  Jussieu \textbf{1} (2002), no.~3, 331--439. \MR{1956055}

\bibitem[CS]{CS}
Ana Caraiani and Peter Scholze, \emph{{O}n the generic part of the cohomology
  of compact unitary {S}himura varieties}, {A}nnals of {M}ath., to appear.

\bibitem[Far17]{Far17}
Laurent Fargues, \emph{Simple connexit{\'e} des fibres d'une application
  d'{A}bel-{J}acobi et corps de classe local}, preprint,
  https://webusers.imj-prg.fr/~laurent.fargues/cdc.pdf.

\bibitem[FF]{FF08}
Laurent Fargues and Jean-Marc Fontaine, \emph{Courbes et fibres vectoriels en
  theorie de hodge p-adique}.

\bibitem[FF14]{FF14}
\bysame, \emph{Vector bundles on curves and {$p$}-adic {H}odge theory},
  Automorphic forms and {G}alois representations. {V}ol. 2, London Math. Soc.
  Lecture Note Ser., vol. 415, Cambridge Univ. Press, Cambridge, 2014,
  pp.~17--104. \MR{3444231}

\bibitem[Han17]{Hanvb}
David Hansen, \emph{Degenerating vector bundles in $p$-adic {H}odge theory},
  preprint, http://www.math.columbia.edu/~hansen/degen.pdf.

\bibitem[Hub96]{Huberbook}
Roland Huber, \emph{\'etale cohomology of rigid analytic varieties and adic
  spaces}, Aspects of Mathematics, E30, Friedr. Vieweg \& Sohn, Braunschweig,
  1996. \MR{1734903}

\bibitem[Ked]{Ked17}
Kiran Kedlaya, \emph{Sheaves, stacks, and shtukas}, preprint,
  http://swc.math.arizona.edu/aws/2017/2017KedlayaNotes.pdf.

\bibitem[Ked08]{Ked08}
Kiran~S. Kedlaya, \emph{Slope filtrations for relative {F}robenius},
  Ast\'erisque (2008), no.~319, 259--301, Repr\'esentations $p$-adiques de
  groupes $p$-adiques. I. Repr\'esentations galoisiennes et
  $(\phi,\Gamma)$-modules. \MR{2493220}

\bibitem[Ked16]{KedNoeth}
\bysame, \emph{Noetherian properties of {F}argues-{F}ontaine curves}, Int.
  Math. Res. Not. IMRN (2016), no.~8, 2544--2567. \MR{3519123}

\bibitem[KL15]{KL15}
Kiran~S. Kedlaya and Ruochuan Liu, \emph{Relative {$p$}-adic {H}odge theory:
  foundations}, Ast\'erisque (2015), no.~371, 239. \MR{3379653}

\bibitem[LB]{leBras}
Arthur-C\'esar Le~Bras, \emph{Espaces de banach-colmez et faisceaux cohérents
  sur la courbe de fargues-fontaine}, preprint.

\bibitem[Sch]{Sch}
Peter Scholze, \emph{{\'E}tale cohomology of diamonds}, preprint,
  http://www.math.uni-bonn.de/people/scholze/EtCohDiamonds.pdf.

\bibitem[SW]{SW15}
Peter Scholze and Jared Weinstein, \emph{Lectures on $p$-adic geometry}.

\bibitem[SW13]{SW13}
\bysame, \emph{Moduli of {$p$}-divisible groups}, Camb. J. Math. \textbf{1}
  (2013), no.~2, 145--237. \MR{3272049}

\bibitem[Wei]{Wei}
Jared Weinstein, \emph{{$\mathrm{Gal}(\overline{\mathbf{Q}_p}/\mathbf{Q}_p)$}
  as a geometric fundamental group}, preprint.

\end{thebibliography}
	
\end{document}